\documentclass[11pt]{article}
\usepackage{amsmath,amssymb}
\usepackage{graphicx}
\usepackage{amsfonts}

\usepackage{amsthm}
\newtheorem{thm}{Theorem}[section]
\newtheorem{prop}[thm]{Proposition}

\newtheorem{corollary}[thm]{Corollary}

\usepackage[USenglish]{babel} 
\usepackage[T1]{fontenc}
\usepackage[ansinew]{inputenc}
\usepackage{mathrsfs}
\usepackage{stmaryrd} 
\usepackage{epstopdf} 
\usepackage{caption}  
\usepackage{authblk}

\numberwithin{equation}{section}
\numberwithin{figure}{section}

\title{R-adaptive multisymplectic and variational integrators}
\date{}
\author[1,2]{Tomasz M. Tyranowski\thanks{\texttt{tomasz.tyranowski@ipp.mpg.de}}}
\author[2]{Mathieu Desbrun\thanks{\texttt{mathieu@caltech.edu}}}

\affil[1]{\small Max-Planck-Institut f\"ur Plasmaphysik \authorcr Boltzmannstra{\ss}e 2, 85748 Garching, Germany}
\affil[2]{\small California Institute of Technology, Computing {\small $+$} Mathematical Sciences \authorcr 1200 E. California Blvd., Pasadena, CA 91125, USA}

\begin{document}

\maketitle

\begin{abstract}
Moving mesh methods (also called $r$-adaptive methods) are space-adaptive strategies used for the numerical simulation of time-dependent partial differential equations. These methods keep the total number of mesh points fixed during the simulation, but redistribute them over time to follow the areas where a higher mesh point density is required. There are a very limited number of moving mesh methods designed for solving field-theoretic partial differential equations, and the numerical analysis of the resulting schemes is challenging. In this paper we present two ways to construct $r$-adaptive variational and multisymplectic integrators for (1+1)-dimensional Lagrangian field theories. The first method uses a variational discretization of the physical equations and the mesh equations are then coupled in a way typical of the existing $r$-adaptive schemes. The second method treats the mesh points as pseudo-particles and incorporates their dynamics directly into the variational principle. A user-specified adaptation strategy is then enforced through Lagrange multipliers as a constraint on the dynamics of both the physical field and the mesh points. We discuss the advantages and limitations of our methods. Numerical results for the Sine-Gordon equation are also presented.
\end{abstract}

\section{Introduction}
\label{sec:intro}
The purpose of this work is to design, analyze and implement variational and multisymplectic integrators for Lagrangian partial differential equations with space-adaptive meshes. In this paper we combine geometric numerical integration and $r$-adaptive methods for the numerical solution of PDEs. We show that these two fields are compatible, mostly due to the fact that in $r$-adaptation the number of mesh points remains constant and we can treat them as additional pseudo-particles whose dynamics is coupled to the dynamics of the physical field of interest.

Geometric (or structure-preserving) integrators are numerical methods that preserve geometric properties of the flow of a differential equation (see \cite{HLWGeometric}). This encompasses symplectic integrators for Hamiltonian systems, variational integrators for Lagrangian systems, and numerical methods on manifolds, including Lie group methods and integrators for constrained mechanical systems. Geometric integrators proved to be extremely useful for numerical computations in astronomy, molecular dynamics, mechanics and theoretical physics. The main motivation for developing structure-preserving algorithms lies in the fact that they show excellent numerical behavior, especially for long-time integration of equations possessing geometric properties.

An important class of structure-preserving integrators are \emph{variational integrators} for Lagrangian systems (\cite{HLWGeometric}, \cite{MarsdenWestVarInt}). This type of integrator is based on discrete variational principles. The variational approach provides a unified framework for the analysis of many symplectic algorithms and is characterized by a natural treatment of the discrete Noether theorem, as well as forced, dissipative and constrained systems. Variational integrators were first introduced in the context of finite-dimensional mechanical systems, but later Marsden \& Patrick \& Shkoller~\cite{MarsdenPatrickShkoller} generalized this idea to field theories. Variational integrators have since then been successfully applied in many computations, for example in elasticity (\cite{LewAVI}), electrodynamics (\cite{SternDesbrun}) or fluid dynamics (\cite{Pavlov}). Existing variational integrators so far have been developed on static, mostly uniform spatial meshes. The main goal of this paper is to design and analyze variational integrators that allow for the use of space-adaptive meshes.

Adaptive meshes used for the numerical solution of partial differential equations fall into three main categories: $h$-adaptive, $p$-adaptive and $r$-adaptive.
$R$-adaptive methods, which are also known as \emph{moving mesh methods} (\cite{HuangRussellREVIEW}, \cite{HuangRussellBOOK}), keep the total number of mesh points fixed during the simulation, but relocate them over time. These methods are designed to minimize the error of the computations by optimally distributing the mesh points, contrasting with $h$-adaptive methods for which the accuracy of the computations is obtained via insertion and deletion of mesh points. Moving mesh methods are a large and interesting research field of applied mathematics, and their role in modern computational modeling is growing. Despite the increasing interest in these methods in recent years, they are still in a relatively early stage of their development compared to the more matured $h$-adaptive methods.

\subsubsection*{Overview}

There are three logical steps to $r$-adaptation:

\begin{itemize}
	\item Discretization of the physical PDE
	\item Mesh adaptation strategy
	\item Coupling the mesh equations to the physical equations
\end{itemize}

\noindent
The key ideas of this paper regard the first and the last step. Following the general spirit of variational integrators, we discretize the underlying action functional rather than the PDE itself, and then derive the discrete equations of motion. We base our adaptation strategies on the equidistribution principle and the resulting moving mesh partial differential equations (MMPDEs). We interpret MMPDEs as constraints, which allows us to consider {\it novel} ways of coupling them to the physical equations. Note that we will restrict our explanations to one time and one space dimension for the sake of simplicity.

Let us consider a (1+1)-dimensional scalar field theory with the action functional

\begin{equation}
\label{eq:action}
S[\phi] = \int_0^{T_{max}} \int_0^{X_{max}} \mathcal{L}(\phi,\phi_X,\phi_t)\,dX\,dt,
\end{equation}

\noindent
where $\phi:[0,X_{max}]\times [0,T_{max}] \longrightarrow \mathbb{R}$ is the field and $\mathcal{L}:\mathbb{R} \times \mathbb{R} \times \mathbb{R} \longrightarrow \mathbb{R}$ its Lagrangian density. For simplicity, we assume the following fixed boundary conditions

\begin{align}
\label{eq:bndcond}
\phi(0,t)&=\phi_L, \nonumber \\
\phi(X_{max},t)&=\phi_R.
\end{align}

\noindent
In order to further consider moving meshes let us perform a change of variables $X=X(x,t)$ such that for all $t$ the map $X(.,t):[0,X_{max}]\longrightarrow [0,X_{max}]$ is a \textquoteleft diffeomorphism\textquoteright---more precisely, we only require that $X(.,t)$ is a homeomorphism such that both $X(.,t)$ and $X(.,t)^{-1}$ are piecewise $C^1$. In the context of mesh adaptation the map $X(x,t)$ represents the spatial position at time $t$ of the mesh point labeled by $x$. Define $\varphi(x,t) = \phi(X(x,t),t)$. Then the partial derivatives of $\phi$ are $\phi_X(X(x,t),t)=\varphi_x/X_x$ and $\phi_t(X(x,t),t)=\varphi_t - \varphi_x X_t/X_x$. Plugging these equations in (\ref{eq:action}) we get

\begin{equation}
\label{eq:action2}
S[\phi]=\int_0^{T_{max}} \int_0^{X_{max}} \mathcal{L} \Big(\varphi,\frac{\varphi_x}{X_x},\varphi_t-\frac{\varphi_x X_t}{X_x}\Big) X_x\,dx\,dt =: \tilde{S}[\varphi], \tilde{S}[\varphi,X]
\end{equation}

\noindent
where the last equality defines two modified, or \textquoteleft reparametrized',  action functionals. For the first one,  $\tilde S$ is considered as a functional of $\varphi$ only, whereas in the second one we also treat it as a functional of $X$. This leads to two different approaches to mesh adaptation, which we dub the \emph{control-theoretic} strategy and the \emph{Lagrange multiplier} strategy, respectively. The \textquoteleft reparametrized' field theories defined by $\tilde{S}[\varphi]$ and $\tilde{S}[\varphi,X]$ are both intrinsically covariant; however, it is convenient for computational purposes to work with a space-time split and formulate the field dynamics as an initial value problem. 

\subsubsection*{Outline}
This paper is organized as follows. In Section~\ref{sec:approach1} and Section~\ref{sec:approach2} we take the view of infinite dimensional manifolds of fields as configuration spaces, and develop the control-theoretic and Lagrange multiplier strategies in that setting. It allows us to discretize our system in space first and consider time discretization later on. It is clear from our exposition that the resulting integrators are variational. In Section~\ref{sec:multisymplecticity} we show how similar integrators can be constructed using the covariant formalism of multisymplectic field theory. We also show how the integrators from the previous sections can be interpreted as multisymplectic. In Section~\ref{sec:numerics} we apply our integrators to the Sine-Gordon equation and we present our numerical results. We summarize our work in Section~\ref{sec:summary} and discuss several directions in which it can be extended. 

\section{Control-theoretic approach to $r$-adaptation}
\label{sec:approach1}

At first glance, it appears that the simplest and most straightforward way to construct an $r$-adaptive variational integrator would be to discretize the physical system in a similar manner to the general approach to variational integration, i.e. discretize the underlying variational principle, and then derive the mesh equations and couple them to the physical equations in a way typical of the existing $r$-adaptive algorithms. We explore this idea in this section and show that it indeed leads to space adaptive integrators that are variational in nature. However, we also show that those integrators do not exhibit the behavior expected of geometric integrators, such as good energy conservation. We will refer to this strategy as \emph{control-theoretic}, since in this description the field $\varphi$ represents the physical state of the system, while $X$ can be interpreted as a control variable and the mesh equations as feedback (see, e.g., \cite{Nijmeijer}).

\subsection{Reparametrized Lagrangian}
\label{subsec:reparametrized lagrangian}

For the moment let us assume that $X(x,t)$ is a known function. We denote by $\xi(X,t)$ the function such that $\xi(.,t)=X(.,t)^{-1}$, that is $\xi(X(x,t),t)=x$ \footnote{We allow a little abuse of notation here: $X$ denotes both the argument of $\xi$ and the change of variables $X(x,t)$. If we wanted to be more precise, we would write $X = h(x,t)$.}. We thus have $\tilde S[\varphi] = S[\varphi(\xi(X,t),t)]$.

\begin{prop}
\label{Thm: App1 extremization equivalence}
Extremizing $S[\phi]$ with respect to $\phi$ is equivalent to extremizing $\tilde S[\varphi]$ with respect to $\varphi$.
\end{prop}

\begin{proof}
The variational derivatives of $S$ and $\tilde S$ are related by the formula

\begin{align}
\label{eq:approach 1 - variational derivatives}
&\delta \tilde S[\varphi] \cdot \delta \varphi(x,t) = \delta S[\varphi(\xi(X,t),t)] \cdot \delta \varphi(\xi(X,t),t).
\end{align} 

\noindent
Suppose $\phi(X,t)$ extremizes $S[\phi]$, i.e. $\delta S[\phi]\cdot\delta \phi=0$ for all variations $\delta \phi$. Given the function $X(x,t)$, define $\varphi(x,t)=\phi(X(x,t),t)$. Then by the formula above we have $\delta \tilde S[\varphi]=0$, so $\varphi$ extremizes $\tilde S$. Conversely, suppose $\varphi(x,t)$ extremizes $\tilde S$, that is $\delta \tilde S[\varphi] \cdot \delta \varphi =0$ for all variations $\delta \varphi$. Since we assume $X(.,t)$ is a homeomorphism, we can define $\phi(X,t)=\varphi(\xi(X,t),t)$. Note that an arbitrary variation $\delta \phi(X,t)$ induces the variation $\delta \varphi(x,t) = \delta \phi(X(x,t),t)$. Then we have $\delta S[\phi] \cdot \delta \phi = \delta \tilde S[\varphi] \cdot \delta \varphi = 0$ for all variations $\delta \phi$, so $\phi(X,t)$ extremizes $S[\phi]$.\\
\end{proof}

The corresponding instantaneous Lagrangian $\tilde L : Q \times W \times \mathbb{R} \longrightarrow \mathbb{R}$ is

\begin{equation}
\label{eq:Lagrangian}
\tilde{L}[\varphi,\varphi_t,t]= \int_0^{X_{max}} \tilde {\mathcal{L}}(\varphi,\varphi_x,\varphi_t,t)\,dx
\end{equation}

\noindent
with the Lagrangian density

\begin{equation}
\label{eq:LagrangianDensity}
\tilde{\mathcal{L}}(\varphi,\varphi_x,\varphi_t,x,t) = \mathcal{L} \Big(\varphi,\frac{\varphi_x}{X_x},\varphi_t-\frac{\varphi_x X_t}{X_x}\Big) X_x.
\end{equation}

\noindent
The function spaces $Q$ and $W$ must be chosen appropriately for the problem at hand, so that (\ref{eq:Lagrangian}) makes sense. For instance, for a free field we will have $Q=H^1([0,X_{max}])$ and $W=L^2([0,X_{max}])$. Since $X(x,t)$ is a function of $t$, we are looking at a time-dependent system. Even though the energy associated with (\ref{eq:Lagrangian}) is not conserved, the energy of the original theory associated with (\ref{eq:action})

\begin{align}
\label{eq:Energy1}
E &= \int_0^{X_{max}} \Big( \phi_t \frac{\partial \mathcal{L}}{\partial \phi_t}(\phi,\phi_X,\phi_t) - \mathcal{L}(\phi,\phi_X,\phi_t)\Big) \, dX \\
\label{eq:Energy2}
  &= \int_0^{X_{max}} \Big[ \Big( \varphi_t-\frac{\varphi_x X_t}{X_x}\Big) \frac{\partial \mathcal{L}}{\partial \phi_t}\Big(\varphi,\frac{\varphi_x}{X_x},\varphi_t-\frac{\varphi_x X_t}{X_x} \Big) - \mathcal{L}\Big(\varphi,\frac{\varphi_x}{X_x},\varphi_t-\frac{\varphi_x X_t}{X_x}\Big)\Big] X_x\, dx
\end{align}

\noindent
is conserved. To see this, note that if $\phi(X,t)$ extremizes $S[\phi]$ then $dE/dt=0$ (computed from (\ref{eq:Energy1})). Trivially, this means that $dE/dt=0$ when formula (\ref{eq:Energy2}) is invoked as well. Moreover, as we have noted earlier, $\phi(X,t)$ extremizes $S[\phi]$ iff $\varphi(x,t)$ extremizes $\tilde S[\varphi]$. This means that the energy (\ref{eq:Energy2}) is constant on solutions of the reparametrized theory.

\subsection{Spatial Finite Element discretization}
\label{subsec:FEM}

We begin with a discretization of the spatial dimension only, thus turning the original infinite-dimensional problem into a time-continuous finite-dimensional Lagrangian system. Let $\Delta x = X_{max}/(N+1)$ and define the reference uniform mesh $x_i=i \cdot \Delta x$ for $i=0,1,...,N+1$, and the corresponding piecewise linear finite elements

\begin{equation}
\label{eq:basis functions}
\eta_i(x)=\left\{
\begin{array}{cc}
\frac{x-x_{i-1}}{\Delta x}, & \mbox{if } x_{i-1} \leq x \leq x_i,\\
-\frac{x-x_{i+1}}{\Delta x}, & \mbox{if } x_i \leq x \leq x_{i+1},\\
0, & \mbox{otherwise.}\\
\end{array}
\right.
\end{equation}

\noindent
We now restrict $X(x,t)$ to be of the form

\begin{equation}
\label{eq:X FEM}
X(x,t) = \sum_{i=0}^{N+1} X_i(t) \eta_i(x)
\end{equation}

\noindent
with $X_0(t)=0$, $X_{N+1}(t)=X_{max}$ and arbitrary $X_i(t)$, $i=1,2,...,N$ as long as $X(.,t)$ is a homeomorphism for all $t$.  In our context of numerical computations, the functions $X_i(t)$ represent the current position of the $i^{\text{th}}$ mesh point. Define the finite element spaces

\begin{equation}
\label{eq:QNWN}
Q_N=W_N=\text{span}(\eta_0,...,\eta_{N+1})
\end{equation}

\noindent
and assume that $Q_N \subset Q$, $W_N \subset W$. Let us denote a generic element of $Q_N$ by $\varphi$ and a generic element of $W_N$ by $\dot \varphi$. We have the decompositions

\begin{equation}
\label{eq:phi FEM}
\varphi(x) = \sum_{i=0}^{N+1} y_i \eta_i(x), \quad\quad\quad \dot \varphi(x) = \sum_{i=0}^{N+1} \dot y_i \eta_i(x).
\end{equation}

\noindent
The numbers $(y_i,\dot y_i)$ thus form natural (global) coordinates on $Q_N \times W_N$. We can now approximate the dynamics of system (\ref{eq:Lagrangian}) in the finite-dimensional space $Q_N \times W_N$. Let us consider the restriction $\tilde L_N= \tilde L|_{Q_N \times W_N \times \mathbb{R}}$ of the Lagrangian (\ref{eq:Lagrangian}) to $Q_N \times W_N \times \mathbb{R}$. In the chosen coordinates we have

\begin{equation}
\label{eq:LN FEM}
\tilde L_N(y_0,...,y_{N+1},\dot y_0,...,\dot y_{N+1}, t) = \tilde L \Big[\sum_{i=0}^{N+1} y_i \eta_i(x),\sum_{i=0}^{N+1} \dot y_i \eta_i(x),t\Big].
\end{equation}

\noindent
Note that, given the boundary conditions (\ref{eq:bndcond}), $y_0$, $y_{N+1}$, $\dot y_0$, and $\dot y_{N+1}$ are fixed. We will thus no longer write them as arguments of $\tilde L_N$.

The advantage of using a finite element discretization lies in the fact that the symplectic structure induced on $Q_N \times W_N$ by $\tilde L_N$ is strictly a restriction (i.e., a pull-back) of the (pre-)symplectic structure\footnote{In most cases the symplectic structure of $(Q \times W, \tilde L)$ is only weakly-nondegenerate; see \cite{GotayPhDThesis}} on $Q \times W$. This establishes a direct link between symplectic integration of the finite-dimensional mechanical system $(Q_N \times W_N, \tilde L_N)$ and the infinite-dimensional field theory $(Q \times W, \tilde L)$

\subsection{DAE formulation and time integration}
\label{sec:timeintegration}

We now consider time integration of the Lagrangian system $(Q_N \times W_N, \tilde L_N)$. If the functions $X_i(t)$ are known, then one can perform variational integration in the standard way, that is, define the discrete Lagrangian $\tilde L_d:\mathbb{R} \times Q_N \times \mathbb{R} \times Q_N \rightarrow \mathbb{R}$ and solve the corresponding discrete Euler-Lagrange equations (see \cite{MarsdenWestVarInt}, \cite{HLWGeometric}). Let $t_n=n\cdot \Delta t$ for $n=0,1,2,\ldots$ be an increasing sequence of times and $\{y^0,y^1,\ldots\}$ the corresponding discrete path of the system in $Q_N$. The discrete Lagrangian $L_d$ is an approximation to the exact discrete Lagrangian $L_d^E$, such that

\begin{equation}
\label{eq:DiscreteLagrangianDefinition}
\tilde L_d(t_n,y^{n},t_{n+1},y^{n+1}) \approx \tilde L_d^E(t_n,y^{n},t_{n+1},y^{n+1}) \equiv \int_{t_n}^{t_{n+1}}\tilde L_N(y(t),\dot y(t),t) \, dt,
\end{equation}

\noindent
where $y^n=(y_1^n,...,y_N^n)$, $y^{n+1}=(y_1^{n+1},...,y_N^{n+1})$ and $y(t)$ is the solution of the Euler-Lagrange equations corresponding to $\tilde L_N$ with the boundary values $y(t_n)=y^n$, $y(t_{n+1})=y^{n+1}$. Depending on the quadrature we use to approximate the integral in (\ref{eq:DiscreteLagrangianDefinition}), we obtain different types of variational integrators. As will be discussed below, in $r$-adaptation one has to deal with stiff differential equations or differential-algebraic equations, therefore higher order implicit integration in time is advisable (see \cite{PetzoldDAE}, \cite{HWODE2}). We will employ variational partitioned Runge-Kutta methods. An $s$-stage Runge Kutta method is constructed by choosing

\begin{equation}
\label{eq:DiscreteLagrangianRK}
\tilde L_d(t_n,y^{n},t_{n+1},y^{n+1}) = (t_{n+1}-t_n) \sum_{i=1}^s b_i \tilde L_N(Y_i,\dot Y_i,t_i),
\end{equation}

\noindent
where $t_i=t_n+c_i(t_{n+1}-t_n)$, the right-hand side is extremized under the constraint $y^{n+1}= y^n+(t_{n+1}-t_n) \sum_{i=1}^s b_i \dot Y_i$, and the internal stage variables $Y_i$, $\dot Y_i$ are related by $Y_i= y^n+(t_{n+1}-t_n) \sum_{j=1}^s a_{ij} \dot Y_j$. It can be shown that the variational integrator with the discrete Lagrangian (\ref{eq:DiscreteLagrangianRK}) is equivalent to an appropriately chosen symplectic partitioned Runge-Kutta method applied to the Hamiltonian system corresponding to $\tilde L_N$ (see \cite{MarsdenWestVarInt}, \cite{HLWGeometric}). With this in mind we turn our semi-discrete Lagrangian system $(Q_N \times W_N, \tilde L_N)$ into the Hamiltonian system $(Q_N \times W_N^*, \tilde H_N)$ via the standard Legendre transform

\begin{equation}
\label{eq:Hamiltonian}
\tilde H_N(y_1,...,y_N,p_1,...,p_N;X_1,...,X_N,\dot X_1,...,\dot X_N) = \sum_{i=1}^N p_i \dot y_i - \tilde L_N(y_1,...,y_{N},\dot y_1,...,\dot y_{N}, t), 
\end{equation}

\noindent
where $p_i = \partial \tilde L_N / \partial \dot y_i$ and we explicitly stated the dependence on the positions $X_i$ and velocities $\dot X_i$ of the mesh points. The Hamiltonian equations take the form\footnote{It is computationally more convenient to directly integrate the implicit Hamiltonian system $p_i = \partial \tilde L_N / \partial \dot y_i$, $\dot p_i = \partial \tilde L_N / \partial y_i$, but as long as system (\ref{eq:action}) is at least weakly-nondegenerate there is no theoretical issue with passing to the Hamiltonian formulation, which we do for the clarity of our exposition.}

\begin{align}
\label{eq:Hamiltonian ODE}
&\dot y_i = \frac{\partial \tilde H_N}{\partial p_i}\Big(y,p; X(t),\dot X(t)\Big),\\
&\dot p_i = -\frac{\partial \tilde H_N}{\partial y_i}\Big(y,p; X(t),\dot X(t)\Big). \nonumber
\end{align}

\noindent
Suppose that the functions $X_i(t)$ are $C^1$ and $H_N$ is smooth as a function of the $y_i$'s, $p_i$'s, $X_i$'s and $\dot X_i$'s (note that these assumptions are used for simplicity, and can be easily relaxed if necessary, depending on the regularity of the considered Lagrangian system). Then the assumptions of Picard's theorem are satisfied and there exists a unique $C^1$ flow $F_{t_0,t} = (F^y_{t_0,t},F^p_{t_0,t}):Q_N \times W_N^* \rightarrow Q_N \times W_N^*$ for (\ref{eq:Hamiltonian ODE}). This flow is symplectic.

However, in practice we do not know the $X_i$'s and we in fact would like to be able to adjust them \textquoteleft on the fly', based on the current behavior of the system. We will do that by introducing additional constraint functions $g_i(y_1,...,y_N,X_1,...,X_N)$ and demanding that the conditions $g_i=0$ be satisfied at all times\footnote{In the context of Control Theory the constraints $g_i=0$ are called \emph{strict static state feedback}. See \cite{Nijmeijer}.}. The choice of these functions will be discussed in Section~\ref{subsec:MMPDEs}. This leads to the following differential-algebraic system of index 1 (see \cite{PetzoldDAE}, \cite{HWODE2}, \cite{HLLectureNotes})

\begin{align}
\label{eq:Hamiltonian DAE}
&\dot y_i = \frac{\partial \tilde H_N}{\partial p_i}\Big(y,p; X,\dot X\Big),\\
&\dot p_i = -\frac{\partial \tilde H_N}{\partial y_i}\Big(y,p; X,\dot X\Big), \nonumber\\
				&0= g_i(y,X), \nonumber\\
				&y_i(t_0)=y_i^{(0)}, \nonumber \\
				&p_i(t_0)=p_i^{(0)} \nonumber
\end{align}

\noindent
for $i=1,...,N$. Note that an initial condition for $X$ is fixed by the constraints. This system is of index 1 because one has to differentiate the algebraic equations with respect to time once in order to reduce it to an implicit ODE system. In fact, the implicit system will take the form

\begin{align}
\label{eq:Hamiltonian IODE}
&\dot y = \frac{\partial \tilde H_N}{\partial p}\Big(y,p; X,\dot X\Big),\\
&\dot p = -\frac{\partial \tilde H_N}{\partial y}\Big(y,p; X,\dot X\Big), \nonumber\\
				&0= \frac{\partial g}{\partial y}(y,X) \dot y + \frac{\partial g}{\partial X}(y,X) \dot X, \nonumber\\
				&y(t_0)=y^{(0)}, \nonumber \\
				&p(t_0)=p^{(0)}, \nonumber \\
				&X(t_0)=X^{(0)}, \nonumber
\end{align}

\noindent
where $X^{(0)}$ is a vector of arbitrary initial condition for the $X_i$'s. Suppose again that $H_N$ is a smooth function of $y$, $p$, $X$ and $\dot X$. Futhermore, suppose that $g$ is a $C^1$ function of $y$, $X$, and $\frac{\partial g}{\partial X}-\frac{\partial g}{\partial y} \frac{\partial^2 H_N}{\partial \dot X \partial p}$ is invertible with its inverse bounded in a neighborhood of the exact solution.\footnote{Again, these assumptions can be relaxed if necessary.} Then, by the Implicit Function Theorem equations (\ref{eq:Hamiltonian IODE}) can be solved explicitly for $\dot y$, $\dot p$, $\dot X$ and the resulting explicit ODE system will satisfy the assumptions of Picard's theorem. Let $(y(t),p(t),X(t))$ be the unique $C^1$ solution to this ODE system (and hence to (\ref{eq:Hamiltonian IODE})). We have the trivial result\\

\begin{prop}
\label{Thm: IODE vs DAE}
If $g(y^{(0)},X^{(0)})=0$, then $(y(t),p(t),X(t))$ is a solution to (\ref{eq:Hamiltonian DAE}).\footnote{Note that there might be other solutions, as for any given $y^{(0)}$ there might be more than one $X^{(0)}$ that solves the constraint equations.}
\end{prop}

In practice we would like to integrate system (\ref{eq:Hamiltonian DAE}). A question arises in what sense is this system symplectic and in what sense a numerical integration scheme for this system can be regarded as variational. Let us address these issues.

\begin{prop}
\label{Thm: DAE vs Hamiltonian ODE}
Let $(y(t),p(t),X(t))$ be a solution to (\ref{eq:Hamiltonian DAE}) and use this $X(t)$ to form the Hamiltonian system (\ref{eq:Hamiltonian ODE}). Then we have that

\begin{equation}
y(t) = F^y_{t_0,t}(y^{(0)},p^{(0)}), \quad\quad p(t) = F^p_{t_0,t}(y^{(0)},p^{(0)}) \nonumber
\end{equation}

\noindent
and
\begin{equation}
g\Big(F^y_{t_0,t}(y^{(0)},p^{(0)}),X(t)\Big)=0,\nonumber
\end{equation}

\noindent
where $F_{t_0,t}(\hat y,\hat p)$ is the symplectic flow for (\ref{eq:Hamiltonian ODE}).
\end{prop}

\begin{proof}
Note that the first two equations of (\ref{eq:Hamiltonian DAE}) are the same as (\ref{eq:Hamiltonian ODE}), therefore $(y(t),p(t))$ trivially satisfies (\ref{eq:Hamiltonian ODE}) with the initial conditions $y(t_0)=y^{(0)}$ and $p(t_0)=p^{(0)}$. Since the flow map $F_{t_0,t}$ is unique, we must have $y(t) = F^y_{t_0,t}(y^{(0)},p^{(0)})$ and $p(t) = F^p_{t_0,t}(y^{(0)},p^{(0)})$. Then we also must have that $g\Big(F^y_{t_0,t}(y^{(0)},p^{(0)}),X(t)\Big)=0$, that is,  the constraints are satisfied along one particular integral curve of (\ref{eq:Hamiltonian ODE}) that passes through $(y^{(0)},p^{(0)})$ at $t_0$.\\
\end{proof}

Suppose we now would like to find a numerical approximation of the solution to (\ref{eq:Hamiltonian ODE}) using an $s$-stage partitioned Runge-Kutta method with coefficients $a_{ij}$, $b_i$, $\bar a_{ij}$, $\bar b_i$, $c_i$ (\cite{HWODE1}, \cite{HLWGeometric}). The numerical scheme will take the form

\begin{align}
\label{eq:PRK for ODE}
\dot Y^i &= \frac{\partial \tilde H_N}{\partial p}\Big(Y^i,P^i; X(t_n+c_i \Delta t),\dot X(t_n+c_i \Delta t)\Big),\\
\dot P^i &= -\frac{\partial \tilde H_N}{\partial y}\Big(Y^i,P^i; X(t_n+c_i \Delta t),\dot X(t_n+c_i \Delta t)\Big), \nonumber\\
Y^i &= y^n + \Delta t \sum_{j=1}^s a_{ij} \dot Y^j, \nonumber \\
P^i &= p^n + \Delta t \sum_{j=1}^s \bar a_{ij} \dot P^j, \nonumber \\
y^{n+1} &= y^n + \Delta t \sum_{i=1}^s b_i \dot Y^i, \nonumber \\
p^{n+1} &= p^n + \Delta t \sum_{i=1}^s \bar b_i \dot P^i,  \nonumber
\end{align}

\noindent
where $Y^i$, $\dot Y^i$, $P^i$, $\dot P^i$ are the internal stages and $\Delta t$ is the integration timestep. Let us apply the same partitioned Runge-Kutta method to (\ref{eq:Hamiltonian DAE}). In order to compute the internal stages $Q^i$, $\dot Q^i$ of the $X$ variable we use the state-space form approach, that is, we demand that the constraints and their time derivatives be satisfied (see \cite{HWODE2}). The new step value $X^{n+1}$ is computed by solving the constraints as well. The resulting numerical scheme is thus

\begin{align}
\label{eq:PRK for DAE}
\dot Y^i &= \frac{\partial \tilde H_N}{\partial p}\Big(Y^i,P^i; Q^i,\dot Q^i\Big),\\
\dot P^i &= -\frac{\partial \tilde H_N}{\partial y}\Big(Y^i,P^i; Q^i,\dot Q^i\Big), \nonumber\\
Y^i &= y^n + \Delta t \sum_{j=1}^s a_{ij} \dot Y^j, \nonumber \\
P^i &= p^n + \Delta t \sum_{j=1}^s \bar a_{ij} \dot P^j, \nonumber \\
0&= g(Y^i,Q^i), \nonumber \\
0&= \frac{\partial g}{\partial y}(Y^i,Q^i)\, \dot Y^i + \frac{\partial g}{\partial X}(Y^i,Q^i) \,\dot Q^i, \nonumber\\
y^{n+1} &= y^n + \Delta t \sum_{i=1}^s b_i \dot Y^i, \nonumber \\
p^{n+1} &= p^n + \Delta t \sum_{i=1}^s \bar b_i \dot P^i,  \nonumber\\
0&= g(y^{n+1},X^{n+1}). \nonumber
\end{align}

\noindent
We have the following trivial observation.

\begin{prop}
\label{Thm: PRK for DAE and ODE}
If $X(t)$ is defined to be a $C^1$ interpolation of the internal stages $Q^i$, $\dot Q^i$ at times $t_n+c_i \Delta t$ (that is, if the values $X(t_n+c_i \Delta t)$, $\dot X(t_n+c_i \Delta t)$ coincide with $Q^i$, $\dot Q^i$), then the schemes (\ref{eq:PRK for ODE}) and (\ref{eq:PRK for DAE}) give the same numerical approximations $y^n$, $p^n$ to the exact solution $y(t)$, $p(t)$.
\end{prop}

Intuitively, Proposition~\ref{Thm: PRK for DAE and ODE} states that we can apply a symplectic partitioned Runge-Kutta method to the DAE system (\ref{eq:Hamiltonian DAE}), which solves both for $X(t)$ and $(y(t),p(t))$, and the result will be the same as if we performed a symplectic integration of the Hamiltonian system (\ref{eq:Hamiltonian ODE}) for $(y(t),p(t))$ with a \emph{known} $X(t)$.

\subsection{Moving mesh partial differential equations}
\label{subsec:MMPDEs}

The concept of equidistribution is the most popular paradigm of $r$-adaptation (see \cite{HuangRussellREVIEW}, \cite{HuangRussellBOOK}). Given a continuous mesh density function $\rho(X)$, the equidistribution principle seeks to find a mesh $0=X_0<X_1<...<X_{N+1}=X_{max}$ such that the following holds

\begin{equation}
\label{eq:EquidistributionPrinciple}
\int_0^{X_1}\rho(X)\,dX =\int_{X_1}^{X_2}\rho(X)\,dX= ...=\int_{X_N}^{X_{max}}\rho(X)\,dX,
\end{equation}

\noindent
that is, the quantity represented by the density function is equidistributed among all cells. In the continuous setting we will say that the reparametrization $X=X(x)$ equidistributes $\rho(X)$ if

\begin{equation}
\label{eq:ContinuousEquidistributionPrinciple}
\int_0^{X(x)}\rho(X)\,dX =\frac{x}{X_{max}} \sigma,
\end{equation}

\noindent
where $\sigma=\int_0^{X_{max}}\rho(X)\,dX$ is the total amount of the equidistributed quantity. Differentiate this equation with respect to $x$ to obtain

\begin{equation}
\label{eq:DifferentiatedContinuousEquidistributionPrinciple}
\rho(X(x)) \frac{\partial X}{\partial x} =\frac{1}{X_{max}} \sigma.
\end{equation}

\noindent
It is still a global condition in the sense that $\sigma$ has to be known. For computational purposes it is convenient to differentiate this relation again and consider the following partial differential equation

\begin{equation}
\label{eq:MMPDE0}
\frac{\partial}{\partial x} \Big( \rho(X(x)) \frac{\partial X}{\partial x} \Big) = 0
\end{equation}

\noindent
with the boundary conditions $X(0)=0$, $X(X_{max})=X_{max}$. The choice of the mesh density function $\rho(X)$ is typically problem-dependent and the subject of much research. A popular example is the generalized solution arclength given by

\begin{equation}
\label{eq:ArclengthDensity}
\rho = \sqrt{1+\alpha^2 \Big(\frac{\partial \phi}{\partial X}\Big)^2}=\sqrt{1+\alpha^2 \Big(\frac{\varphi_x}{X_x}\Big)^2}.
\end{equation}

\noindent
It is often used to construct meshes that can follow moving fronts with locally high gradients (\cite{HuangRussellREVIEW}, \cite{HuangRussellBOOK}). With this choice, equation (\ref{eq:MMPDE0}) is equivalent to

\begin{equation}
\label{eq:MMPDE0 - arclength}
\alpha^2 \varphi_x \varphi_{xx}+X_x X_{xx}=0,
\end{equation}

\noindent
assuming $X_x>0$, which we demand anyway. A finite difference discretization on the mesh $x_i=i\cdot \Delta x$ gives us the set of contraints

\begin{align}
\label{eq:ArclengthConstraint}
g_i(y_1,...,y_N,& X_1,...,X_N) = \nonumber \\
& \alpha^2(y_{i+1}-y_i)^2+(X_{i+1}-X_i)^2-\alpha^2(y_{i}-y_{i-1})^2-(X_{i}-X_{i-1})^2 =0, 
\end{align}

\noindent
with the previously defined $y_i$'s and $X_i$'s. This set of constraints can be used in (\ref{eq:Hamiltonian DAE}).

\subsection{Example}
\label{eq:Example1}

To illustrate these ideas let us consider the Lagrangian density

\begin{equation}
\label{eq:ExampleDensity}
\mathcal{L}(\phi,\phi_X,\phi_t)=\frac{1}{2}\phi_t^2 - W(\phi_X).
\end{equation}

\noindent
The reparametrized Lagrangian (\ref{eq:Lagrangian}) takes the form

\begin{equation}
\label{eq:ExampleLagrangian}
\tilde L[\varphi,\varphi_t,t]=\int_0^{X_{max}} \bigg[\frac{1}{2}X_x \Big(\varphi_t-\frac{\varphi_x}{X_x}X_t\Big)^2-W\Big(\frac{\varphi_x}{X_x}\Big) X_x \bigg] \, dx.
\end{equation}

\noindent
Let $N=1$ and $\phi_L=\phi_R=0$. Then

\begin{equation}
\label{eq:ExampleDecomposition}
\varphi(x,t) = y_1(t) \eta_1(x), \quad\quad\quad X(x,t)=X_1(t) \eta_1(x)+X_{max}\,\eta_2(x).
\end{equation}

\noindent
The semi-discrete Lagrangian is

\begin{align}
\label{eq:ExampleLN}
\tilde L_N(y_1,\dot y_1, t) = &\frac{X_1(t)}{6} \bigg( \dot y_1 - \frac{y_1}{X_1(t)} \dot X_1(t)\bigg)^2 + \frac{X_{max}-X_1(t)}{6} \bigg( \dot y_1 + \frac{y_1}{X_{max}-X_1(t)} \dot X_1(t)\bigg)^2 \nonumber \\
                              &-W\bigg(\frac{y_1}{X_1(t)}\bigg) X_1(t) - W\bigg(-\frac{y_1}{X_{max}-X_1(t)}\bigg) \big(X_{max}-X_1(t)\big).
\end{align}

\noindent
The Legendre transform gives $p_1=\partial \tilde L_N / \partial \dot y_1=X_{max} \dot y_1/3$, hence the semi-discrete Hamiltonian is

\begin{align}
\label{eq:ExampleHN}
\tilde H_N(y_1,p_1;X_1, \dot X_1) = &\frac{3}{2 X_{max}} p_1^2-\frac{1}{6}\frac{X_{max} \dot X_1^2}{X_1(X_{max}-X_1)}y_1^2 \nonumber \\
															&+W\Big(\frac{y_1}{X_1}\Big) X_1 + W\Big(-\frac{y_1}{X_{max}-X_1}\Big) (X_{max}-X_1).
\end{align}

\noindent
The corresponding DAE system is

\begin{align}
\label{eq:ExampleDAE}
&\dot y_1 = \frac{3}{X_{max}}p_1,\\
&\dot p_1 = \frac{1}{3}\frac{X_{max} \dot X_1^2}{X_1(X_{max}-X_1)}y_1-W'\Big(\frac{y_1}{X_1}\Big) + W'\Big(-\frac{y_1}{X_{max}-X_1}\Big), \nonumber\\
				&0= g_1(y_1,X_1). \nonumber
\end{align}

\noindent
This system is to be solved for the unknown functions $y_1(t)$, $p_1(t)$ and $X_1(t)$. It is of index 1, because we have three unknown functions and only two differential equations --- the algebraic equation has to be differentiated once in order to obtain a missing ODE.

\subsection{Backward error analysis}
\label{sec:BEA}

The true power of symplectic integration of Hamiltonian equations is revealed through backward error analysis: it can be shown that a symplectic integrator for a Hamiltonian system with the Hamiltonian $H(q,p)$ defines the \emph{exact} flow for a nearby Hamiltonian system, whose Hamiltonian can be expressed as the asymptotic series

\begin{equation}
\label{eq:ModifiedHamiltonian}
\mathscr{H}(q,p) = H(q,p) + \Delta t H_2(q,p) + \Delta t^2 H_3(q,p) + \ldots
\end{equation}

\noindent
Owing to this fact, under some additional assumptions symplectic numerical schemes nearly conserve the original Hamiltonian $H(q,p)$ over exponentially long time intervals. See \cite{HLWGeometric} for details. 

Let us briefly review the results of backward error analysis for the integrator \eqref{eq:PRK for DAE}. Suppose $g(y,X)$ satisfies the assumptions of the Implicit Function Theorem. Then, at least locally, we can solve the constraint $X = h(y)$. The Hamiltonian DAE system \eqref{eq:Hamiltonian DAE} can be then written as the following (implicit) ODE system for $y$ and $p$

\begin{align}
\label{eq:Hamiltonian DAE for BEA}
&\dot y = \frac{\partial \tilde H_N}{\partial p}\Big(y,p; h(y),h'(y) \dot y\Big),\\
&\dot p = -\frac{\partial \tilde H_N}{\partial y}\Big(y,p; h(y),h'(y) \dot y\Big). \nonumber
\end{align}

\noindent
Since we used the state-space formulation, the numerical scheme \eqref{eq:PRK for DAE} is equivalent to applying the same partitioned Runge-Kutta method to \eqref{eq:Hamiltonian DAE for BEA}, that is, we have $Q^i=h(Y^i)$ and $\dot Q^i = h'(Y^i) \dot Y^i$. We computed the corresponding modified equation for several symplectic methods, namely Gauss and Lobatto IIIA-IIIB quadratures. Unfortunately, none of the quadratures resulted in a form akin to \eqref{eq:Hamiltonian DAE for BEA} for some modified Hamiltonian function $\tilde{\mathscr{H}}_N$ related to $\tilde H_N$ by a series similar to \eqref{eq:ModifiedHamiltonian}. This hints at the fact that we should not expect this integrator to show excellent energy conservation over long integration times. One could also consider the implicit ODE system \eqref{eq:Hamiltonian IODE}, which has an obvious triple partitioned structure, and apply a different Runge-Kutta method to each variable $y$, $p$ and~$X$. Although we did not pursue this idea further, it seems unlikely it would bring a desirable result.

We therefore conclude that the control-theoretic strategy, while yielding a perfectly legitimate numerical method, does not take the full advantage of the underlying geometric structures. Let us point out that, while we used a variational discretization of the governing physical PDE, the mesh equations were coupled in a manner that is typical of the existing $r$-adaptive methods (see \cite{HuangRussellREVIEW}, \cite{HuangRussellBOOK}). We now turn our attention to a second approach, which offers a novel way of coupling the mesh equations to the physical equations.

\section{Lagrange multiplier approach to $r$-adaptation}
\label{sec:approach2}

As we saw in Section~\ref{sec:approach1}, discretization of the variational principle alone is not sufficient if we would like to accurately capture the geometric properties of the physical system described by \eqref{eq:action}. In this section we propose a new technique of coupling the mesh equations to the physical equations. Our idea is based on the observation that in $r$-adaptation the number of mesh points is constant, therefore we can treat them as pseudo-particles, and we can incorporate their dynamics into the variational principle. We show that this strategy results in integrators that much better preserve the energy of the considered system.

\subsection{Reparametrized Lagrangian}
\label{subsec:reparametrized lagrangian 2}

In this approach, we treat $X(x,t)$ as an independent field, that is, another degree of freedom, and we will treat the \textquoteleft modified' action \eqref{eq:action2} as a functional of both $\varphi$ and $X$: $\tilde S = \tilde S[\varphi,X]$. For the purpose of the derivations below, we assume that $\varphi(.,t)$ and $X(.,t)$ are continuous and piecewise $C^1$. One could consider the closure of this space in the topology of either Hilbert or Banach space of sufficiently integrable functions and interpret differentiation in a sufficiently weak sense, but this functional-analytic aspect is of little importance for the developments in this section. We refer the interested reader to \cite{EbinMarsden} and \cite{Evans}. As in Section~\ref{subsec:reparametrized lagrangian}, let $\xi(X,t)$ be the function such that $\xi(.,t)=X(.,t)^{-1}$, that is $\xi(X(x,t),t)=x$. Then $\tilde S[\varphi,X] = S[\varphi(\xi(X,t),t)]$. We begin with two propositions and one corollary which will be important for the rest of our exposition.

\begin{prop}
\label{Thm: App2 extremization equivalence}
Extremizing $S[\phi]$ with respect to $\phi$ is equivalent to extremizing $\tilde S[\varphi,X]$ with respect to both $\varphi$ and $X$.
\end{prop}

\begin{proof}
The variational derivatives of $S$ and $\tilde S$ are related by the formula

\begin{align}
\label{eq:approach 2 - variational derivatives}
&\delta_1 \tilde S[\varphi,X] \cdot \delta \varphi(x,t) = \delta S[\varphi(\xi(X,t),t)] \cdot \delta \varphi(\xi(X,t),t),\\
&\delta_2 \tilde S[\varphi,X] \cdot \delta X(x,t) = \delta S[\varphi(\xi(X,t),t)] \cdot \Big( -\frac{\varphi_x(\xi(X,t),t)}{X_x(\xi(X,t),t)} \delta X(\xi(X,t),t) \Big), \nonumber
\end{align} 

\noindent
where $\delta_1$ and $\delta_2$ denote differentiation with respect to the first and second argument, respectively. Suppose $\phi(X,t)$ extremizes $S[\phi]$, i.e. $\delta S[\phi]\cdot\delta \phi=0$ for all variations $\delta \phi$. Choose an arbitrary $X(x,t)$, such that $X(.,t)$ is a (sufficiently smooth) homeomorphism and define $\varphi(x,t)=\phi(X(x,t),t)$. Then by the formula above we have $\delta_1 \tilde S[\varphi,X]=0$ and $\delta_2 \tilde S[\varphi,X]=0$, so the pair $(\varphi,X)$ extremizes $\tilde S$. Conversely, suppose the pair $(\varphi,X)$ extremizes $\tilde S$, that is $\delta_1 \tilde S[\varphi,X] \cdot \delta \varphi =0 $ and $\delta_2 \tilde S[\varphi,X] \cdot \delta X =0$ for all variations $\delta \varphi$ and $\delta X$. Since we assume $X(.,t)$ is a homeomorphism, we can define $\phi(X,t)=\varphi(\xi(X,t),t)$. Note that an arbitrary variation $\delta \phi(X,t)$ induces the variation $\delta \varphi(x,t) = \delta \phi(X(x,t),t)$. Then we have $\delta S[\phi] \cdot \delta \phi = \delta_1 \tilde S[\varphi,X] \cdot \delta \varphi = 0$ for all variations $\delta \phi$, so $\phi(X,t)$ extremizes $S[\phi]$.\\
\end{proof}

\begin{prop}
\label{Thm: E-L equations dependence}
The equation $\delta_2 \tilde S[\varphi,X] = 0$ is implied by the equation $\delta_1 \tilde S[\varphi,X] = 0$.
\end{prop}

\begin{proof}
As we saw in the proof of Proposition~\ref{Thm: App2 extremization equivalence}, the condition $\delta_1 \tilde S[\varphi,X] \cdot \delta \varphi = 0$ implies $\delta S =0$. By (\ref{eq:approach 2 - variational derivatives}), this in turn implies $\delta_2 \tilde S[\varphi,X] \cdot \delta X = 0$ for all $\delta X$. Note that this argument cannot be reversed: $\delta_2 \tilde S[\varphi,X] \cdot \delta X = 0$ does not imply $\delta S =0$ when $\varphi_x=0$.\\
\end{proof}

\begin{corollary}
\label{Thm: Degeneracy}
The field theory described by $\tilde S[\varphi,X]$ is degenerate and the solutions to the Euler-Lagrange equations are not unique.
\end{corollary}

\subsection{Spatial Finite Element discretization}
\label{subsec:FEM2}

The Lagrangian of the \textquoteleft reparametrized' theory $\tilde L: Q \times G \times W \times Z \longrightarrow \mathbb{R}$,

\begin{equation}
\label{eq:Lagrangian2}
\tilde{L}[\varphi,X,\varphi_t,X_t]= \int_0^{X_{max}} \mathcal{L} \Big(\varphi,\frac{\varphi_x}{X_x},\varphi_t-\frac{\varphi_x X_t}{X_x}\Big) X_x\,dx,
\end{equation}

\noindent
has the same form as \eqref{eq:Lagrangian} (we only treat it as a functional of $X$ and $X_t$ as well), where $Q$, $G$, $W$ and $Z$ are spaces of continuous and piecewise $C^1$ functions, as mentioned before. We again let $\Delta x = X_{max}/(N+1)$ and define the uniform mesh $x_i=i \cdot \Delta x$ for $i=0,1,...,N+1$. Define the finite element spaces

\begin{equation}
\label{eq:QNGNWNZN}
Q_N=G_N=W_N=Z_N=\text{span}(\eta_0,...,\eta_{N+1}),
\end{equation}

\noindent
where we used the finite elements \eqref{eq:basis functions}. We have $Q_N \subset Q$, $G_N \subset G$, $W_N \subset W$, $Z_N \subset Z$. In addition to \eqref{eq:phi FEM} we also consider

\begin{equation}
\label{eq:X FEM2}
X(x) = \sum_{i=0}^{N+1} X_i \eta_i(x), \quad\quad\quad \dot X(x) = \sum_{i=0}^{N+1} \dot X_i \eta_i(x).
\end{equation}

\noindent
The numbers $(y_i,X_i,\dot y_i,\dot X_i)$ thus form natural (global) coordinates on $Q_N \!\times\! G_N \!\times\!  W_N \!\times\!  Z_N$. We again consider the restricted Lagrangian $\tilde L_N= \tilde L|_{Q_N \times G_N \times W_N \times Z_N}$. In the chosen coordinates

\begin{equation}
\label{eq:LN FEM2}
\tilde L_N(y_1,...,y_{N},X_1,...,X_N,\dot y_1,...,\dot y_{N}, \dot X_1,...,\dot X_{N}) = \tilde L \Big[\varphi(x),X(x),\dot \varphi(x), \dot X(x) \Big],
\end{equation}

\noindent
where $\varphi(x)$, $X(x)$, $\dot \varphi(x)$, $\dot X(x)$ are defined by \eqref{eq:phi FEM} and \eqref{eq:X FEM2}. Once again, we refrain from writing $y_0$, $y_{N+1}$, $\dot y_0$, $\dot y_{N+1}$, $X_0$, $X_{N+1}$, $\dot X_0$ and $\dot X_{N+1}$ as arguments of $\tilde L_N$ in the remainder of this section, as those are not actual degrees of freedom.

\subsection{Invertibility of the Legendre Transform}
\label{subsec:Invertibility}

For simplicity, let us restrict our considerations to Lagrangian densities of the form

\begin{equation}
\label{eq:ParticularDensity}
\mathcal{L}(\phi,\phi_X,\phi_t)=\frac{1}{2}\phi_t^2 - R(\phi_X, \phi).
\end{equation}

\noindent
We chose a kinetic term that is most common in applications. The corresponding \textquoteleft reparametrized' Lagrangian is

\begin{equation}
\label{eq:ParticularLagrangian}
\tilde L[\varphi,X,\varphi_t,X_t]=\int_0^{X_{max}} \frac{1}{2}X_x \Big(\varphi_t-\frac{\varphi_x}{X_x}X_t\Big)^2 \, dx -\ldots,
\end{equation}

\noindent
where we kept only the terms that involve the velocities $\varphi_t$ and $X_t$. The semi-discrete Lagrangian becomes

\begin{align}
\label{eq:ParticularLN}
\tilde L_N = \sum_{i=0}^N \frac{X_{i+1}-X_i}{6} \Big[  \Big(\dot y_i-\frac{y_{i+1}-y_i}{X_{i+1}-X_i} \dot X_i\Big)^2 &+\Big(\dot y_i-\frac{y_{i+1}-y_i}{X_{i+1}-X_i} \dot X_i\Big)\Big(\dot y_{i+1}-\frac{y_{i+1}-y_i}{X_{i+1}-X_{i}} \dot X_{i+1}\Big) \nonumber \\
                    &+ \Big(\dot y_{i+1}-\frac{y_{i+1}-y_i}{X_{i+1}-X_{i}} \dot X_{i+1}\Big)^2 \Big] - \ldots
\end{align}

\noindent
Let us define the conjugate momenta via the Legendre Transform

\begin{equation}
\label{eq:LegendreTransform}
p_i = \frac{\partial \tilde L_N}{\partial \dot y_i}, \quad\quad\quad\quad\quad S_i = \frac{\partial \tilde L_N}{\partial \dot X_i}, \quad\quad\quad\quad\quad i=1,2,...,N.
\end{equation}

\noindent
This can be written as

\begin{equation}
\label{eq:MatrixLegendre}
\left(\begin{array}{c}p_1\\S_1\\ \vdots \\ p_N\\S_N \end{array} \right)
=
\tilde M_N(y,X) \cdot
\left(\begin{array}{c}\dot y_1\\\dot X_1\\ \vdots \\ \dot y_N\\\dot X_N \end{array} \right),
\end{equation}

\noindent
where the $2N\times 2N$ mass matrix $\tilde M_N(y,X)$ has the following block tridiagonal structure

\begin{equation}
\label{eq:MassMatrix}
\tilde M_N(y,X) =
\left(\begin{array}{cccccc}
A_1 & B_1 &       &       &         &\\ 
B_1 & A_2 & B_2   &       &         &\\
    & B_2 & A_3   & B_3   &         &\\
    &     &\ddots &\ddots &\ddots   &\\
    &     &       &\ddots & \ddots  &B_{N-1}\\
    &     &       &       & B_{N-1} & A_N\\

\end{array} \right),
\end{equation}

\noindent
with the $2\times 2$ blocks

\begin{equation}
\label{eq:MassMatrixBlocks}
A_i =
\left(\begin{array}{cc}
\frac{1}{3}\delta_{i-1} + \frac{1}{3}\delta_i & -\frac{1}{3}\delta_{i-1}\gamma_{i-1} - \frac{1}{3}\delta_i\gamma_i \\
-\frac{1}{3}\delta_{i-1}\gamma_{i-1} - \frac{1}{3}\delta_i\gamma_i & \frac{1}{3}\delta_{i-1}\gamma_{i-1}^2 + \frac{1}{3}\delta_i\gamma_i^2
\end{array} \right),\quad\quad\:
B_i =
\left(\begin{array}{cc}
\frac{1}{6}\delta_i &  -\frac{1}{6}\delta_i\gamma_i \\
-\frac{1}{6}\delta_i\gamma_i & \frac{1}{6}\delta_i\gamma_i^2
\end{array} \right),
\end{equation}

\noindent
where

\begin{equation}
\label{eq:GammaDelta}
\delta_i = X_{i+1}-X_i, \quad\quad\quad\quad \gamma_i = \frac{y_{i+1}-y_i}{X_{i+1}-X_i}.
\end{equation}

\noindent
From now on we will always assume $\delta_i>0$, as we demand that $X(x) = \sum_{i=0}^{N+1} X_i \eta_i(x)$ be a homeomorphism. We also have

\begin{equation}
\label{eq:detAi}
\det A_i = \frac{1}{9} \delta_{i-1} \delta_i (\gamma_{i-1}-\gamma_i)^2.
\end{equation}

\noindent

\begin{prop}
\label{Thm: Mass Matrix Degeneracy}
The mass matrix $\tilde M_N(y,X)$ is non-singular almost everywhere (as a function of the $y_i$'s and $X_i$'s) and singular iff $\gamma_{i-1}=\gamma_i$ for some $i$.
\end{prop}

\begin{proof} 
We will compute the determinant of $\tilde M_N(y,X)$ by transforming \eqref{eq:MassMatrix} into a block upper triangular form by zeroing the blocks $B_i$ below the diagonal. Let us start with the block $B_1$. We use linear combinations of the first two rows of the mass matrix to zero the elements of the block $B_1$ below the diagonal. Suppose $\gamma_0=\gamma_1$. Then it is easy to see that the first two rows of the mass matrix are not linearly independent, so the determinant of the mass matrix is zero. Assume $\gamma_0\neq\gamma_1$. Then by \eqref{eq:detAi} the block $A_1$ is invertible. We multiply the first two rows of the mass matrix by $B_1 A_1^{-1}$ and subtract the result from the third and fourth rows. This zeroes the block $B_1$ below the diagonal and replaces the block $A_2$ by

\begin{equation}
\label{eq:C2}
C_2 = A_2 - B_1 A_1^{-1} B_1.
\end{equation}

\noindent
We now zero the block $B_2$ below the diagonal in a similar fashion. After $n-1$ steps of this procedure the mass matrix is transformed into

\begin{equation}
\label{eq:MassMatrixZeroing}
\left(\begin{array}{ccccccc}
C_1 & B_1 &       &       &         &       &\\ 
    & C_2 & B_2   &       &         &       &\\
    &     &\ddots &\ddots &         &       &\\
    &     &       &C_n    & B_n     &       &\\
    &     &       &B_n    & A_{n+1} &\ddots &\\
    &     &       &       & \ddots  &\ddots &B_{N-1}\\
    &     &       &       &         &B_{N-1}&A_N\\
\end{array} \right).
\end{equation}

\noindent
In a moment we will see that $C_n$ is singular iff $\gamma_{n-1}=\gamma_n$ and in that case the two rows of the matrix above that contain $C_n$ and $B_n$ are linearly dependent, thus making the mass matrix singular. Suppose $\gamma_{n-1}\neq\gamma_n$, so that $C_n$ is invertible. In the next step of our procedure the block $A_{n+1}$ is replaced by

\begin{equation}
\label{eq:Cn1}
C_{n+1} = A_{n+1} - B_n C_n^{-1} B_n.
\end{equation}

\noindent
Together with the condition $C_1=A_1$ this gives us a recurrence. By induction on $n$ we find that

\begin{equation}
\label{eq:C general formula}
C_n=
\left(\begin{array}{cc}
\frac{1}{4}\delta_{n-1} + \frac{1}{3}\delta_n & -\frac{1}{4}\delta_{n-1}\gamma_{n-1} - \frac{1}{3}\delta_n\gamma_n \\
-\frac{1}{4}\delta_{n-1}\gamma_{n-1} - \frac{1}{3}\delta_n\gamma_n & \frac{1}{4}\delta_{n-1}\gamma_{n-1}^2 + \frac{1}{3}\delta_n\gamma_n^2
\end{array} \right)
\end{equation}

\noindent
and

\begin{equation}
\label{eq:detCi}
\det C_i = \frac{1}{12} \delta_{i-1} \delta_i (\gamma_{i-1}-\gamma_i)^2,
\end{equation}

\noindent
which justifies our assumptions on the invertibility of the blocks $C_i$. We can now express the determinant of the mass matrix as $\det C_1 \cdot ... \cdot \det C_N$. The final formula is

\begin{equation}
\label{eq:detMassMatrix}
\det \tilde M_N(y,X) = \frac{\delta_0 \delta_1^2...\delta_{N-1}^2 \delta_N}{9\cdot 12^{N-1}}(\gamma_0-\gamma_1)^2...(\gamma_{N-1}-\gamma_N)^2.
\end{equation}

\noindent
We see that the mass matrix becomes singular iff $\gamma_{i-1}=\gamma_i$ for some $i$ and this condition defines a measure zero subset of $\mathbb{R}^{2N}$.\\
\end{proof}

\paragraph{Remark I.} This result shows that the finite-dimensional system described by the semi-discrete Lagrangian \eqref{eq:ParticularLN} is non-degenerate almost everywhere. This means that, unlike in the continuous case, the Euler-Lagrange equations corresponding to the variations of the $y_i$'s and $X_i$'s are independent of each other (almost everywhere) and the equations corresponding to the $X_i$'s are in fact necessary for the correct description of the dynamics. This can also be seen in a more general way. Owing to the fact we are considering a finite element approximation, the semi-discrete action functional $\tilde S_N$ is simply a restriction of $\tilde S$, and therefore formulas \eqref{eq:approach 2 - variational derivatives} still hold. The corresponding Euler-Lagrange equations take the form

\begin{align}
\label{eq:Semi-discrete E-L}
&\delta_1 \tilde S[\varphi,X] \cdot \delta \varphi(x,t) = 0,\\
&\delta_2 \tilde S[\varphi,X] \cdot \delta X(x,t) = 0, \nonumber
\end{align} 

\noindent
which must hold for all variations $\delta \varphi(x,t) \!=\! \sum_{i=1}^N \delta y_i(t) \eta_i(x)$ and $\delta X(x,t) \!=\! \sum_{i=1}^N \delta X_i(t) \eta_i(x)$. Since we are working in a finite dimensional subspace, the second equation now does not follow from the first equation. To see this, consider a particular variation $\delta X(x,t) = \delta X_k(t) \eta_k(x)$ for some $k$, where $\delta X_k\not\equiv 0$. Then we have

\begin{equation}
\label{eq:Discontinuous variation}
-\frac{\varphi_x}{X_x} \delta X_k(t) = 
\left\{
\begin{array}{cc}
-\gamma_{k-1} \, \delta X_k(t) \, \eta_k(x), & \quad\mbox{if } x_{k-1} \leq x \leq x_k,\\
-\gamma_{k} \, \delta X_k(t) \, \eta_k(x), & \quad \mbox{if } x_k \leq x \leq x_{k+1},\\
0, & \mbox{otherwise,}\\
\end{array}
\right.
\end{equation}

\noindent
which is discontinuous at $x=x_k$ and cannot be expressed as $\sum_{i=1}^N \delta y_i(t) \eta_i(x)$ for any $\delta y_i(t)$, unless $\gamma_{k-1}=\gamma_k$. Therefore, we cannot invoke the first equation to show that $\delta_2 \tilde S[\varphi,X] \cdot \delta X(x,t) = 0$. The second equation becomes independent.

\paragraph{Remark II.} It is also instructive to realize what exactly happens when $\gamma_{k-1}=\gamma_k$. This means that locally in the interval $[X_{k-1},X_{k+1}]$ the field $\phi(X,t)$ is a straight line with the slope $\gamma_k$. It also means that there are infinitely many values $(X_k,y_k)$ that reproduce the same local shape of $\phi(X,t)$. This reflects the arbitrariness of $X(x,t)$ in the infinite-dimensional setting. In the finite element setting, however, this holds only when the points $(X_{k-1},y_{k-1})$, $(X_{k},y_{k})$ and $(X_{k+1},y_{k+1})$ line up. Otherwise any change to the middle point changes the shape of $\phi(X,t)$. See Figure~\ref{fig: Degeneracy}.

\begin{figure}[tbp]
	\centering
		\includegraphics[width=0.9\textwidth]{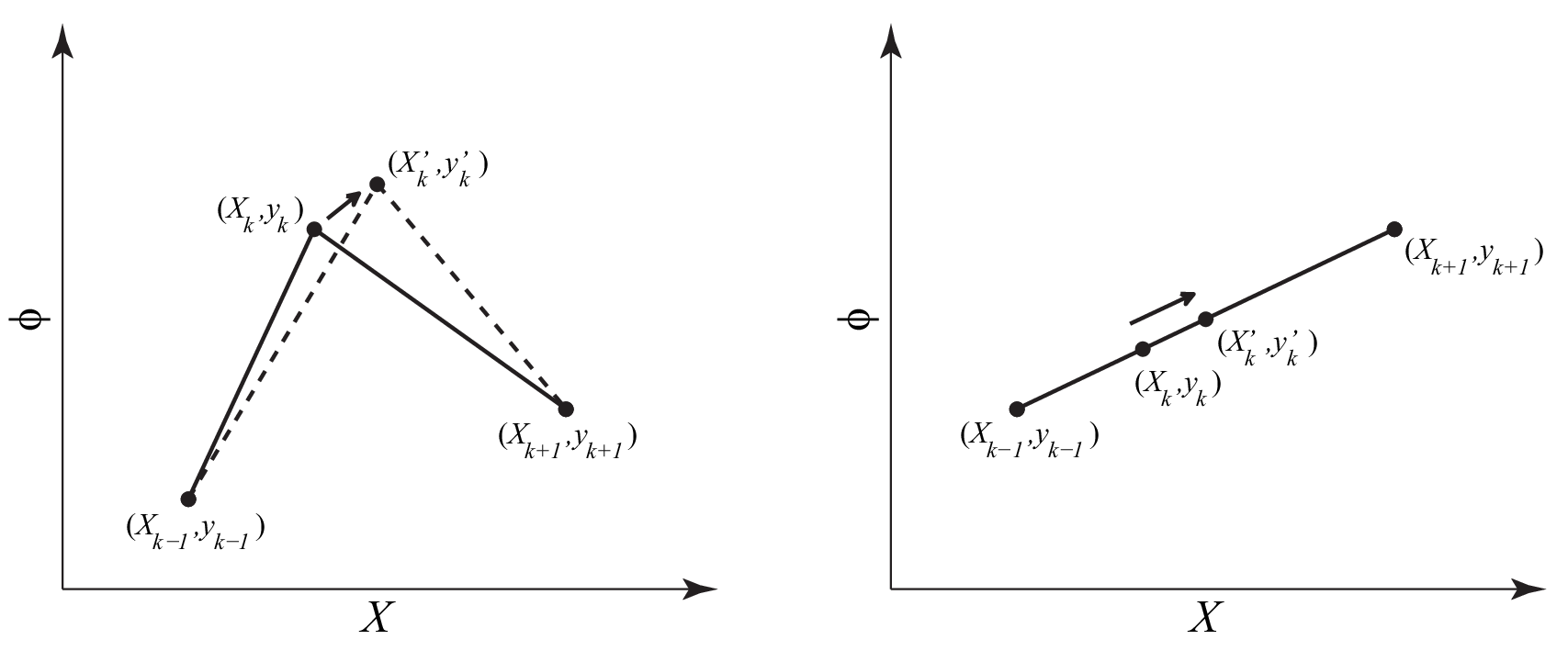}
		\caption{\emph{Left:} If $\gamma_{k-1}\not=\gamma_k$, then any change to the middle point changes the local shape of $\phi(X,t)$. \emph{Right:} If $\gamma_{k-1}=\gamma_k$, then there are infinitely many possible positions for $(X_{k},y_{k})$ that reproduce the local linear shape of $\phi(X,t)$.}
		\label{fig: Degeneracy}
\end{figure}

\subsection{Existence and uniqueness of solutions}
\label{subsec:Existence and Uniqueness}

Since the Legendre Transform \eqref{eq:MatrixLegendre} becomes singular at some points, this raises a question about the existence and uniqueness of the solutions to the Euler-Lagrange equations \eqref{eq:Semi-discrete E-L}. In this section we provide a partial answer to this problem. We will begin by computing the Lagrangian symplectic form

\begin{equation}
\label{eq:Symplectic form definition}
\tilde \Omega_N = \sum_{i=1}^N dy_i \wedge dp_i + dX_i \wedge dS_i, 
\end{equation}

\noindent
where $p_i$ and $S_i$ are given by \eqref{eq:LegendreTransform}. For notational convenience we will collectively denote $q=(y_1,X_1,...,y_N,X_N)^T$ and $\dot q=(\dot y_1, \dot X_1,...,\dot y_N,\dot X_N)^T$. Then in the ordered basis $(\frac{\partial}{\partial q_1},...,\frac{\partial}{\partial q_{2N}},\frac{\partial}{\partial \dot q_1},...,\frac{\partial}{\partial \dot q_{2N}})$ the symplectic form can be represented by the matrix

\begin{equation}
\label{eq:Symplectic Form Blocks}
\tilde \Omega_N (q, \dot q) =
\left(\begin{array}{cc}
 \tilde \Delta_N(q,\dot q) & \tilde M_N(q)  \\
 -\tilde M_N(q) & 0 
\end{array} \right),
\end{equation}

\noindent
where the $2N \times 2N$ block $\tilde \Delta_N (q, \dot q)$ has the further block tridiagonal structure

\begin{equation}
\label{eq:Delta Matrix Blocks}
\tilde \Delta_N(q,\dot q) =
\left(\begin{array}{cccccc}
\Gamma_1     & \Lambda_1    &             &             &         				 &\\ 
-\Lambda_1^T & \Gamma_2     & \Lambda_2   &             &         				 &\\
             & -\Lambda_2^T & \Gamma_3    & \Lambda_3   &         				 &\\
             &              &\ddots       &\ddots       &\ddots   				 &\\
   					 & 					    &   			    &\ddots 			& \ddots           &\Lambda_{N-1}\\
   					 &					    &			        &   		      & -\Lambda_{N-1}^T & \Gamma_N\\

\end{array} \right)
\end{equation}

\noindent
with the $2 \times 2$ blocks

\begin{align}
\label{eq:Gamma Lambda Blocks}
\Gamma_i &=
\left(\begin{array}{cc}
 0 & -\frac{\dot y_{i+1}-\dot y_{i-1}}{3}-\frac{\dot X_{i-1}+2\dot X_{i}}{3} \gamma_{i-1}+\frac{2\dot X_{i}+\dot X_{i+1}}{3} \gamma_{i} \\
\frac{\dot y_{i+1}-\dot y_{i-1}}{3}+\frac{\dot X_{i-1}+2\dot X_{i}}{3} \gamma_{i-1}-\frac{2\dot X_{i}+\dot X_{i+1}}{3} \gamma_{i} & 0
\end{array} \right), \nonumber \\
\Lambda_i &=
\left(\begin{array}{cc}
-\frac{\dot X_i + \dot X_{i+1}}{2}  &  -\frac{\dot y_{i+1}-\dot y_{i}}{6}+\frac{\dot X_{i}+2\dot X_{i+1}}{3} \gamma_{i} \\
\frac{\dot y_{i+1}-\dot y_{i}}{6}+\frac{2 \dot X_{i}+\dot X_{i+1}}{3} \gamma_{i} & -\frac{\dot X_i + \dot X_{i+1}}{2} \gamma_i^2
\end{array} \right).
\end{align}

\noindent
In this form, it is easy to see that

\begin{equation}
\label{eq:Symplectic Form Determinant}
\det \tilde \Omega_N(q,\dot q) = \Big(\det \tilde M_N(q)\Big)^2,
\end{equation}

\noindent
so the symplectic form is singular whenever the mass matrix is.

The energy corresponding to the Lagrangian \eqref{eq:ParticularLN} can be written as

\begin{equation}
\label{eq:EN with R terms}
\tilde E_N(q,\dot q) = \frac{1}{2} \dot q^T \tilde M_N(q) \, \dot q + \sum_{k=0}^N \int_{x_k}^{x_{k+1}} R\Big( \gamma_k,y_k \eta_k(x)+y_{k+1} \eta_{k+1}(x)\Big)\frac{X_{k+1}-X_k}{\Delta x} \, dx.
\end{equation}

\noindent
In the chosen coordinates, $d\tilde E_N$ can be represented by the row vector $d\tilde E_N=(\partial \tilde E_N / \partial q_1,...,\partial \tilde E_N / \partial \dot q_{2N})$. It turns out that

\begin{equation}
\label{eq:dE}
d\tilde E_N^T(q,\dot q) = 
\left(\begin{array}{c}
 \xi \\
\tilde M_N(q)\dot q
\end{array} \right),
\end{equation}

\noindent
where the vector $\xi$ has the following block structure

\begin{equation}
\label{eq:Ksi}
\xi = 
\left(\begin{array}{c}
\xi_1\\
\vdots\\
\xi_N
\end{array} \right).
\end{equation}

\noindent
Each of these blocks has the form $\xi_k = (\xi_{k,1}, \xi_{k,2})^T$. Through basic algebraic manipulations and integration by parts, one finds that

\begin{align}
\label{eq:Ksi1}
\xi_{k,1} = &\quad\,\frac{\dot y_{k+1} (2 \dot X_{k+1}+\dot X_k) + \dot y_k (\dot X_{k+1}-\dot X_{k-1})-\dot y_{k-1}(\dot X_{k}+2 \dot X_{k-1})}{6} \nonumber\\
						&+\frac{\dot X_k^2 +\dot X_k \dot X_{k-1}+\dot X_{k-1}^2}{3} \gamma_{k-1} - \frac{\dot X_{k+1}^2 +\dot X_{k+1} \dot X_{k}+\dot X_{k}^2}{3} \gamma_{k} \nonumber \\
						&+\frac{1}{\Delta x} \int_{x_{k-1}}^{x_k} \frac{\partial R}{\partial \phi_X}\Big( \gamma_{k-1},y_{k-1} \eta_{k-1}(x)+y_{k} \eta_{k}(x)\Big)\,dx \nonumber \\
						&- \frac{1}{\Delta x} \int_{x_{k}}^{x_{k+1}} \frac{\partial R}{\partial \phi_X}\Big( \gamma_{k},y_{k} \eta_{k}(x)+y_{k+1} \eta_{k+1}(x)\Big)\,dx\\
						&+\frac{1}{\gamma_{k-1}} \Big[R(\gamma_{k-1},y_k) - \frac{1}{\Delta x} \int_{x_{k-1}}^{x_k} R\Big( \gamma_{k-1},y_{k-1} \eta_{k-1}(x)+y_{k} \eta_{k}(x)\Big)\,dx \Big] \nonumber \\
						&-\frac{1}{\gamma_{k}} \Big[R(\gamma_{k},y_k) - \frac{1}{\Delta x} \int_{x_{k}}^{x_{k+1}} R\Big( \gamma_{k},y_{k} \eta_{k}(x)+y_{k+1} \eta_{k+1}(x)\Big)\,dx \Big], \nonumber
\end{align}						
						
\noindent
and						

\begingroup
\allowdisplaybreaks						
\begin{align}
\label{eq:Ksi2}
\xi_{k,2} = &\quad\,\frac{\dot y_{k-1}^2+\dot y_{k-1} \dot y_{k} -\dot y_{k}\dot y_{k+1} - \dot y_{k+1}^2 }{6} \nonumber\\
						&-\frac{\dot X_k^2 +\dot X_k \dot X_{k-1}+\dot X_{k-1}^2}{6} \gamma_{k-1}^2 + \frac{\dot X_{k+1}^2 +\dot X_{k+1} \dot X_{k}+\dot X_{k}^2}{6} \gamma_{k}^2 \nonumber \\
						&-\frac{\gamma_{k-1}}{\Delta x} \int_{x_{k-1}}^{x_k} \frac{\partial R}{\partial \phi_X}\Big( \gamma_{k-1},y_{k-1} \eta_{k-1}(x)+y_{k} \eta_{k}(x)\Big)\,dx \nonumber \\
						&+ \frac{\gamma_k}{\Delta x} \int_{x_{k}}^{x_{k+1}} \frac{\partial R}{\partial \phi_X}\Big( \gamma_{k},y_{k} \eta_{k}(x)+y_{k+1} \eta_{k+1}(x)\Big)\,dx \\
						&+\frac{1}{\Delta x} \int_{x_{k-1}}^{x_k} R\Big( \gamma_{k-1},y_{k-1} \eta_{k-1}(x)+y_{k} \eta_{k}(x)\Big)\,dx \nonumber \\
						&-\frac{1}{\Delta x} \int_{x_{k}}^{x_{k+1}} R\Big( \gamma_{k},y_{k} \eta_{k}(x)+y_{k+1} \eta_{k+1}(x)\Big)\,dx. \nonumber
\end{align}
\endgroup

\noindent
We are now ready to consider the generalized Hamiltonian equation 

\begin{equation}
i_Z \tilde \Omega_N = d \tilde E_N, 
\end{equation}

\noindent
which we solve for the vector field $Z = \sum_{i=1}^{2N} \alpha_i \,\partial/\partial q_i + \beta_i \, \partial / \partial \dot q_i$. In the matrix representation this equation takes the form

\begin{equation}
\label{eq:Generalized Hamiltonian Equation}
\tilde \Omega_N^T (q,\dot q) \cdot
\left(\begin{array}{c}
 \alpha \\
 \beta
\end{array} \right)
= d\tilde E_N^T(q,\dot q).
\end{equation}

\noindent
Equations of this form are called (quasilinear) implicit ODEs (see \cite{Rabier6}, \cite{Reissig}). If the symplectic form is nonsingular in a neighborhood of $(q^{(0)},\dot q^{(0)})$, then the equation can be solved directly via

$$Z = [\tilde \Omega_N^T (q,\dot q)]^{-1} d\tilde E_N^T(q,\dot q)$$

\noindent
to obtain the standard explicit ODE form and standard existence/uniqueness theorems (Picard's, Peano's, etc.) of ODE theory can be invoked to show local existence and uniqueness of the flow of $Z$ in a neighborhood of $(q^{(0)},\dot q^{(0)})$. If, however, the symplectic form is singular at $(q^{(0)},\dot q^{(0)})$, then there are two possibilities. The first case is

\begin{equation}
\label{eq:Algebraic singularity}
d\tilde E_N^T(q^{(0)},\dot q^{(0)}) \not \in \text{Range } \tilde \Omega_N^T (q^{(0)},\dot q^{(0)})
\end{equation}

\noindent
and it means there is no solution for $Z$ at $(q^{(0)},\dot q^{(0)})$. This type of singularity is called an \emph{algebraic} one and it leads to so called \emph{impasse points} (see \cite{Rabier1}-\cite{Rabier6}, \cite{Reissig}). 

The other case is

\begin{equation}
\label{eq:Geometric singularity}
d\tilde E_N^T(q^{(0)},\dot q^{(0)}) \in \text{Range } \tilde \Omega_N^T (q^{(0)},\dot q^{(0)})
\end{equation}

\noindent
and it means that there exists a nonunique solution $Z$ at $(q^{(0)},\dot q^{(0)})$. This type of singularity is called a \emph{geometric} one. If $(q^{(0)},\dot q^{(0)})$ is a limit of regular points of \eqref{eq:Generalized Hamiltonian Equation} (i.e. points where the symplectic form is nonsingular), then there might exist an integral curve of $Z$ passing through $(q^{(0)},\dot q^{(0)})$. See \cite{Rabier1}, \cite{Rabier2}, \cite{Rabier3}, \cite{Rabier4}, \cite{Rabier5}, \cite{Rabier6}, \cite{Reissig} for more details.

\begin{prop}
\label{Thm: Geometric singularities}
The singularities of the symplectic form $\tilde \Omega_N (q,\dot q)$ are geometric.
\end{prop}

\begin{proof}
Suppose that the mass matrix (and thus the symplectic form) is singular at $(q^{(0)},\dot q^{(0)})$. Using the block structures \eqref{eq:Symplectic Form Blocks} and \eqref{eq:dE} we can write \eqref{eq:Generalized Hamiltonian Equation} as the system

\begin{align}
\label{eq:Split Generalized Hamiltonian Equations}
-\tilde \Delta_N(q^{(0)},\dot q^{(0)})\, \alpha - \tilde M_N(q^{(0)})\, \beta &= \xi, \nonumber \\
\tilde M_N(q^{(0)}) \,\alpha & = \tilde M_N(q^{(0)})\,\dot q^{(0)}.
\end{align}

\noindent
The second equation implies that there exists a solution $\alpha = \dot q^{(0)}$. In fact this is the only solution we are interested in, since it satisfies the second order condition: the Euler-Lagrange equations underlying the variationl principle are second order, so we are only interested in solutions of the form $Z = \sum_{i=1}^{2N} \dot q_i \,\partial/\partial q_i + \beta_i \, \partial / \partial \dot q_i$. The first equation can be rewritten as

\begin{align}
\label{eq:Equation for beta}
\tilde M_N(q^{(0)})\, \beta &= -\xi - \tilde \Delta_N(q^{(0)},\dot q^{(0)})\, \dot q^{(0)}.
\end{align}

\noindent
Since the mass matrix is singular, we must have $\gamma_{k-1}=\gamma_k$ for some $k$. As we saw in Section~\ref{subsec:Invertibility}, this means that the two rows of the $k^{\text{th}}$ \textquoteleft block row' of the mass matrix (i.e., the rows containing the blocks $B_{k-1}$, $A_k$ and $B_k$) are not linearly independent. In fact we have

\begin{equation}
\label{eq:Linear dependence of kth block row}
(B_{k-1})_{2*} = -\gamma_k (B_{k-1})_{1*}, \quad\quad\quad(A_k)_{2*} = -\gamma_k (A_k)_{1*}, \quad\quad\quad(B_k)_{2*} = -\gamma_k (B_k)_{1*},
\end{equation}

\noindent
where $a_{m*}$ denotes the $m^\textrm{th}$ row of the matrix $a$. Equation \eqref{eq:Equation for beta} will have a solution for $\beta$ iff the RHS satisfies a similar scaling condition in the the $k^{\text{th}}$ \textquoteleft block element'. Using formulas \eqref{eq:Gamma Lambda Blocks}, \eqref{eq:Ksi1} and \eqref{eq:Ksi2}, we show that $-\xi - \tilde \Delta_N\, \dot q^{(0)}$ indeed has this property. Hence, $d\tilde E_N^T(q^{(0)},\dot q^{(0)}) \in \text{Range } \tilde \Omega_N^T (q^{(0)},\dot q^{(0)})$ and $(q^{(0)},\dot q^{(0)})$ is a geometric singularity. Moreover, since $\gamma_{k-1}=\gamma_k$ defines a hypersurface in $\mathbb{R}^{2 N} \times \mathbb{R}^{2 N}$, $(q^{(0)},\dot q^{(0)})$ is a limit of regular points. 
\end{proof}

\paragraph{Remark I.} Numerical time integration of the semi-discrete equations of motion (\ref{eq:Generalized Hamiltonian Equation}) has to deal with the singularity points of the symplectic form. While there are some numerical algorithms allowing one to get past singular hypersurfaces (see \cite{Rabier6}), it might not be very practical from the application point of view. Note that, unlike in the continuous case, the time evolution of the meshpoints $X_i$'s is governed by the equations of motion, so the user does not have any influence on how the mesh is adapted. More importantly, there is no built-in mechanism that would prevent mesh tangling. Some preliminary numerical experiments show that the mesh points eventually collapse when started with nonzero initial velocities.

\paragraph{Remark II.} The singularities of the mass matrix \eqref{eq:MassMatrix} bear some similarities to the singularities of the mass matrices encountered in the Moving Finite Element method. In \cite{Miller1} and \cite{Miller2} the authors proposed introducing a small \textquoteleft internodal' viscosity which penalizes the method for relative motion between the nodes and thus regularizes the mass matrix. A similar idea could be applied in our case: one could add some small $\varepsilon$ kinetic terms to the Lagrangian \eqref{eq:ParticularLN} in order to regularize the Legendre Transform. In light of the remark made above, we did not follow this idea further and decided to take a different route instead, as described in the following sections. However, investigating further similarities between our variational approach and the Moving Finite Element method might be worthwhile. There also might be some connection to the $r$-adaptive method presented in \cite{Zielonka}: the evolution of the mesh in that method is also set by the equations of motion, although the authors considered a different variational principle and different theoretical reasoning to justify the validity of their approach.

\subsection{Constraints and adaptation strategy}
\label{subsec:Constraints and adaptation strategy}

As we saw in Section~\ref{subsec:Existence and Uniqueness}, upon discretization we lose the arbitrariness of $X(x,t)$ and the evolution of $X_i(t)$ is governed by the equations of motion, while we still want to be able to select a desired mesh adaptation strategy, like \eqref{eq:ArclengthConstraint}. This could be done by augmenting the Lagrangian \eqref{eq:ParticularLN} with Lagrange multipliers corresponding to each constraint $g_i$. However, it is not obvious that the dynamics of the constrained system as defined would reflect in any way the behavior of the approximated system \eqref{eq:ParticularDensity}. We will show that the constraints can be added via Lagrange multipliers already at the continuous level \eqref{eq:ParticularDensity} and the continuous system as defined can be then discretized to arrive at \eqref{eq:ParticularLN} with the desired adaptation constraints.

\subsubsection{Global constraint}
\label{subsubsec:Global constraint}

As mentioned before, eventually we would like to impose the constraints 

\begin{equation}
\label{eq:constraints discrete}
g_i(y_1,...,y_N,X_1,...,X_N)=0 \quad\quad\quad i=1,...,N
\end{equation}

\noindent
on the semi-discrete system \eqref{eq:ParticularLN}. Let us assume that $g: \mathbb{R}^{2N}\longrightarrow \mathbb{R}^N$, $g = (g_1,...,g_N)^T$ is $C^1$ and $0$ is a regular value of $g$, so that \eqref{eq:constraints discrete} defines a submanifold. To see how these constraints can be introduced at the continuous level, let us select uniformly distributed points $x_i=i\cdot \Delta x$, $i=0,...,N+1$, $\Delta x = X_{max}/(N+1)$ and demand that the constraints

\begin{equation}
\label{eq:constraints continuous}
g_i\Big(\varphi(x_1,t),...,\varphi(x_N,t),X(x_1,t),...,X(x_N,t)\Big)=0, \quad\quad\quad i=1,...,N
\end{equation}

\noindent
be satisfied by $\varphi(x,t)$ and $X(x,t)$. One way of imposing these constraints is solving the system

\begin{align}
\label{eq:Adaptation Approach 1}
&\delta_1 \tilde S[\varphi,X] \cdot \delta \varphi(x,t) = 0 \quad\quad \text{for all }\delta \varphi(x,t),\\
&g_i\Big(\varphi(x_1,t),...,\varphi(x_N,t),X(x_1,t),...,X(x_N,t)\Big)=0, \quad\quad\quad i=1,...,N. \nonumber
\end{align} 

\noindent
This system consists of one Euler-Lagrange equation that corresponds to extremizing $\tilde S$ with respect to $\varphi$ (we saw in Section~\ref{subsec:reparametrized lagrangian 2} that the other Euler-Lagrange equation is not independent) and a set of constraints enforced at some pre-selected points $x_i$. Note, that upon finite element discretization on a mesh coinciding with the pre-selected points this system reduces to the approach presented in Section~\ref{sec:approach1}: we minimize the discrete action with respect to the $y_i$'s only and supplement the resulting equations with the constraints \eqref{eq:constraints discrete}.

Another way that we want to explore consists in using Lagrange multipliers. Define the auxiliary action functional

\begin{equation}
\label{eq:Auxiliary action}
\tilde S_C[\varphi,X,\lambda_k]=\tilde S[\varphi,X] - \sum_{i=1}^N \int_0^{T_{max}} \lambda_i(t)\cdot g_i\Big(\varphi(x_1,t),...,\varphi(x_N,t),X(x_1,t),...,X(x_N,t)\Big) \, dt.
\end{equation}

\noindent
We will assume that the Lagrange multipliers $\lambda_i(t)$ are at least continuous in time. According to the method of Lagrange multipliers,  we seek the stationary points of $\tilde S_C$. This leads to the following system of equations

\begin{align}
\label{eq:Adaptation Approach 2}
\delta_1 \tilde S[\varphi,X] \cdot \delta \varphi(x,t) - \sum_{i=1}^N \sum_{j=1}^N \int_0^{T_{max}} \lambda_i(t)\, \frac{\partial g_i}{\partial y_j} \,\delta \varphi(x_j,t) \,dt&= 0 \quad\quad\quad \text{for all }\delta \varphi(x,t),\nonumber\\
\delta_2 \tilde S[\varphi,X] \cdot \delta X(x,t) - \sum_{i=1}^N \sum_{j=1}^N \int_0^{T_{max}} \lambda_i(t)\, \frac{\partial g_i}{\partial X_j} \,\delta X(x_j,t) \,dt&= 0 \quad\quad\quad \text{for all }\delta X(x,t),\nonumber \\
g_i\Big(\varphi(x_1,t),...,\varphi(x_N,t),X(x_1,t),...,X(x_N,t)\Big)&=0, \quad\quad\:\: i=1,...,N,
\end{align}

\noindent
where for clarity we suppressed writing the arguments of $\frac{\partial g_i}{\partial y_j}$ and $\frac{\partial g_i}{\partial X_j}$.

Equation \eqref{eq:Adaptation Approach 1} is more intuitive, because we directly use the arbitrariness of $X(x,t)$ and simply restrict it further by imposing constraints. It is not immediately obvious how solutions of \eqref{eq:Adaptation Approach 1} and \eqref{eq:Adaptation Approach 2} relate to each other. We would like both systems to be \textquoteleft equivalent' in some sense, or at least their solution sets to overlap. Let us investigate this issue in more detail.

Suppose $(\varphi,X)$ satisfy \eqref{eq:Adaptation Approach 1}. Then it is quite trivial to see that $(\varphi, X, \lambda_1,...,\lambda_N)$ such that $\lambda_k \equiv 0$ satisfy \eqref{eq:Adaptation Approach 2}: the second equation is implied by the first one and the other equations coincide with those of \eqref{eq:Adaptation Approach 1}. At this point it should be obvious that system \eqref{eq:Adaptation Approach 2} may have more solutions for $\varphi$ and $X$ than system \eqref{eq:Adaptation Approach 1}.

\begin{prop}
\label{Thm: Zero Lagrange multipliers}
The only solutions $(\varphi, X, \lambda_1,...,\lambda_N)$ to \eqref{eq:Adaptation Approach 2} that satisfy \eqref{eq:Adaptation Approach 1} as well are those with $\lambda_k \equiv 0$ for all $k$.
\end{prop}

\begin{proof}
Suppose $(\varphi, X, \lambda_1,...,\lambda_N)$ satisfy both \eqref{eq:Adaptation Approach 1} and \eqref{eq:Adaptation Approach 2}. System \eqref{eq:Adaptation Approach 1} implies that $\delta_1 \tilde S \cdot \delta \varphi=0$ and $\delta_2 \tilde S \cdot \delta X=0$. Using this in system \eqref{eq:Adaptation Approach 2} gives

\begin{align}
\label{eq:Proving lambda=0 - variational equation}
\sum_{j=1}^N \int_0^{T_{max}}dt \, \delta \varphi(x_j,t) \, \sum_{i=1}^N \lambda_i(t)\, \frac{\partial g_i}{\partial y_j}&= 0 \quad\quad\quad \text{for all }\delta \varphi(x,t),\nonumber\\
\sum_{j=1}^N \int_0^{T_{max}}dt \, \delta X(x_j,t) \, \sum_{i=1}^N \lambda_i(t)\, \frac{\partial g_i}{\partial X_j}&= 0 \quad\quad\quad \text{for all }\delta X(x,t).
\end{align}

\noindent
In particular, this has to hold for variations $\delta \varphi$ and $\delta X$ such that $\delta \varphi(x_j,t) = \delta X(x_j,t) = \nu(t) \cdot \delta_{kj}$, where $\nu(t)$ is an arbitrary continuous function of time. If we further assume that for all $x \in [0,X_{max}]$ the functions $\varphi(x,.)$ and $X(x,.)$ are continuous, both $\sum_{i=1}^N \lambda_i(t)\, \frac{\partial g_i}{\partial y_k}$ and $\sum_{i=1}^N \lambda_i(t)\, \frac{\partial g_i}{\partial X_k}$ are continuous and we get

\begin{equation}
\label{eq:Proving lambda=0 - matrix equation}
Dg\Big(\varphi(x_1,t),...,\varphi(x_N,t),X(x_1,t),...,X(x_N,t)\Big)^T \cdot \lambda(t) = 0
\end{equation}

\noindent
for all $t$, where $\lambda = (\lambda_1, ...,\lambda_N)^T$ and the $N \times 2N$ matrix $Dg = \Big[\frac{\partial g_i}{\partial y_k} \, \frac{\partial g_i}{\partial X_k}\Big]_{i,k=1,...,N}$ is the derivative of $g$. Since we assumed that $0$ is a regular value of $g$ and the constraint $g=0$ is satisfied by $\varphi$ and $X$, we have that for all $t$ the matrix $Dg$ has full rank---that is, there exists a nonsingular $N \times N$ submatrix $\Xi$. Then the equation $\Xi^T \lambda(t)=0$ implies $\lambda \equiv 0$.
\end{proof}

We see that considering Lagrange multipliers in \eqref{eq:Auxiliary action} makes sense at the continuous level. We can now perform a finite element discretization. The auxiliary Lagrangian $\tilde L_C:Q \times G \times W \times Z \times \mathbb{R}^N \longrightarrow \mathbb{R}$ corresponding to \eqref{eq:Auxiliary action} can be written as

\begin{equation}
\label{eq:Auxiliary Lagrangian}
\tilde L_C[\varphi,X,\varphi_t,X_t,\lambda_k]=\tilde L[\varphi,X,\varphi_t,X_t] - \sum_{i=1}^N\lambda_i\cdot g_i\Big(\varphi(x_1),...,\varphi(x_N),X(x_1),...,X(x_N)\Big),
\end{equation}

\noindent
where $\tilde L$ is the Lagrangian of the unconstrained theory and has been defined by \eqref{eq:Lagrangian2}. Let us choose a uniform mesh coinciding with the pre-selected points $x_i$. As in Section~\ref{subsec:FEM2}, we consider the restriction $\tilde L_{CN} = \tilde L_C |_{Q_N \times G_N \times W_N \times Z_N \times \mathbb{R}^N}$ and we get

\begin{equation}
\label{eq:LCN}
\tilde L_{CN}(y_i,X_j,\dot y_k,\dot X_l,\lambda_m) = \tilde L_{N}(y_i,X_j,\dot y_k,\dot X_l)-\sum_{i=1}^N \lambda_i\cdot g_i(y_1,...,y_N,X_1,...,X_N).
\end{equation}

\noindent
We see that the semi-discrete Lagrangian $\tilde L_{CN}$ is obtained from the semi-discrete Lagrangian $\tilde L_{N}$ by adding the constraints $g_i$ directly at the semi-discrete level, which is exactly what we set out to do at the beginning of this section. However, in the semi-discrete setting we cannot expect the Lagrange multipliers to vanish for solutions of interest. This is because there is no semi-discrete counterpart of Proposition~\ref{Thm: Zero Lagrange multipliers}. On one hand, the semi-discrete version of \eqref{eq:Adaptation Approach 1} (that is, the approach presented in Section~\ref{sec:approach1}) does not imply that $\delta_2 \tilde S \cdot \delta X =0$, so the above proof will not work. On the other hand, if we supplement \eqref{eq:Adaptation Approach 1} with the equation corresponding to variations of $X$, then the finite element discretization will not have solutions, unless the constraint functions are integrals of motion of the system described by $\tilde L_{N}(y_i,X_j,\dot y_k,\dot X_l)$, which generally is not the case. Nonetheless, it is reasonable to expect that if the continuous system \eqref{eq:Adaptation Approach 1} has a solution, then the Lagrange multipliers of the semi-discrete system \eqref{eq:LCN} should remain small. 

Defining constraints by Equations \eqref{eq:constraints continuous} allowed us to use the same finite element discretization for both $\tilde L$ and the constraints, and to prove some correspondence between the solutions of (\ref{eq:Adaptation Approach 1}) and (\ref{eq:Adaptation Approach 2}). However, constraints \eqref{eq:constraints continuous} are global in the sense that they depend on the values of the fields $\varphi$ and $X$ at different points in space. Moreover, these constraints do not determine unique solutions to (\ref{eq:Adaptation Approach 1}) and (\ref{eq:Adaptation Approach 2}), which is a little cumbersome when discussing multisymplecticity (see Section~\ref{sec:multisymplecticity}).

\subsubsection{Local constraint}
\label{subsubsec:Local constraint}

In Section~\ref{subsec:MMPDEs} we discussed how some adaptation constraints of interest can be derived from certain partial differential equations based on the equidistribution principle, for instance equation \eqref{eq:MMPDE0 - arclength}. We can view these PDEs as local constraints that only depend on pointwise values of the fields $\varphi$, $X$ and their spatial derivatives. Let $G=G(\varphi,X,\varphi_x,X_x,\varphi_{xx},X_{xx},...)$ represent such a local constraint. Then, similarly to \eqref{eq:Adaptation Approach 1}, we can write our control-theoretic strategy from Section~\ref{sec:approach1} as

\begin{align}
\label{eq:Adaptation Approach 1 - local constraint}
&\delta_1 \tilde S[\varphi,X] \cdot \delta \varphi(x,t) = 0 \quad\quad \text{for all }\delta \varphi(x,t),\\
&G(\varphi,X,\varphi_x,X_x,\varphi_{xx},X_{xx},...)=0. \nonumber
\end{align} 

\noindent
Note that higher order derivatives of the fields may require the use of higher degree basis functions than the ones in \eqref{eq:basis functions}, or of finite differences instead. 

The Lagrange multiplier approach consists in defining the auxiliary Lagrangian

\begin{equation}
\label{eq:Auxiliary Lagrangian - approach 2}
\tilde L_C[\varphi,X,\varphi_t,X_t,\lambda]=\tilde L[\varphi,X,\varphi_t,X_t] - \int_0^{X_{max}} \lambda(x)\cdot G(\varphi,X,\varphi_x,X_x,\varphi_{xx},X_{xx},...)\,dx.
\end{equation}

\noindent
Suppose that the pair $(\varphi,X)$ satisfies \eqref{eq:Adaptation Approach 1 - local constraint}. Then, much like in Section~\ref{subsubsec:Global constraint}, one can easily check that the triple $(\varphi,X,\lambda \equiv 0)$ satisfies the Euler-Lagrange equations associated with \eqref{eq:Auxiliary Lagrangian - approach 2}. However, an analog of Proposition~\ref{Thm: Zero Lagrange multipliers} does not seem to be very interesting in this case, therefore we are not proving it here.

Introducing the constraints this way is convenient, because the Lagrangian \eqref{eq:Auxiliary Lagrangian - approach 2} then represents a constrained multisymplectic field theory with a local constraint, which makes the analysis of multisymplecticity easier (see Section~\ref{sec:multisymplecticity}). The disadvantage is that discretization of \eqref{eq:Auxiliary Lagrangian - approach 2} requires mixed methods. We will use the linear finite elements \eqref{eq:basis functions} to discretize $\tilde L[\varphi,X,\varphi_t,X_t]$, but the constraint term will be approximated via finite differences. This way we again obtain the semi-discrete Lagrangian \eqref{eq:LCN}, where $g_i$ represents the discretization of $G$ at the point $x=x_i$.

In summary, the methods presented in Section~\ref{subsubsec:Global constraint} and Section~\ref{subsubsec:Local constraint} both lead to the same semi-discrete Lagrangian, but have different theoretical advantages.

\subsection{DAE formulation of the equations of motion}
\label{subsec:DAE formulations of the equations of motion}

The Lagrangian \eqref{eq:LCN} can be written as

\begin{equation}
\label{eq:LCN2}
\tilde L_{CN}(q,\dot q,\lambda)= \frac{1}{2} \dot q^T \tilde M_N(q) \, \dot q - R_N(q) - \lambda^T g(q),
\end{equation}

\noindent
where

\begin{equation}
\label{eq:RN}
R_N(q) = \sum_{k=0}^N \int_{x_k}^{x_{k+1}} R\Big( \gamma_k,y_k \eta_k(x)+y_{k+1} \eta_{k+1}(x)\Big)\frac{X_{k+1}-X_k}{\Delta x} \, dx.
\end{equation}

\noindent
The Euler-Lagrange equations thus take the form

\begin{align}
\label{eq:E-L equations - index 3}
\dot q &= u, \nonumber \\
\tilde M_N(q) \, \dot u &= f(q,u) - Dg(q)^T \,\lambda, \nonumber \\
g(q) &= 0,
\end{align}

\noindent
where

\begin{equation}
\label{eq:f(q)}
f_k(q,u) = -\frac{\partial R_N}{\partial q_k} + \sum_{i,j=1}^{2N} \Big(\frac{1}{2} \frac{\partial (\tilde M_N)_{ij}}{\partial q_k} - \frac{\partial (\tilde M_N)_{ki}}{\partial q_j} \Big) u_i u_j.
\end{equation}

\noindent
System \eqref{eq:E-L equations - index 3} is to be solved for the unknown functions $q(t)$, $u(t)$ and $\lambda(t)$. This is a DAE system of index 3, since we are lacking a differential equation for $\lambda(t)$ and the constraint equation has to be differentiated three times in order to express $\dot \lambda$ as a function of $q$, $u$ and $\lambda$, provided that certain regularity conditions are satisfied. Let us determine these conditions. Differentiate the constraint equation with respect to time twice to obtain the acceleration level constraint

\begin{equation}
\label{eq:Acceleration level constraint}
Dg(q) \, \dot u = h(q,u),
\end{equation}

\noindent
where

\begin{equation}
\label{eq:h(q,u) term}
h_k(q,u) = -\sum_{i,j=1}^{2N} \frac{\partial^2 g_k}{\partial q_i \partial q_j} u_i u_j.
\end{equation}

\noindent
We can then write \eqref{eq:Acceleration level constraint} and the second equation of \eqref{eq:E-L equations - index 3} together as

\begin{equation}
\label{eq:Solving for the Lagrange multipliers}
\left(\begin{array}{cc}
 \tilde M_N(q) & Dg(q)^T  \\
 Dg(q) & 0 
\end{array} \right)
\left(\begin{array}{c}
 \dot u\\
 \lambda 
\end{array} \right)=
\left(\begin{array}{c}
 f(q,u)\\
 h(q,u) 
\end{array} \right).
\end{equation}

\noindent
If we could solve this equation for $\dot u$ and $\lambda$ in terms of $q$ and $u$, then we could simply differentiate the expression for $\lambda$ one more time to obtain the missing differential equation, thus showing system \eqref{eq:E-L equations - index 3} is of index 3. System \eqref{eq:Solving for the Lagrange multipliers} is solvable if its matrix is invertible. Hence, for system \eqref{eq:E-L equations - index 3} to be of index 3 the following condition

\begin{equation}
\label{eq:Condition for index 3}
\det
\left(\begin{array}{cc}
 \tilde M_N(q) & Dg(q)^T  \\
 Dg(q) & 0 
\end{array} \right)
\not = 0
\end{equation}

\noindent
has to be satisfied for all $q$ or at least in a neighborhood of the points satisfying $g(q)=0$. Note that with suitably chosen constraints this condition allows the mass matrix to be singular.

We would like to perform time integration of this mechanical system using the symplectic (variational) Lobatto IIIA-IIIB quadratures for constrained systems (see \cite{HLWGeometric}, \cite{HWODE2}, \cite{JayLobatto}, \cite{JaySPARK}, \cite{MarsdenWestVarInt}, \cite{LeimkuhlerSkeel1994}, \cite{LeimkuhlerReichBook}, \cite{LeyendeckerMarsden2008}). However, due to the singularity of the Runge-Kutta coefficient matrices $(a_{ij})$ and $(\bar a_{ij})$ for the Lobatto IIIA and IIIB schemes, the assumption \eqref{eq:Condition for index 3} does not guarantee that these quadratures define a unique numerical solution: the mass matrix would need to be invertible. To circumvent this numerical obstacle we resort to a trick described in \cite{JaySPARK}. We embed our mechanical system in a higher dimensional configuration space by adding slack degrees of freedom $r$ and $\dot r$ and form the augmented Lagrangian $\tilde L_{N}^A$ by modifying the kinetic term of $\tilde L_{N}$ to read

\begin{equation}
\label{eq:Augmented Lagrangian Big Mass Matrix}
\tilde L_N^A(q,r,\dot q,\dot r) = \frac{1}{2}
\left(\begin{array}{cc}
\dot q^T & \dot r^T\\ 
\end{array} \right) \cdot
\left(\begin{array}{cc}
 \tilde M_N(q) & Dg(q)^T  \\
 Dg(q) & 0 
\end{array} \right) \cdot
\left(\begin{array}{c}\dot q\\\dot r\end{array} \right)- R_N(q).
\end{equation}

\noindent
Assuming \eqref{eq:Condition for index 3}, the augmented system has a non-singular mass matrix. If we multiply out the terms we obtain simply

\begin{equation}
\label{eq:Augmented Lagrangian General}
\tilde L_N^A(q,r,\dot q,\dot r) = \tilde L_N(q,\dot q)+\dot r^T Dg(q) \,\dot q.
\end{equation}

\noindent
This formula in fact holds for general Lagrangians, not only for \eqref{eq:ParticularLN}. In addition to $g(q)=0$ we further impose the constraint $r=0$. Then the augmented constrained Lagrangian takes the form

\begin{align}
\label{eq:Constrained Augmented Lagrangian General}
\tilde L_{CN}^A(q,r,\dot q,\dot r,\lambda,\mu)= \tilde L_{N}(q,\dot q) +\dot r^T Dg(q) \,\dot q - \lambda^T g(q) - \mu^T r.
\end{align}

\noindent
The corresponding Euler-Lagrange equations are

\begin{align}
\label{eq:E-L equations - augmented}
\dot q &= u, \nonumber \\
\dot r &= w, \nonumber \\
\tilde M_N(q) \, \dot u + Dg(q)^T \, \dot w &= f(q,u) - Dg(q)^T \,\lambda, \nonumber \\
Dg(q) \, \dot u &= h(q,u)-\mu, \nonumber \\
g(q) &= 0, \nonumber \\
r &= 0.
\end{align}

\noindent
It is straightforward to verify that $r(t)=0$, $w(t)=0$, $\mu(t)=0$ is the exact solution and the remaining equations reduce to \eqref{eq:E-L equations - index 3}, that is, the evolution of the augmented system coincides with the evolution of the original system, by construction. The advantage is that the augmented system is now regular and we can readily apply the Lobatto IIIA-IIIB method for constrained systems to compute a numerical solution. It should be intuitively clear that this numerical solution will approximate the solution of \eqref{eq:E-L equations - index 3} as well. What is not immediately obvious is whether a variational integrator based on \eqref{eq:Augmented Lagrangian General} can be interpreted as a variational integrator based on $\tilde L_N$. This can be elegantly justified with the help of exact constrained discrete Lagrangians. Let $\mathcal{N} \subset Q_N \times G_N$ be the constraint submanifold defined by $g(q)=0$. The exact constrained discrete Lagrangian $\tilde L_N^{C,E}:\mathcal{N} \times \mathcal{N} \longrightarrow \mathbb{R}$ is defined by

\begin{equation}
\label{eq:Exact Constrained Discrete Lagrangian}
\tilde L_N^{C,E} \big(q^{(1)},q^{(2)}\big) = \int_0^{\Delta t} \tilde L_N\big(q(t),\dot q(t)\big) \, dt,
\end{equation}

\noindent
where $q(t)$ is the solution to the constrained Euler-Lagrange equations \eqref{eq:E-L equations - index 3} such that it satisfies the boundary conditions $q(0)=q^{(1)}$ and $q(\Delta t)=q^{(2)}$. Note that $\mathcal{N}\times\{0\}\subset (Q_N \times G_N)\times \mathbb{R}^N$ is the constraint submanifold defined by $g(q)=0$ and $r=0$. Since necessarily $r^{(1)}=r^{(2)}=0$, we can define the exact augmented constrained discrete Lagrangian $\tilde L_N^{A,C,E}:\mathcal{N} \times \mathcal{N} \longrightarrow \mathbb{R}$ by

\begin{equation}
\label{eq:Exact Augmented Constrained Discrete Lagrangian}
\tilde L_N^{A,C,E} \big(q^{(1)},q^{(2)}\big) = \int_0^{\Delta t} \tilde L_N^A\big(q(t),r(t),\dot q(t),\dot r(t)\big) \, dt,
\end{equation}

\noindent
where $q(t)$, $r(t)$ are the solutions to the augmented constrained Euler-Lagrange equations \eqref{eq:E-L equations - augmented} such that the boundary conditions $q(0)=q^{(1)}$, $q(\Delta t)=q^{(2)}$ and $r(0)=r(\Delta t)=0$ are satisfied.

\begin{prop}
The exact discrete Lagrangians $\tilde L_N^{A,C,E}$ and $\tilde L_N^{C,E}$ are equal.
\end{prop}

\begin{proof}
Let $q(t)$ and $r(t)$ be the solutions to \eqref{eq:E-L equations - augmented} such that the boundary conditions $q(0)=q^{(1)}$, $q(\Delta t)=q^{(2)}$ and $r(0)=r(\Delta t)=0$ are satisfied. As argued before, we in fact have $r(t)=0$ and $q(t)$ satisfies \eqref{eq:E-L equations - index 3} as well. By \eqref{eq:Augmented Lagrangian General} we have

\begin{equation*}
\tilde L_N^A\big(q(t),r(t),\dot q(t),\dot r(t)\big)=\tilde L_N\big(q(t),\dot q(t)\big)
\end{equation*}

\noindent
for all $t\in[0,\Delta t]$, and consequently $\tilde L_N^{A,C,E}=\tilde L_N^{C,E}$.\\
\end{proof}

\noindent
This means that any discrete Lagrangian $\tilde L_d:(Q_N \times G_N)\times \mathbb{R}^N \times (Q_N \times G_N)\times \mathbb{R}^N\longrightarrow\mathbb{R}$ that approximates $\tilde L_N^{A,C,E}$ to order $s$ also approximates $\tilde L_N^{C,E}$ to the same order, that is, a variational integrator for \eqref{eq:E-L equations - augmented}, in particular our Lobatto IIIA-IIIB scheme, is also a variational integrator for \eqref{eq:E-L equations - index 3}.

\paragraph{Backward error analysis.} The advantage of the Lagrange multiplier approach is the fact that upon spatial discretization we deal with a constrained mechanical system. Backward error analysis of symplectic/variational numerical schemes for such systems shows that the modified equations also describe a constrained mechanical system for a nearby Hamiltonian (see Theorem~5.6 in Section~IX.5.2 of \cite{HLWGeometric}). Therefore, we expect the Lagrange multiplier strategy to demonstrate better performance in terms of energy conservation than the control-theoretic strategy. The Lagrange multiplier approach makes better use of the geometry underlying the field theory we consider, the key idea being to \emph{treat the reparametrization field $X(x,t)$ as an additional dynamical degree of freedom on equal footing with $\varphi(x,t)$}.

 \section{Multisymplectic field theory formalism}
\label{sec:multisymplecticity}

In Section~\ref{sec:approach1} and Section~\ref{sec:approach2} we took the view of infinite dimensional manifolds of fields as configuration spaces and presented a way to construct space-adaptive variational integrators in that formalism. We essentially applied symplectic integrators to semi-discretized Lagrangian field theories. In this section we show how $r$-adaptive integrators can be described in the more general framework of multisymplectic geometry. In particular we show that some of the integrators obtained in the previous sections can be interpreted as multisymplectic variational integrators. Multisymplectic geometry provides a covariant formalism for the study of field theories in which time and space are treated on equal footing, as a conseqence of which multisymplectic variational integrators allow for more general discretizations of spacetime, such that, for instance, each element of space may be integrated with a different timestep (see \cite{LewAVI}). For the convenience of the reader, below we briefly review some background material and provide relevant references for further details. We then proceed to reformulate our adaptation strategies in the language of multisymplectic field theory.

\subsection{Background material}
\label{subsec: multisymp background}

\subsubsection*{Lagrangian mechanics and Veselov-type discretizations}

Let $Q$ be the configuration manifold of a certain mechanical system and $TQ$ its tangent bundle. Denote the coordinates on $Q$ by $q^i$, and on $TQ$ by $(q^i, \dot q^i)$, where $i=1,2,...,n$. The system is described by defining the Lagrangian $L:TQ \longrightarrow \mathbb{R}$ and the corresponding action functional $S[q(t)] = \int_a^b L\big(q^i(t),\dot q^i(t)\big)\, dt$. The dynamics is obtained through Hamilton's principle, which seeks the curves $q(t)$ for which the functional $S[q(t)]$ is stationary under variations of $q(t)$ with fixed endpoints, i.e. we seek $q(t)$ such that

\begin{align}
\label{eq:Hamilton's principle}
dS[q(t)] \cdot \delta q(t)=\frac{d}{d\epsilon} \bigg|_{\epsilon=0}S[q_\epsilon(t)]=0
\end{align}

\noindent
for all $\delta q(t)$ with $\delta q(a)=\delta q(b)=0$, where $q_\epsilon (t)$ is a smooth family of curves satisfying $q_0=q$ and $\frac{d}{d\epsilon} \big|_{\epsilon=0} q_\epsilon = \delta q$. By using integration by parts, the Euler-Lagrange equations follow as

\begin{align}
\label{eq:Euler-Lagrange Equations}
\frac{\partial L}{\partial q^i} -\frac{d}{dt} \frac{\partial L}{\partial \dot q^i}=0.
\end{align}

\noindent
The canonical symplectic form $\Omega$ on $T^*Q$, the $2n$-dimensional cotangent bundle of $Q$, is given by $\Omega = dq^i \wedge dp_i$, where summation over $i$ is implied and $(q^i,p_i)$ are the canonical coordinates on $T^*Q$. The Lagrangian defines the Legendre transformation $\mathbb{F}L: TQ \longrightarrow T^*Q$, which in coordinates is given by $(q^i,p_i)=(q^i,\frac{\partial L}{\partial \dot q^i})$. We then define the Lagrange 2-form on $TQ$ by pulling back the canonical symplectic form, i.e. $\Omega_L = \mathbb{F}L^* \Omega$. If the Legendre transformation is a local diffeomorphism, then $\Omega_L$ is a symplectic form. The Lagrange vector field is a vector field $X_E$ on $TQ$ that satisfies $X_E \lrcorner \Omega_L=dE$, where the energy $E$ is defined by $E(v_q)=\mathbb{F}L(v_q)\cdot v_q - L(v_q)$ and $\lrcorner$ denotes the interior product, i.e. the contraction of a differential form with a vector field. It can be shown that the flow $F_t$ of this vector field preserves the symplectic form, that is, $F_t^* \Omega_L = \Omega_L$. The flow $F_t$ is obtained by solving the Euler-Lagrange equations \eqref{eq:Euler-Lagrange Equations}.

For a Veselov-type discretization we essentially replace $TQ$ with $Q \times Q$, which serves as a discrete approximation of the tangent bundle. We define a discrete Lagrangian $L_d$ as a smooth map $L_d: Q \times Q \longrightarrow \mathbb{R}$ and the corresponding discrete action $S=\sum_{k=0}^{N-1} L_d(q_k,q_{k+1})$. The variational principle now seeks a sequence $q_0$, $q_1$, $...$, $q_N$ that extremizes $S$ for variations holding the endpoints $q_0$ and $q_N$ fixed. The Discrete Euler-Lagrange equations follow

\begin{equation}
\label{eq:Discrete Euler-Lagrange equations}
D_2L_d(q_{k-1},q_k) + D_1 L_d(q_k,q_{k+1}) = 0.
\end{equation}

\noindent
This implicitly defines a discrete flow $F: Q \times Q \longrightarrow Q \times Q$ such that $F(q_{k-1},q_k)=(q_k,q_{k+1})$. One can define the discrete Lagrange 2-form on $Q \times Q$ by $\omega_L = \frac{\partial^2 L_d}{\partial q_0^i \partial q_1^j} dq_0^i \wedge dq_1^j$, where $(q_0^i,q_1^j)$ denotes the coordinates on $Q \times Q$. It then follows that the discrete flow $F$ is symplectic, i.e. $F^* \omega_L = \omega_L$.

Given a continuous Lagrangian system with $L:TQ \longrightarrow \mathbb{R}$ one chooses a corresponding discrete Lagrangian as an approximation $L_d(q_k,q_{k+1}) \approx \int_{t_k}^{t_{k+1}} L\big(q(t),\dot q(t)\big)\,dt$, where $q(t)$ is the solution of the Euler-Lagrange equations corresponding to $L$ with the boundary values $q(t_k)=q_k$ and $q(t_{k+1})=q_{k+1}$. 

For more details regarding Lagrangian mechanics, variational principles, and symplectic geometry, see \cite{MarsdenRatiuSymmetry}. Discrete Mechanics and variational integrators are discussed in \cite{MarsdenWestVarInt}.

\subsubsection*{Multisymplectic geometry and Lagrangian field theory}

Let $\mathcal{X}$ be an oriented manifold representing the $(n+1)$-dimensional spacetime with local coordinates $(x^0,x^1,\ldots,x^n)\equiv (t,x)$, where $x^0 \equiv t$ is time and $(x^1,\ldots,x^n)\equiv x$ are space coordinates. Physical fields are sections of a configuration fiber bundle $\pi_{\mathcal{X}Y}:Y\longrightarrow \mathcal{X}$, that is, continuous maps $\phi:\mathcal{X}\longrightarrow Y$ such that $\pi_{\mathcal{X}Y}\circ \phi  = \mathrm{id}_\mathcal{X}$. This means that for every $(t,x) \in \mathcal{X}$, $\phi(t,x)$ is in the fiber over $(t,x)$, which is $Y_{(t,x)}=\pi_{\mathcal{X}Y}^{-1}((t,x))$. The evolution of the field takes place on the first jet bundle $J^1 Y$, which is the analog of $TQ$ for mechanical systems. $J^1 Y$ is defined as the affine bundle over $Y$ such that for $y\in Y_{(t,x)}$ the fiber $J^1_y Y$ consists of linear maps $\vartheta: T_{(t,x)} \mathcal{X}\rightarrow T_y Y$ satisfying the condition $T\pi_{\mathcal{X}Y}\circ \vartheta = \mathrm{id}_{T_{(t,x)}\mathcal{X}}$. The local coordinates $(x^\mu,y^a)$ on $Y$ induce the coordinates $(x^\mu,y^a, v^a_{\phantom{a}\mu})$ on $J^1 Y$. Intuitively, the first jet bundle consists of the configuration bundle $Y$, and of the first partial derivatives of the field variables with respect to the independent variables. Let $\phi(x^0,\ldots,x^n)=(x^0,\ldots,x^n, y^1,\ldots,y^m)$ in coordinates and let $v^a_{\phantom{a}\mu}=y^a_{\phantom{a},\mu}=\partial y^a / \partial x^\mu$ denote the partial derivatives. We can think of $J^1 Y$ as a fiber bundle over $\mathcal{X}$. Given a section $\phi:\mathcal{X}\longrightarrow Y$, we can define its first jet prolongation $j^1 \phi: \mathcal{X}\longrightarrow J^1 Y$, in coordinates given by $j^1\phi(x^0,x^1,\ldots,x^n)=(x^0,x^1,\ldots,x^n,y^1,\ldots,y^m,y^1_{\phantom{1},0},\ldots,y^m_{\phantom{m},n})$, which is a section of the fiber bundle $J^1 Y$ over $\mathcal{X}$. For higher order field theories we consider higher order jet bundles, defined iteratively by $J^2 Y = J^1(J^1 Y)$ and so on. The local coordinates on $J^2 Y$ are denoted $(x^\mu,y^a, v^a_{\phantom{a}\mu},w^a_{\phantom{a}\mu},\kappa^a_{\mu\nu})$. The second jet prolongation $j^2\phi:\mathcal{X}\longrightarrow J^2 Y$ is given in coordinates by $j^2 \phi(x^\mu)=(x^\mu,y^a,y^a_{\phantom{a},\mu},y^a_{\phantom{a},\mu},y^a_{\phantom{a},\mu,\nu})$.

Lagrangian density for first order field theories is defined as a map $\mathcal{L}:J^1 Y \longrightarrow \mathbb{R}$. The corresponding action functional is $S[\phi]=\int_\mathcal{U} \mathcal{L}(j^1 \phi)\,d^{n+1}x$, where $\mathcal{U}\subset\mathcal{X}$. Hamilton's principle seeks fields $\phi(t,x)$ that extremize $S$, that is

\begin{equation}
\label{eq:HamiltonPrincipleForFields}
\frac{d}{d\lambda} \bigg|_{\lambda=0} S[\eta_Y^\lambda \circ \phi]=0
\end{equation}

\noindent
for all $\eta_Y^\lambda$ that keep the boundary conditions on $\partial \mathcal{U}$ fixed, where $\eta_Y^\lambda:Y\longrightarrow Y$ is the flow of a vertical vector field $V$ on $Y$. This leads to the Euler-Lagrange equations

\begin{equation}
\label{eq:Euler-Lagrange Equations for Fields}
\frac{\partial\mathcal{L}}{\partial y^a}(j^1 \phi) - \frac{\partial}{\partial x^\mu} \bigg(\frac{\partial \mathcal{L}}{\partial v^a_{\phantom{a}\mu}}(j^1 \phi)\bigg)=0.
\end{equation}

\noindent
Given the Lagrangian density $\mathcal{L}$ one can define the Cartan $(n+1)$-form $\Theta_\mathcal{L}$ on $J^1 Y$, in local coordinates given by $\Theta_\mathcal{L} = \frac{\partial \mathcal{L}}{\partial v^a_{\phantom{a}\mu}}dy^a \wedge d^n x_\mu + (\mathcal{L}-\frac{\partial \mathcal{L}}{\partial v^a_{\phantom{a}\mu}} v^a_{\phantom{a}\mu}) d^{n+1}x$, where $d^n x_\mu = \partial_\mu \,\lrcorner\, d^{n+1}x$. The multisymplectic $(n+2)$-form is then defined by $\Omega_{\mathcal{L}}=-d\Theta_\mathcal{L}$. Let $\mathcal{P}$ be the set of solutions of the Euler-Lagrange equations, that is, the set of sections $\phi$ satisfying \eqref{eq:HamiltonPrincipleForFields} or \eqref{eq:Euler-Lagrange Equations for Fields}. For a given $\phi \in \mathcal{P}$, let $\mathcal{F}$ be the set of first variations, that is, the set of vector fields $V$ on $J^1 Y$ such that $(t,x)\rightarrow \eta^\epsilon_Y\circ \phi(t,x)$ is also a solution, where $\eta^\epsilon_Y$ is the flow of $V$. The multisymplectic form formula states that if $\phi \in \mathcal{P}$ then for all $V$ and $W$ in $\mathcal{F}$, 

\begin{equation}
\label{eq:MultisymplecticFormFormula}
\int_{\partial \mathcal{U}} (j^1 \phi)^* \big(j^1 V \, \lrcorner \, j^1 W \, \lrcorner \, \Omega_{\mathcal{L}}\big) = 0,
\end{equation}

\noindent
where $j^1V$ is the jet prolongation of $V$, that is, the vector field on $J^1 Y$ in local coordinates given by $j^1 V = (V^\mu,V^a,\frac{\partial V^a}{\partial x^\mu}+\frac{\partial V^a}{\partial y^b} v^b_{\phantom{a}\mu} - v^a_{\phantom{a}\nu} \frac{\partial V^\nu}{\partial x^\mu})$, where $V=(V^\mu,V^a)$ in local coordinates. The multisymplectic form formula is the multisymplectic counterpart of the fact that in finite-dimensional mechanics, the flow of a mechanical system consists of symplectic maps.

For a $k^{\textrm{th}}$-order Lagrangian field theory with the Lagrangian density $\mathcal{L}:J^kY\longrightarrow \mathbb{R}$, analogous geometric structures are defined on $J^{2k-1}Y$. In particular, for a second-order field theory the multisymplectic $(n+2)$-form $\Omega_\mathcal{L}$ is defined on $J^3 Y$ and a similar multisymplectic form formula can be proven. If the Lagrangian density does not depend on the second order time derivatives of the field, it is convenient to define the subbundle $J^2_0 Y \subset J^2 Y$ such that $J^2_0 Y = \{\vartheta \in J^2 Y \, | \, \kappa^a_{00}=0 \}$.

For more information about the geometry of jet bundles, see \cite{Saunders}. The multisymplectic formalism in field theory is discussed in \cite{GIMMSY}. The multisymplectic form formula for first-order field theories is derived in \cite{MarsdenPatrickShkoller}, and generalized for second-order field theories in \cite{KouranbaevaShkoller}. Higher order field theory is considered in \cite{GotayHighOrderJet}.

\subsubsection*{Multisymplectic variational integrators}
Veselov-type discretization can be generalized to multisymplectic field theory. We take $\mathcal{X}=\mathbb{Z}\times\mathbb{Z} = \{(j,i)\}$, where for simplicity we consider $\dim \mathcal{X} =2$, i.e. $n=1$. The configuration fiber bundle is $Y=\mathcal{X} \times \mathscr{F}$ for some smooth manifold $\mathscr{F}$. The fiber over $(j,i)\in\mathcal{X}$ is denoted $Y_{ji}$ and its elements $y_{ji}$. A rectangle $\square$ of $\mathcal{X}$ is an ordered 4-tuple of the form $\square=\big( (j,i),(j,i+1),(j+1,i+1),(j+1,i) \big) = (\square^1,\square^2,\square^3,\square^4)$. The set of all rectangles in $\mathcal{X}$ is denoted $\mathcal{X}^\square$. A point $(j,i)$ is touched by a rectangle if it is a vertex of that rectangle. Let $\mathcal{U} \subset \mathcal{X}$. Then $(j,i) \in \mathcal{U}$ is an interior point of $\mathcal{U}$ if $\mathcal{U}$ contains all four rectangles that touch $(j,i)$. The interior $\textrm{int}\,\mathcal{U}$ is the set of all interior points of $\mathcal{U}$. The closure $\textrm{cl}\,\mathcal{U}$ is the union of all rectangles touching interior points of $\mathcal{U}$. The boundary of $\mathcal{U}$ is defined by $\partial\mathcal{U}=(\mathcal{U} \cap \textrm{cl}\,\mathcal{U})\backslash \textrm{int}\,\mathcal{U}$. A section of $Y$ is a map $\phi: \mathcal{U}\subset \mathcal{X} \rightarrow Y$ such that $\phi(j,i) \in Y_{ji}$. We can now define the discrete first jet bundle of $Y$ as

\begin{align}
\label{eq:DiscreteFirstJetBundle}
J^1Y &= \big\{(y_{ji},y_{j\,i+1},y_{j+1\,i+1},y_{j+1\,i})\,\, \big| \,\, (j,i) \in \mathcal{X},\, y_{ji},y_{j\,i+1},y_{j+1\,i+1},y_{j+1\,i} \in \mathscr{F} \big\} \nonumber \\ 
     &=\mathcal{X}^\square \times \mathscr{F}^4.
\end{align}

\noindent
Intuitively, the discrete first jet bundle is the set of all rectangles together with four values assigned to their vertices. Those four values are enough to approximate the first derivatives of a smooth section with respect to time and space using, for instance, finite differences. The first jet prolongation of a section $\phi$ of $Y$ is the map $j^1 \phi: \mathcal{X}^\square \rightarrow J^1 Y$ defined by $j^1 \phi (\square) = (\square,\phi(\square^1),\phi(\square^2),\phi(\square^3),\phi(\square^4))$. For a vector field $V$ on $Y$, let $V_{ji}$ be its restriction to $Y_{ji}$. Define a discrete Lagrangian $L:J^1Y\rightarrow \mathbb{R}$, $L=L(y_1,y_2,y_3,y_4)$, where for convenience we omit writing the base rectangle. The associated discrete action is given by 

\begin{equation*}
S[\phi]=\sum_{\square \subset \mathcal{U}} L \circ j^1 \phi(\square). 
\end{equation*}

\noindent
The discrete variational principle seeks sections that extremize the discrete action, that is, mappings $\phi(j,i)$ such that

\begin{equation}
\label{eq:DiscreteVariationalPrincipleForFields}
\frac{d}{d \lambda} \bigg|_{\lambda=0} S[\phi_\lambda]=0
\end{equation}

\noindent
for all vector fields $V$ on $Y$ that keep the boundary conditions on $\partial \mathcal{U}$ fixed, where $\phi_\lambda(j,i)=F^{V_{ji}}_\lambda (\phi(j,i))$ and $F^{V_{ji}}_\lambda$ is the flow of $V_{ji}$ on $\mathscr{F}$. This is equivalent to the discrete Euler-Lagrange equations

\begin{align}
\label{eq:DELF}
\frac{\partial L}{\partial y_1} (y_{ji},y_{j\,i+1}&,y_{j+1\,i+1},y_{j+1\,i}) + \frac{\partial L}{\partial y_2} (y_{j\,i-1},y_{ji},y_{j+1\,i},y_{j+1\,i-1}) + \nonumber\\
&+ \frac{\partial L}{\partial y_3} (y_{j-1\,i-1},y_{j-1\,i},y_{ji},y_{j\,i-1}) + \frac{\partial L}{\partial y_4} (y_{j-1\,i},y_{j-1\,i+1},y_{j\,i+1},y_{ji}) = 0
\end{align}

\noindent
for all $(j,i)\in\textrm{int}\,\mathcal{U}$, where we adopt the convention $\phi(j,i)=y_{ji}$. In analogy to the Veselov discretization of mechanics, we can define four 2-forms $\Omega^l_L$ on $J^1Y$, where $l=1,2,3,4$ and $\Omega^1_L+\Omega^2_L+\Omega^3_L+\Omega^4_L=0$, that is, only three 2-forms of these forms are independent. The 4-tuple $(\Omega^1_L,\Omega^2_L,\Omega^3_L,\Omega^4_L)$ is the discrete analog of the multisymplectic form $\Omega_\mathcal{L}$. We refer the reader to the literature for details, e.g. \cite{MarsdenPatrickShkoller}. By analogy to the continuous case, let $\mathcal{P}$ be the set of solutions of the discrete Euler-Lagrange equations \eqref{eq:DELF}. For a given $\phi \in \mathcal{P}$, let $\mathcal{F}$ be the set of first variations, that is, the set of vector fields $V$ on $J^1 Y$ defined similarly as in the continuous case. The discrete multisymplectic form formula then states that if $\phi \in \mathcal{P}$ then for all $V$ and $W$ in $\mathcal{F}$,

\begin{equation}
\label{eq:Discrete Multisymplectic Form Formula}
\sum_{\substack{\square \\[0.7mm] \scriptscriptstyle \square \cap \mathcal{U} \not = \emptyset}} \Bigg( \sum_{\substack{l\\[0.3mm] \scriptscriptstyle \square^l \in \partial \mathcal{U}}} \Big[ (j^1 \phi)^* (j^1 V \, \lrcorner \, j^1 W \, \lrcorner \, \Omega^l_L)\Big] (\square) \Bigg)=0,
\end{equation}

\noindent
where the jet prolongations are defined to be 

\begin{equation}
j^1 V (y_{\square^1},y_{\square^2},y_{\square^3},y_{\square^4}) = \big( V_{\square^1}(y_{\square^1}),V_{\square^2}(y_{\square^2}),V_{\square^3}(y_{\square^3}),V_{\square^4}(y_{\square^4}) \big).
\end{equation}

\noindent
The discrete form formula \eqref{eq:Discrete Multisymplectic Form Formula} is in direct analogy to the multisymplectic form formula \eqref{eq:MultisymplecticFormFormula} that holds in the continuous case.

Given a continuous Lagrangian density $\mathcal{L}$ one chooses a corresponding discrete Lagrangian as an approximation $L(y_{\square^1},y_{\square^2},y_{\square^3},y_{\square^4}) \approx \int_{\overline{\square}} \mathcal{L}\circ j^1 \bar \phi \,dx \,dt$, where $\overline{\square}$ is the rectangular region of the continuous spacetime that contains $\square$ and $\bar \phi (t,x)$ is the solution of the Euler-Lagrange equations corresponding to $\mathcal{L}$ with the boundary values at the vertices of $\square$ corresponding to $y_{\square^1}$, $y_{\square^2}$, $y_{\square^3}$, and $y_{\square^4}$. 

The discrete second jet bundle $J^2 Y$ can be defined by considering ordered 9-tuples 

\begin{align}
\label{eq:9-tuple}
\boxplus=\big( (j-1,i-1),(j-1,i),\, &(j-1,i+1),(j,i-1),\nonumber\\
               &(j,i),(j,i+1),(j+1,i-1),(j+1,i),(j+1,i+1) \big)\nonumber\\
               =(\boxplus^1,\boxplus^2,\boxplus^3,\boxplus^4,\boxplus^5,\boxplus^6,&\boxplus^7,\boxplus^8,\boxplus^9)
\end{align}

\noindent
instead of rectangles $\square$, and the discrete subbundle $J^2_0 Y$ can be defined by considering 6-tuples

\begin{align}
\label{eq:6-tuple}
\boxbar&=\big( (j,i-1),(j,i),(j,i+1),(j+1,i+1),(j+1,i),(j+1,i-1) \big)\nonumber\\
       &=(\boxbar^1,\boxbar^2,\boxbar^3,\boxbar^4,\boxbar^5,\boxbar^6).
\end{align}

\noindent
Similar constructions then follow and a similar discrete multisymplectic form formula can be derived for a second order field theory.

Multisymplectic variational integrators for first order field theories are introduced in \cite{MarsdenPatrickShkoller}, and generalized for second-order field theories in \cite{KouranbaevaShkoller}.

\subsection{Analysis of the control-theoretic approach}
\label{subsec: multisymp approach1}

\subsubsection*{Continuous setting}

We now discuss a multisymplectic setting for the approach presented in Section~\ref{sec:approach1}. Let the computational spacetime be $\mathcal{X}=\mathbb{R} \times \mathbb{R}$ with coordinates $(t,x)$ and consider the trivial configuration bundle $Y=\mathcal{X}\times \mathbb{R}$ with coordinates $(t,x,y)$. Let $\mathcal{U}=[0,T_{max}]\times[0,X_{max}]$ and let our scalar field be represented by a section $\tilde \varphi: \mathcal{U} \longrightarrow Y$ with the coordinate representation $\tilde \varphi(t,x)=(t,x,\varphi(t,x))$. Let $(t,x,y,v_t,v_x)$ denote local coordinates on $J^1 Y$. In these coordinates the first jet prolongation of $\tilde \varphi$ is represented by $j^1 \tilde\varphi(t,x) = (t,x,\varphi(t,x), \varphi_t(t,x), \varphi_x(t,x))$. Then the Lagrangian density \eqref{eq:LagrangianDensity} can be viewed as a mapping $\tilde{\mathcal{L}}: J^1Y \longrightarrow \mathbb{R}$. The corresponding action \eqref{eq:action2} can now be expressed as

\begin{equation}
\label{eq:ActionMultisymplectic}
\tilde S [\tilde \varphi] = \int_\mathcal{U} \tilde{\mathcal{L}} \big(j^1 \tilde\varphi\big)\,dt \wedge dx,
\end{equation}

\noindent
Just like in Section~\ref{sec:approach1}, let us for the moment assume that the function $X: \mathcal{U}\longrightarrow[0,X_{max}]$ is known, so that we can view $\tilde{\mathcal{L}}$ as being time and space dependent. The dynamics is obtained by extremizing $\tilde S$ with respect to $\tilde \varphi$, that is, by solving for $\tilde \varphi$ such that

\begin{equation}
\label{eq:Multisymplectic Extremization of S tilde}
\frac{d}{d\lambda} \bigg|_{\lambda=0} \tilde S[\eta_Y^\lambda \circ \tilde \varphi]=0
\end{equation}

\noindent
for all $\eta_Y^\lambda$ that keep the boundary conditions on $\partial \mathcal{U}$ fixed, where $\eta_Y^\lambda:Y\longrightarrow Y$ is the flow of a vertical vector field $V$ on $Y$. Therefore, for an \emph{a priori} known $X(t,x)$ the multisymplectic form formula \eqref{eq:MultisymplecticFormFormula} is satisfied for solutions of \eqref{eq:Multisymplectic Extremization of S tilde}.

Consider the additional bundle $\pi_{\mathcal{X}\mathcal{B}}:\mathcal{B}=\mathcal{X}\times[0,X_{max}] \longrightarrow \mathcal{X}$ whose sections $\tilde X:\mathcal{U}\longrightarrow \mathcal{B}$ represent our diffeomorphisms. Let $\tilde X(t,x) = (t,x, X(t,x))$ denote a local coordinate representation and assume $X(t,.)$ is a diffeomorphism. Then define $\tilde Y = Y \oplus \mathcal{B}$. We have $J^k \tilde Y \cong J^k Y\oplus J^k \mathcal{B}$. In Section~\ref{subsubsec:Local constraint} we argued that the moving mesh partial differential equation \eqref{eq:MMPDE0} can be interpreted as a local constraint on the fields $\tilde \varphi$, $\tilde X$ and their spatial derivatives. This constraint can be represented by a function $G:J^k \tilde Y \longrightarrow \mathbb{R}$. Sections $\tilde \varphi$ and $\tilde X$ satisfy the constraint if $G(j^k \tilde \varphi, j^k \tilde X)=0$. Therefore our control-theoretic strategy expressed in equations \eqref{eq:Adaptation Approach 1 - local constraint} can be rewritten as

\begin{align}
\label{eq:Multisymplectic control-theoretic strategy}
\frac{d}{d\lambda} \bigg|_{\lambda=0} \tilde S[\eta_Y^\lambda \circ \tilde \varphi]&=0, \nonumber \\
G(j^k \tilde \varphi, j^k \tilde X)&=0,
\end{align}

\noindent
for all $\eta_Y^\lambda$, similarly as above. Let us argue how to interpret the notion of multisymplecticity for this problem. Intuitively, multisymplecticity should be understood in a sense similar to Proposition~\ref{Thm: DAE vs Hamiltonian ODE}. We first solve the problem \eqref{eq:Multisymplectic control-theoretic strategy} for $\tilde \varphi$ and $\tilde X$, given some initial and boundary conditions. Then we substitute this $\tilde X$ into the problem \eqref{eq:Multisymplectic Extremization of S tilde}. Let $\mathcal{P}$ be the set of solutions to this problem. Naturally, $\tilde \varphi \in \mathcal{P}$. The multisymplectic form formula \eqref{eq:MultisymplecticFormFormula} will be satisfied for all fields in $\mathcal{P}$, but the constraint $G=0$ will be satisfied only for $\tilde \varphi$.

\subsubsection*{Discretization}

Discretize the computational spacetime $\mathbb{R} \times \mathbb{R}$ by picking the discrete set of points $t_j = j \cdot \Delta t$, $x_i = i \cdot \Delta x$, and define $\mathcal{X} = \{(j,i)\, |\, j,i \in \mathbb{Z}\}$. Let $\mathcal{X}^\square$ and $\mathcal{X}^\boxbar$ be the set of rectangles and 6-tuples in $\mathcal{X}$, respectively. The discrete configuration bundle is $Y=\mathcal{X} \times \mathbb{R}$ and for convenience of notation let the elements of the fiber $Y_{ji}$ be denoted by $y^j_i$. Let $\mathcal{U}=\{(j,i)\, |\, j=0,1,\ldots,M+1, \, i=0,1,\ldots,N+1\}$, where $\Delta x = X_{max}/(N+1)$ and $\Delta t = T_{max}/(M+1)$. Suppose we have a discrete Lagrangian $\tilde L: J^1 Y \longrightarrow \mathbb{R}$ and the corresponding discrete action $\tilde S$ that approximates \eqref{eq:ActionMultisymplectic}, where we assume that $X(t,x)$ is known and of the form \eqref{eq:X FEM}. A variational integrator is obtained by solving

\begin{equation}
\label{eq:DiscreteVariationalPrincipleForApproach1}
\frac{d}{d \lambda} \bigg|_{\lambda=0} \tilde S[\tilde \varphi_\lambda]=0
\end{equation}

\noindent
for a discrete section $\tilde \varphi : \mathcal{U} \longrightarrow Y$, as described in Section~\ref{subsec: multisymp background}. This integrator is multisymplectic, i.e. the discrete multisymplectic form formula \eqref{eq:Discrete Multisymplectic Form Formula} is satisfied.

\paragraph{Example: Midpoint rule.} In \eqref{eq:PRK for ODE} consider the 1-stage symplectic partitioned Runge-Kutta method with the coefficients $a_{11}=\bar a_{11}=c_1=1/2$ and $b_1=\bar b_1 = 1$. This method is often called the \emph{midpoint rule} and is a 2-nd order member of the Gauss family of quadratures. It can be easily shown (see \cite{HLWGeometric}, \cite{MarsdenWestVarInt}) that the discrete Lagrangian \eqref{eq:DiscreteLagrangianRK} for this method is given by

\begin{equation}
\label{eq: Discrete Lagrangian for the Midpoint Rule - definition}
\tilde{L}_d(t_j,y^j,t_{j+1},y^{j+1})=\Delta t \cdot \tilde{L}_N \bigg( \frac{y^j+y^{j+1}}{2},\frac{y^{j+1}-y^j}{\Delta t},t_j+\frac{1}{2}\Delta t \bigg),
\end{equation} 

\noindent
where $\Delta t = t_{j+1}-t_{j}$ and $y^j=(y^j_1,\ldots,y^j_N)$. Using \eqref{eq:Lagrangian} and \eqref{eq:LN FEM} we can write

\begin{equation}
\label{eq: Discrete Lagrangian for the Midpoint Rule in terms of L tilde}
\tilde{L}_d(t_j,y^j,t_{j+1},y^{j+1}) = \sum_{i=0}^N \tilde{L} \big( y^{j}_{i},y^{j}_{i+1},y^{j+1}_{i+1},y^{j+1}_{i} \big),
\end{equation}

\noindent
where we defined the discrete Lagrangian $\tilde L: J^1 Y \longrightarrow \mathbb{R}$ by the formula

\begin{equation}
\label{eq: Discrete Multisymplectic Lagrangian for the Midpoint Rule}
\tilde{L}(y^{j}_{i},y^{j}_{i+1},y^{j+1}_{i+1},y^{j+1}_{i})=\Delta t \int_{x_i}^{x_{i+1}} \tilde{\mathcal{L}}\bigg(\bar \varphi(x),\bar \varphi_x(x),\bar \varphi_t(x), x,t_j+\frac{1}{2}\Delta t  \bigg)\,dx
\end{equation}

\noindent
with

\begin{align}
\label{eq: Auxiliary terms for Discrete Lagrangian for the Midpoint Rule in terms of L tilde }
\bar \varphi(x) &= \frac{y^j_i+y^{j+1}_i}{2} \eta_i(x)+\frac{y^j_{i+1}+y^{j+1}_{i+1}}{2} \eta_{i+1}(x), \nonumber \\
\bar \varphi_x(x) &= \frac{1}{2}\frac{y^{j}_{i+1}-y^{j}_{i}}{\Delta x} + \frac{1}{2}\frac{y^{j+1}_{i+1}-y^{j+1}_{i}}{\Delta x}, \nonumber \\
\bar \varphi_t(x) &= \frac{y^{j+1}_i-y^j_i}{\Delta t} \eta_i(x)+\frac{y^{j+1}_{i+1}-y^j_{i+1}}{\Delta t} \eta_{i+1}(x).
\end{align}

\noindent
Given the Lagrangian density $\tilde{\mathcal{L}}$ as in \eqref{eq:LagrangianDensity}, and assuming $X(t,x)$ is known, one can evaluate the integral in \eqref{eq: Discrete Multisymplectic Lagrangian for the Midpoint Rule} explicitly. It is now a straightforward calculation to show that the discrete variational principle \eqref{eq:DiscreteVariationalPrincipleForApproach1} for the discrete Lagrangian $\tilde{L}$ as defined is equivalent to the Discrete Euler-Lagrange equations \eqref{eq:Discrete Euler-Lagrange equations} for $\tilde{L}_d$, and consequently to \eqref{eq:PRK for ODE}.

This shows that the 2-nd order Gauss method applied to \eqref{eq:PRK for ODE} defines a multisymplectic method in the sense of formula \eqref{eq:Discrete Multisymplectic Form Formula}. However, for other symplectic partitioned Runge-Kutta methods of interest to us, namely the 4-th order Gauss and the 2-nd/4-th order Lobatto IIIA-IIIB methods, it is not possible to isolate a discrete Lagrangian $\tilde L$ that would only depend on four values $y^{j}_{i}$, $y^{j}_{i+1}$, $y^{j+1}_{i+1}$, $y^{j+1}_{i}$. The mentioned methods have more internal stages, and the equations \eqref{eq:PRK for ODE} couple them in a nontrivial way. Effectively, at any given time step the internal stages depend on all the values $y^{j}_{1}$, \ldots, $y^{j}_{N}$ and $y^{j+1}_{1}$, \ldots, $y^{j+1}_{N}$, and it it not possible to express the discrete Lagrangian \eqref{eq:DiscreteLagrangianRK} as a sum similar to \eqref{eq: Discrete Lagrangian for the Midpoint Rule in terms of L tilde}. The resulting integrators are still variational, since they are derived by applying the discrete variational principle \eqref{eq:DiscreteVariationalPrincipleForApproach1} to some discrete action $\tilde S$, but this action cannot be expressed as the sum of $\tilde L$ over all rectangles. Therefore, these integrators are not multisymplectic, at least not in the sense of formula \eqref{eq:Discrete Multisymplectic Form Formula}.

\paragraph{Constraints.} Let the additional bundle be $\mathcal{B}=\mathcal{X}\times[0,X_{max}]$ and denote by $X_j^n$ the elements of the fiber $\mathcal{B}_{ji}$. Define $\tilde Y = Y \oplus \mathcal{B}$. We have $J^k \tilde Y \cong J^k Y\oplus J^k \mathcal{B}$. Suppose $G: J^k \tilde Y \longrightarrow \mathbb{R}$ represents a discretization of the continuous constraint. For instance, one can enforce a uniform mesh by defining $G:J^1 \tilde Y\rightarrow \mathbb{R}$, $G(j^1 \tilde \varphi,j^1 \tilde X)=X_x-1$ at the continuous level. The discrete counterpart will be defined on the discrete jet bundle $J^1 \tilde Y$ by the formula

\begin{equation}
\label{eq:G for k1}
G(y^j_i,y^j_{i+1},y^{j+1}_{i+1},y^{j+1}_i,X^j_i,X^j_{i+1},X^{j+1}_{i+1},X^{j+1}_i)=\frac{X^j_{i+1}-X^j_i}{\Delta x}-1.
\end{equation}

\noindent
Arc-length equidistribution can be realized by enforcing \eqref{eq:MMPDE0 - arclength}, that is, $G:J^2_0 \tilde Y\rightarrow \mathbb{R}$, $G(j^2_0 \tilde \varphi,j^2_0 \tilde X)=\alpha^2 \varphi_x \varphi_{xx}+X_x X_{xx}$. The discrete counterpart will be defined on the discrete subbundle $J^2_0 \tilde Y$ by the formula

\begin{align}
\label{eq:G for k2}
G(y_{\boxbar^l},X_{\boxbar^r})=\alpha^2(y_{\boxbar^3}-y_{\boxbar^2})^2+(X_{\boxbar^3}-X_{\boxbar^2})^2-\alpha^2(y_{\boxbar^2}-y_{\boxbar^1})^2-(X_{\boxbar^2}-X_{\boxbar^1})^2,
\end{align}

\noindent
where for convenience we used the notation introduced in \eqref{eq:6-tuple} and $l,r=1,\ldots,6$. Note that \eqref{eq:G for k2} coincides with \eqref{eq:ArclengthConstraint}. In fact, $g_i$ in \eqref{eq:ArclengthConstraint} is nothing else but $G$ computed on an element of $J^2_0 \tilde Y$ over the base 6-tuple $\boxbar$ such that $\boxbar^2=(j,i)$. The only difference is that in \eqref{eq:ArclengthConstraint} we assumed $g_i$ might depend on \emph{all} the field values at a given time step, while $G$ only takes arguments \emph{locally}, i.e. it depends on \emph{at most} 6 field values on a given 6-tuple.

A numerical scheme is now obtained by simultaneously solving the discrete Euler-Lagrange equations \eqref{eq:DELF} resulting from \eqref{eq:DiscreteVariationalPrincipleForApproach1} and the equation $G=0$. If we know $y^{j-1}_i$, $X^{j-1}_i$, $y^{j}_i$ and $X^{j}_i$ for $i=1,\ldots,N$, this system of equations allows us to solve for $y^{j+1}_i$, $X^{j+1}_i$. This numerical scheme is multisymplectic in the sense similar to Proposition~\ref{Thm: PRK for DAE and ODE}. If we take $X(t,x)$ to be a sufficiently smooth interpolation of the values $X^j_i$ and substitute it in the problem \eqref{eq:DiscreteVariationalPrincipleForApproach1}, then the resulting multisymplectic integrator will yield the same numerical values $y^{j+1}_i$.

\subsection{Analysis of the Lagrange multiplier approach}
\label{sec:multisymp approach2}

\subsubsection*{Continuous setting}

We now turn to describing the Lagrange multiplier approach in a multisymplectic setting. Similarly as in Section~\ref{subsec: multisymp approach1}, let the computational spacetime be $\mathcal{X}=\mathbb{R} \times [0,X_{max}]$ with coordinates $(t,x)$ and consider the trivial configuration bundles $\pi_{\mathcal{X}Y}:Y=\mathcal{X}\times \mathbb{R} \longrightarrow \mathcal{X}$ and $\pi_{\mathcal{X}\mathcal{B}}:\mathcal{B}=\mathcal{X}\times[0,X_{max}] \longrightarrow \mathcal{X}$. Let our scalar field be represented by a section $\tilde \varphi: \mathcal{X} \longrightarrow Y$ with the coordinate representation $\tilde \varphi(t,x)=(t,x,\varphi(t,x))$ and our diffeomorphism by a section $\tilde X:\mathcal{X}\longrightarrow \mathcal{B}$ with the local representation $\tilde X(t,x) = (t,x, X(t,x))$. Let the total configuration bundle be $\tilde Y = Y \oplus \mathcal{B}$. Then the Lagrangian density \eqref{eq:LagrangianDensity} can be viewed as a mapping $\tilde{\mathcal{L}}: J^1 \tilde Y \cong J^1 Y\oplus J^1 \mathcal{B} \longrightarrow \mathbb{R}$. The corresponding action \eqref{eq:action2} can now be expressed as

\begin{equation}
\label{eq:Action2Multisymplectic}
\tilde S [\tilde \varphi, \tilde X] = \int_\mathcal{U} \tilde{\mathcal{L}} \big(j^1 \tilde\varphi, j^1 \tilde X\big)\,dt \wedge dx,
\end{equation}

\noindent
where $\mathcal{U}=[0,T_{max}]\times[0,X_{max}]$. As before, the MMPDE constraint can be represented by a function $G:J^k \tilde Y \longrightarrow \mathbb{R}$. Two sections $\tilde \varphi$ and $\tilde X$ satisfy the constraint if 

\begin{equation}
\label{eq:MultisymplecticConstraint}
G(j^k \tilde \varphi, j^k \tilde X)=0. 
\end{equation}

\paragraph{Vakonomic formulation.} We now face the problem of finding the right equations of motion. We want to extremize the action functional \eqref{eq:Action2Multisymplectic} in some sense, subject to the constraint \eqref{eq:MultisymplecticConstraint}. Note that the constraint is essentially \emph{nonholonomic}, as it depends on the derivatives of the fields. Assuming $G$ is a submersion, $G=0$ defines a submanifold of $J^k \tilde Y$, but this submanifold will not in general be the $k$-th jet of any subbundle of $\tilde Y$. Two distinct approaches are possible here. One could follow the {\it Lagrange-d'Alembert principle} and take variations of $\tilde S$ first, but choosing variations $V$ (vertical vector fields on $\tilde Y$) such that the jet prolongations $j^k V$ are tangent to the submanifold $G=0$, and then enforce the constraint $G=0$. On the other hand, one could consider the {\it variational nonholonomic} problem (also called {\it vakonomic}), and minimize $\tilde S$ over the set of all sections $(\tilde \varphi, \tilde X)$ that satisfy the constraint $G=0$, that is, enforce the constraint before taking the variations. If the constraint is holonomic, both approaches yield the same equations of motion. However, if the constraint is nonholonomic, the resulting equations are in general different. Which equations are correct is really a matter of experimental verification. It has been established that the Lagrange-d'Alembert principle gives the right equations of motion for nonholonomic mechanical systems, whereas the vakonomic setting is appropriate for optimal control problems (see \cite{BlochBOOK}, \cite{BlochCrouch}, \cite{BlochMarsden}, \cite{CortesLeonDiego2002}). 

We will argue that the vakonomic approach is the right one in our case. In Proposition~\ref{Thm: App2 extremization equivalence} we showed that in the unconstrained case extremizing $S[\phi]$ with respect to $\phi$ was equivalent to extremizing $\tilde S[\tilde \varphi, \tilde X]$ with respect to $\tilde \varphi$, and in Proposition~\ref{Thm: E-L equations dependence} we showed that extremizing with respect to $\tilde X$ did not yield new information. This is because there was no restriction on the fields $\tilde \varphi$ and $\tilde X$, and for any given $\tilde X$ there was a one-to-one correspondence between $\phi$ and $\tilde \varphi$ given by the formula $\varphi(t,x)=\phi(t,X(t,x))$, so extremizing over all possible $\tilde \varphi$ was equivalent to extremizing over all possible $\phi$. Now, let $\mathcal{N}$ be the set of all smooth sections $(\tilde \varphi,\tilde X)$ that satisfy the constraint \eqref{eq:MultisymplecticConstraint} such that $X(t,.)$ is a diffeomorphism for all $t$. It should be intuitively clear that under appropriate assumptions on the mesh density function $\rho$, for any given smooth function $\phi(t,X)$, equation \eqref{eq:MMPDE0} together with $\varphi(t,x)=\phi(t,X(t,x))$ define a unique pair $(\tilde \varphi,\tilde X)\in \mathcal{N}$ (since our main purpose here is to only justify the application of the vakonomic approach, we do not attempt to specify those analytic assumptions precisely). Conversely, any given pair $(\tilde \varphi,\tilde X)\in \mathcal{N}$ defines a unique function $\phi$ through the formula $\phi(t,X)=\varphi(t,\xi(t,X))$, where $\xi(t,.)=X(t,.)^{-1}$, as in Section~\ref{subsec:reparametrized lagrangian 2}. Given this one-to-one correspondence and the fact that $S[\phi]=\tilde S[\tilde \varphi, \tilde X]$ by definition, we see that extremizing $S$ with respect to all smooth $\phi$ is equivalent to extremizing $\tilde S$ over all smooth sections $(\tilde \varphi,\tilde X)\in \mathcal{N}$. We conclude that the vakonomic approach is appropriate in our case, since it follows from Hamilton's principle for the original, physically meaningful, action functional $S$.

Let us also note that our constraint depends on spatial derivatives only. Therefore, in the setting presented in Section~\ref{sec:approach1} and Section~\ref{sec:approach2} it can be considered holonomic, as it restricts the infinite-dimensional configuration manifold of fields that we used as our configuration space. In that case it is valid to use Hamilton's principle and minimize the action functional over the set of all allowable fields, i.e. those that satisfy the constraint $G=0$. We did that by considering the augmented instantaneous Lagrangian \eqref{eq:Auxiliary Lagrangian - approach 2}. 

In order to minimize $\tilde S$ over the set of sections satisfying the constraint \eqref{eq:MultisymplecticConstraint} we will use the bundle-theoretic version of the Lagrange multiplier theorem, which we cite below after \cite{MarsdenPekarskyShkollerWest}.

\begin{thm}[{\bf Lagrange multiplier theorem}]
\label{Thm: Lagrange multiplier theorem}
Let $\pi_{\mathcal{M},\mathcal{E}}:\mathcal{E}\longrightarrow \mathcal{M}$ be an inner product bundle over a smooth manifold $\mathcal{M}$, $\Psi$ a smooth section of $\pi_{\mathcal{M},\mathcal{E}}$, and $h:\mathcal{M} \longrightarrow \mathbb{R}$ a smooth function. Setting $\mathcal{N}=\Psi^{-1}(0)$, the following are equivalent:

\begin{enumerate}
\item $\sigma \in \mathcal{N}$ is an extremum of $h|_\mathcal{N}$,
\item there exists an extremum $\bar \sigma \in \mathcal{E}$ of $\bar h : \mathcal{E} \longrightarrow \mathbb{R}$ such that $\pi_{\mathcal{M},\mathcal{E}}(\bar \sigma)=\sigma$,
\end{enumerate}

\noindent
where $\bar h (\bar \sigma) = h( \pi_{\mathcal{M},\mathcal{E}}(\bar \sigma) ) - \big\langle \bar \sigma, \Psi(\pi_{\mathcal{M},\mathcal{E}}(\bar \sigma)) \big\rangle_\mathcal{E}$.\\
\end{thm}

Let us briefly review the ideas presented in \cite{MarsdenPekarskyShkollerWest}, adjusting the notation to our problem and generalizing when necessary. Let 

\begin{equation}
\label{eq:SmoothSectionsY}
C_\mathcal{U}^\infty(\tilde Y) = \{\sigma=(\tilde \varphi,\tilde X): \mathcal{U}\subset\mathcal{X}\longrightarrow \tilde Y \} 
\end{equation}

\noindent
be the set of smooth sections of $\pi_{\mathcal{X}\tilde Y}$ on $\mathcal{U}$. Then $\tilde S: C_\mathcal{U}^\infty(\tilde Y) \longrightarrow \mathbb{R}$ can be identified with $h$ in Theorem~\ref{Thm: Lagrange multiplier theorem}, where $\mathcal{M}=C_\mathcal{U}^\infty(\tilde Y)$. Furthermore, define the trivial bundle 

\begin{equation}
\label{eq:BundleV}
\pi_{\mathcal{X}\mathcal{V}}: \mathcal{V}=\mathcal{X}\times \mathbb{R} \longrightarrow \mathcal{X}
\end{equation}

\noindent
and let $C_\mathcal{U}^\infty(\mathcal{V})$ be the set of smooth sections $\tilde \lambda : \mathcal{U}\longrightarrow \mathcal{V}$, which represent our Lagrange multipliers and in local coordinates have the representation $\tilde \lambda(t,x) = (t,x,\lambda(t,x))$. The set $C_\mathcal{U}^\infty(\mathcal{V})$ is an inner product space with $\langle\tilde \lambda_1,\tilde \lambda_2\rangle=\int_\mathcal{U} \lambda_1 \lambda_2\,dt\wedge dx$. Take 

\begin{equation}
\label{eq:BundleE}
\mathcal{E}=C_\mathcal{U}^\infty(\tilde Y) \times C_\mathcal{U}^\infty(\mathcal{V}). 
\end{equation}

\noindent
This is an inner product bundle over $C_\mathcal{U}^\infty(\tilde Y)$ with the inner product defined by 

\begin{equation}
\label{eq:Inner Product on E}
\Big\langle (\sigma,\tilde \lambda_1), (\sigma,\tilde \lambda_2) \Big\rangle_\mathcal{E} = \langle\tilde \lambda_1,\tilde \lambda_2\rangle.
\end{equation}

\noindent
We now have to construct a smooth section $\Psi:C_\mathcal{U}^\infty(\tilde Y)\longrightarrow \mathcal{E}$ that will realize our constraint \eqref{eq:MultisymplecticConstraint}. Define the fiber-preserving mapping $\tilde G: J^k \tilde Y \longrightarrow \mathcal{V}$ such that for $\vartheta \in J^k\tilde Y$

\begin{equation}
\label{eq:Gtilde}
\tilde G(\vartheta)=\big( \pi_{\mathcal{X},J^k \tilde Y} (\vartheta),  G(\vartheta)\big).
\end{equation}

\noindent
For instance, for $k=1$, in local coordinates we have $\tilde G (t,x,y,v_t,v_x) = (t,x, G (t,x,y,v_t,v_x) )$. Then we can define

\begin{equation}
\label{eq:Psi}
\Psi(\sigma) = (\sigma,\tilde G\circ j^k\sigma).
\end{equation}

\noindent
The set of allowable sections $\mathcal{N}\subset C_\mathcal{U}^\infty(\tilde Y)$ is now defined by $\mathcal{N}=\Psi^{-1}(0)$. That is, $(\tilde \varphi, \tilde X) \in \mathcal{N}$ provided that $G(j^k \tilde \varphi, j^k \tilde X)=0$.

The augmented action functional $\tilde S_C : \mathcal{E}\longrightarrow \mathbb{R}$ is now given by

\begin{equation}
\label{eq:Augmented Action Multisymplectic}
\tilde S_C [\bar \sigma] = \tilde S[ \pi_{\mathcal{M},\mathcal{E}}(\bar \sigma) ] - \big\langle \bar \sigma, \Psi(\pi_{\mathcal{M},\mathcal{E}}(\bar \sigma)) \big\rangle_\mathcal{E},
\end{equation}

\noindent
or denoting $\bar \sigma = (\tilde \varphi, \tilde X,\tilde \lambda)$

\begin{align}
\label{eq:Augmented Action Multisymplectic Final}
\tilde S_C[\tilde \varphi, \tilde X,\tilde \lambda] &=\tilde S[ \tilde \varphi, \tilde X ] - \big\langle \tilde \lambda, \tilde G \circ (j^k \tilde \varphi,j^k \tilde X) \big\rangle \nonumber \\
&=\int_\mathcal{U} \tilde{\mathcal{L}} \big(j^1 \tilde\varphi, j^1 \tilde X\big)\,dt \wedge dx  - \int_\mathcal{U} \lambda(t,x) \, G(j^k \tilde \varphi,j^k \tilde X) \,dt \wedge dx \nonumber \\
&=\int_\mathcal{U} \Big[ \tilde{\mathcal{L}} \big(j^1 \tilde\varphi, j^1 \tilde X\big)-\lambda(t,x) \, G(j^k \tilde \varphi,j^k \tilde X)\Big] \,dt \wedge dx.
\end{align}

\noindent
Theorem~\ref{Thm: Lagrange multiplier theorem} states, that if $(\tilde \varphi, \tilde X,\tilde \lambda)$ is an extremum of $\tilde S_C$, then $(\tilde \varphi, \tilde X)$ extremizes $\tilde S$ over the set $\mathcal{N}$ of sections satisfying the constraint $G=0$. Note that using the multisymplectic formalism we obtained the same result as \eqref{eq:Auxiliary Lagrangian - approach 2} in the instantaneous formulation, where we could treat $G$ as a holonomic constraint. The dynamics is obtained by solving for a triple $(\tilde \varphi, \tilde X,\tilde \lambda)$ such that 

\begin{equation}
\label{eq:Multisymplectic Extremization of Augmented SC tilde}
\frac{d}{d\epsilon} \bigg|_{\epsilon=0} \tilde S_C[\eta_Y^\epsilon \circ \tilde \varphi,\eta_\mathcal{B}^\epsilon \circ \tilde X,\eta_\mathcal{V}^\epsilon \circ \tilde \lambda]=0
\end{equation}

\noindent
for all $\eta_Y^\epsilon$, $\eta_\mathcal{B}^\epsilon$, $\eta_\mathcal{V}^\epsilon$ that keep the boundary conditions on $\partial \mathcal{U}$ fixed, where $\eta^\epsilon$ denotes the flow of vertical vector fields on respective bundles.

Note that we can define $\tilde Y_C = Y \oplus \mathcal{B} \oplus \mathcal{V}$ and $\tilde{\mathcal{L}_C}: J^k \tilde Y_C \longrightarrow \mathbb{R}$ by setting $\tilde{\mathcal{L}_C} = \tilde{\mathcal{L}} - \lambda\cdot G$, i.e., we can consider a $k$-th order field theory. If $k=1,2$ then an appropriate multisymplectic form formula in terms of the fields $\tilde \varphi$, $\tilde X$ and $\tilde \lambda$ will hold. Presumably, this can be generalized for $k>2$ using the techniques put forth in \cite{KouranbaevaShkoller}. However, it is an interesting question whether there exists any multisymplectic form formula defined in terms of $\tilde \varphi$, $\tilde X$ and objects on $J^k \tilde Y$ only. It appears to be an open problem. This would be the multisymplectic analog of the fact that the flow of a constrained mechanical system is symplectic on the constraint submanifold of the configuration space.

\subsubsection*{Discretization}

Let us use the same discretization as discussed in Section~\ref{subsec: multisymp approach1}. Assume we have a discrete Lagrangian $\tilde L: J^1 \tilde Y \longrightarrow \mathbb{R}$, the corresponding discrete action $\tilde S[\tilde \varphi, \tilde X]$, and a discrete constraint $G:J^1 \tilde Y \longrightarrow \mathbb{R}$ or $G:J^2_0 \tilde Y \longrightarrow \mathbb{R}$. Note that $\tilde S$ is essentially a function of $2MN$ variables and we want to extremize it subject to the set of algebraic constraints $G=0$. The standard Lagrange multiplier theorem proved in basic calculus textbooks applies here. However, let us work out a discrete counterpart of the formalism introduced at the continuous level. This will facilitate the discussion of the discrete notion of multisymplecticity. Let

\begin{equation}
\label{eq:DiscreteSectionsY}
C_\mathcal{U}(\tilde Y) = \{\sigma=(\tilde \varphi,\tilde X): \mathcal{U}\subset\mathcal{X}\longrightarrow \tilde Y \} 
\end{equation}

\noindent
be the set of discrete sections of $\pi_{\mathcal{X}\tilde Y}:\tilde Y \longrightarrow \mathcal{X}$. Similarly, define the discrete bundle $\mathcal{V}=\mathcal{X}\times \mathbb{R}$ and let $C_{\mathcal{U}_0}(\mathcal{V})$ be the set of discrete sections $\tilde \lambda:\mathcal{U}_0\longrightarrow\mathcal{V}$ representing the Lagrange multipliers, where $\mathcal{U}_0\subset\mathcal{U}$ is defined below. Let $\tilde \lambda(j,i) = (j,i,\lambda(j,i))$ with $\lambda^j_i\equiv\lambda(j,i)$ be the local representation. The set $C_{\mathcal{U}_0}(\mathcal{V})$ is an inner product space with $\langle\tilde \lambda,\tilde \mu\rangle = \sum_{(j,i)\in \mathcal{U}_0}\lambda^j_i \mu^j_i$. Take $\mathcal{E}=C_\mathcal{U}(\tilde Y) \times C_{\mathcal{U}_0}(\mathcal{V})$. Just like at the continuous level, $\mathcal{E}$ is an inner product bundle. However, at the discrete level it is more convenient to define the inner product on $\mathcal{E}$ in a slightly modified way. Since there are some nuances in the notation, let us consider the cases $k=1$ and $k=2$ separately.

\paragraph{Case $k=1$.} Let $\mathcal{U}_0=\{(j,i)\in \mathcal{U} \,|\,j\leq M,i\leq N\}$. Define the trivial bundle $\hat{\mathcal{V}} = \mathcal{X}^\square \times \mathbb{R}$ and let $C_{\mathcal{U}^\square}(\hat{\mathcal{V}})$ be the set of all sections of $\hat{\mathcal{V}}$ defined on $\mathcal{U}^\square$. For a given section $\tilde \lambda \in C_{\mathcal{U}_0}(\mathcal{V})$ we define its extension $\hat \lambda \in C_{\mathcal{U}^\square}(\hat{\mathcal{V}})$ by

\begin{equation}
\label{eq:Lambda Extension for k=1}
\hat \lambda(\square)=\big(\square,\lambda(\square^1)\big),
\end{equation}

\noindent
that is, $\hat \lambda$ assigns to the square $\square$ the value that $\tilde \lambda$ takes on the \emph{first} vertex of that square. Note that this operation is invertible: given a section of $C_{\mathcal{U}^\square}(\hat{\mathcal{V}})$ we can uniquely determine a section of $C_{\mathcal{U}_0}(\mathcal{V})$. We can define the inner product

\begin{equation}
\label{eq:Inner Product on the Space of Extensions for k=1}
\langle \hat \lambda, \hat \mu \rangle = \sum_{\square \subset \mathcal{U}} \lambda(\square^1) \mu(\square^1).
\end{equation}

\noindent
One can easily see that we have $\langle \hat \lambda, \hat \mu \rangle = \langle \tilde \lambda, \tilde \mu \rangle$, so by a slight abuse of notation we can use the same symbol $\langle.,.\rangle$ for both inner products. It will be clear from the context which definition should be invoked. We can now define an inner product on the fibers of $\mathcal{E}$ as

\begin{equation}
\label{eq:Discrete Inner Product on E for k=1}
\Big\langle (\sigma,\tilde \lambda), (\sigma,\tilde \mu) \Big\rangle_\mathcal{E} = \langle \hat\lambda,\hat\mu\rangle = \langle \tilde \lambda,\tilde \mu\rangle.
\end{equation}

\noindent 
Let us now construct a section $\Psi:C_\mathcal{U}(\tilde Y) \longrightarrow \mathcal{E}$ that will realize our discrete constraint $G$. First, in analogy to \eqref{eq:Gtilde}, define the fiber-preserving mapping $\tilde G:J^1\tilde Y \longrightarrow \hat{\mathcal{V}}$ such that

\begin{equation}
\label{eq:Gtilde Discrete k=1}
\tilde G(y_{\square^l},X_{\square^r}) = \big(\square,G(y_{\square^l},X_{\square^r})\big),
\end{equation}

\noindent
where $l,r=1,2,3,4$. We now define $\Psi$ by requiring that for $\sigma \in C_\mathcal{U}(\tilde Y)$ the extension \eqref{eq:Lambda Extension for k=1} of $\Psi(\sigma)$ is given by

\begin{equation}
\label{eq:Psi Discrete for k=1}
\hat \Psi(\sigma) = (\sigma,\tilde G\circ j^1\sigma).
\end{equation}

\noindent
The set of allowable sections $\mathcal{N}\subset C_\mathcal{U}(\tilde Y)$ is now defined by $\mathcal{N}=\Psi^{-1}(0)$---that is, $(\tilde \varphi, \tilde X) \in \mathcal{N}$ provided that $G(j^1 \tilde \varphi, j^1 \tilde X)=0$ for all $\square \in \mathcal{U}^\square$. The augmented discrete action $\tilde S_C:\mathcal{E}\longrightarrow \mathbb{R}$ is therefore

\begin{align}
\label{eq:Discrete Augmented Action for k=1}
\tilde S_C[\sigma,\tilde \lambda] &= \tilde S[\sigma] - \Big\langle (\sigma,\tilde \lambda), \Psi(\sigma) \Big\rangle_\mathcal{E} \nonumber \\
&= \tilde S[\sigma] - \Big\langle \hat \lambda, \tilde G\circ j^1\sigma \Big\rangle \nonumber\\
&= \sum_{\square \subset \mathcal{U}}\tilde L(j^1 \sigma) - \sum_{\square \subset \mathcal{U}}\lambda(\square^1) G(j^1 \sigma) \nonumber\\
&= \sum_{\square \subset \mathcal{U}} \Big( \tilde L(j^1 \sigma) - \lambda(\square^1) G(j^1 \sigma) \Big).
\end{align}

\noindent
By the standard Lagrange multiplier theorem, if $(\tilde \varphi, \tilde X,\tilde \lambda)$ is an extremum of $\tilde S_C$, then $(\tilde \varphi, \tilde X)$ is an extremum of $\tilde S$ over the set $\mathcal{N}$ of sections satisfying the constraint $G=0$. The discrete Hamilton principle can be expressed as

\begin{equation}
\label{eq:Discrete Variational Principle for k=1}
\frac{d}{d \epsilon} \bigg|_{\epsilon=0} \tilde S_C[\tilde \varphi_\epsilon,\tilde X_\epsilon, \tilde \lambda_\epsilon]=0
\end{equation}

\noindent
for all vector fields $V$ on $Y$, $W$ on $\mathcal{B}$, and $Z$ on $\mathcal{V}$ that keep the boundary conditions on $\partial \mathcal{U}$ fixed, where $\tilde \varphi_\epsilon(j,i)=F^{V_{ji}}_\epsilon (\tilde \varphi(j,i))$ and $F^{V_{ji}}_\epsilon$ is the flow of $V_{ji}$ on $\mathbb{R}$, and similarly for $\tilde X_\epsilon$ and $\tilde \lambda_\epsilon$. The discrete Euler-Lagrange equations can be conveniently computed if in \eqref{eq:Discrete Variational Principle for k=1} one focuses on some $(j,i)\in \textrm{int}\,\mathcal{U}$. With the convention $\tilde \varphi(j,i)=y^j_i$, $\tilde X(j,i)=X^j_i$, $\tilde \lambda(j,i)=\lambda^j_i$, we write the terms of $\tilde S_C$ containing $y^j_i$, $X^j_i$ and $\lambda^j_i$ explicitly as

\begin{align}
\label{Terms in SC containing (j,i) terms for k=1}
\tilde S_C = \ldots &+ \tilde L \big(y^j_i,y^j_{i+1},y^{j+1}_{i+1},y^{j+1}_i,X^j_i,X^j_{i+1},X^{j+1}_{i+1},X^{j+1}_i\big) \nonumber \\
&+ \tilde L \big(y^j_{i-1},y^j_{i},y^{j+1}_{i},y^{j+1}_{i-1},X^j_{i-1},X^j_{i},X^{j+1}_{i},X^{j+1}_{i-1}\big) \nonumber \\
&+ \tilde L \big(y^{j-1}_{i-1},y^{j-1}_{i},y^{j}_{i},y^{j}_{i-1},X^{j-1}_{i-1},X^{j-1}_{i},X^{j}_{i},X^{j}_{i-1}\big) \nonumber \\
&+ \tilde L \big(y^{j-1}_{i},y^{j-1}_{i+1},y^{j}_{i+1},y^{j}_{i},X^{j-1}_{i},X^{j-1}_{i+1},X^{j}_{i+1},X^{j}_{i}\big) \nonumber \\
&+ \lambda^j_i G \big(y^j_i,y^j_{i+1},y^{j+1}_{i+1},y^{j+1}_i,X^j_i,X^j_{i+1},X^{j+1}_{i+1},X^{j+1}_i\big) \nonumber \\
&+ \lambda^j_{i-1} G \big(y^j_{i-1},y^j_{i},y^{j+1}_{i},y^{j+1}_{i-1},X^j_{i-1},X^j_{i},X^{j+1}_{i},X^{j+1}_{i-1}\big) \nonumber \\
&+ \lambda^{j-1}_{i-1} G \big(y^{j-1}_{i-1},y^{j-1}_{i},y^{j}_{i},y^{j}_{i-1},X^{j-1}_{i-1},X^{j-1}_{i},X^{j}_{i},X^{j}_{i-1}\big) \nonumber \\
&+ \lambda^{j-1}_i G \big(y^{j-1}_{i},y^{j-1}_{i+1},y^{j}_{i+1},y^{j}_{i},X^{j-1}_{i},X^{j-1}_{i+1},X^{j}_{i+1},X^{j}_{i}\big) + \ldots
\end{align}

\noindent
The discrete Euler-Lagrange equations are obtained by differentiating with respect to $y^j_i$, $X^j_i$ and $\lambda^j_i$, and can be written compactly as

\begin{align}
\label{DELF for k=1}
\sum_{\substack{l,\square\\ \scriptscriptstyle (j,i)=\square^l}} \bigg[ \frac{\partial \tilde L}{\partial y^l} (y_{\square^1},\ldots,y_{\square^4},&X_{\square^1},\ldots,X_{\square^4}) + \nonumber \\
&+\lambda_{\square^1}\frac{\partial G}{\partial y^l} (y_{\square^1},\ldots,y_{\square^4},X_{\square^1},\ldots,X_{\square^4}) \bigg]=0, \nonumber \\
\nonumber \\
\sum_{\substack{l,\square\\ \scriptscriptstyle (j,i)=\square^l}} \bigg[ \frac{\partial \tilde L}{\partial X^l} (y_{\square^1},\ldots,y_{\square^4},&X_{\square^1},\ldots,X_{\square^4}) + \nonumber \\
&+\lambda_{\square^1}\frac{\partial G}{\partial X^l} (y_{\square^1},\ldots,y_{\square^4},X_{\square^1},\ldots,X_{\square^4}) \bigg]=0, \nonumber \\
\nonumber \\
G \big(y^j_i,y^j_{i+1},y^{j+1}_{i+1},y^{j+1}_i,X&^j_i,X^j_{i+1},X^{j+1}_{i+1},X^{j+1}_i\big)=0
\end{align}

\noindent
for all $(j,i) \in \textrm{int}\,\mathcal{U}$. If we know $y^{j-1}_i$, $X^{j-1}_i$, $y^{j}_i$, $X^{j}_i$ and $\lambda^{j-1}_i$ for $i=1,\ldots,N$, this system of equations allows us to solve for $y^{j+1}_i$, $X^{j+1}_i$ and $\lambda^{j}_i$.

Note that we can define $\tilde Y_C = Y \oplus \mathcal{B} \oplus \mathcal{V}$ and the augmented Lagrangian $\tilde{L}_C: J^1 \tilde Y_C \longrightarrow \mathbb{R}$ by setting

\begin{equation}
\label{eq:Discrete Augmented Lagrangian for k=1}
\tilde{L}_C(j^1 \tilde \varphi, j^1 \tilde X,j^1 \tilde\lambda) = \tilde{L}(j^1 \tilde \varphi, j^1 \tilde X) - \lambda(\square^1)\cdot G(j^1 \tilde \varphi, j^1 \tilde X),
\end{equation} 

\noindent
that is, we can consider an unconstrained field theory in terms of the fields $\tilde \varphi$, $\tilde X$ and $\tilde \lambda$. Then, the solutions of \eqref{DELF for k=1} satisfy the multisymplectic form formula \eqref{eq:Discrete Multisymplectic Form Formula} in terms of objects defined on $J^1 \tilde Y_C$.

\paragraph{Case $k=2$.} Let $\mathcal{U}_0=\{(j,i)\in \mathcal{U} \,|\,j\leq M,1\leq i\leq N\}$. Define the trivial bundle $\hat{\mathcal{V}} = \mathcal{X}^\boxbar \times \mathbb{R}$ and let $C_{\mathcal{U}^\boxbar}(\hat{\mathcal{V}})$ be the set of all sections of $\hat{\mathcal{V}}$ defined on $\mathcal{U}^\boxbar$. For a given section $\tilde \lambda \in C_{\mathcal{U}_0}(\mathcal{V})$ we define its extension $\hat \lambda \in C_{\mathcal{U}^\boxbar}(\hat{\mathcal{V}})$ by

\begin{equation}
\label{eq:Lambda Extension for k=2}
\hat \lambda(\boxbar)=\big(\boxbar,\lambda(\boxbar^2)\big),
\end{equation}

\noindent
that is, $\hat \lambda$ assigns to the 6-tuple $\boxbar$ the value that $\tilde \lambda$ takes on the \emph{second} vertex of that 6-tuple. Like before, this operation is invertible. We can define the inner product

\begin{equation}
\label{eq:Inner Product on the Space of Extensions for k=2}
\langle \hat \lambda, \hat \mu \rangle = \sum_{\boxbar \subset \mathcal{U}} \lambda(\boxbar^2) \mu(\boxbar^2)
\end{equation}

\noindent
and the inner product on $\mathcal{E}$ as in \eqref{eq:Discrete Inner Product on E for k=1}. Define the fiber-preserving mapping $\tilde G:J^2_0\tilde Y \longrightarrow \hat{\mathcal{V}}$ such that

\begin{equation}
\label{eq:Gtilde Discrete k=2}
\tilde G(y_{\boxbar^l},X_{\boxbar^r}) = \big(\boxbar,G(y_{\boxbar^l},X_{\boxbar^r})\big),
\end{equation}

\noindent
where $l,r=1,\ldots,6$. We now define $\Psi$ by requiring that for $\sigma \in C_\mathcal{U}(\tilde Y)$ the extension \eqref{eq:Lambda Extension for k=2} of $\Psi(\sigma)$ is given by

\begin{equation}
\label{eq:Psi Discrete for k=2}
\hat \Psi(\sigma) = (\sigma,\tilde G\circ j^2_0\sigma).
\end{equation}

\noindent
Again, the set of allowable sections is $\mathcal{N}=\Psi^{-1}(0)$. That is, $(\tilde \varphi, \tilde X) \in \mathcal{N}$ provided that $G(j^2_0 \tilde \varphi, j^2_0 \tilde X)=0$ for all $\boxbar \in \mathcal{U}^\boxbar$. The augmented discrete action $\tilde S_C:\mathcal{E}\longrightarrow \mathbb{R}$ is therefore

\begin{align}
\label{eq:Discrete Augmented Action for k=2}
\tilde S_C[\sigma,\tilde \lambda] &= \tilde S[\sigma] - \Big\langle (\sigma,\tilde \lambda), \Psi(\sigma) \Big\rangle_\mathcal{E} \nonumber \\
&= \tilde S[\sigma] - \Big\langle \hat \lambda, \tilde G\circ j^2_0\sigma \Big\rangle \nonumber\\
&= \sum_{\square \subset \mathcal{U}}\tilde L(j^1 \sigma) - \sum_{\boxbar \subset \mathcal{U}}\lambda(\boxbar^2) G(j^2_0 \sigma).
\end{align}

\noindent
Writing out the terms involving $y^j_i$, $X^j_i$ and $\lambda^j_i$ explicitly, as in \eqref{Terms in SC containing (j,i) terms for k=1}, and invoking the discrete Hamilton principle \eqref{eq:Discrete Variational Principle for k=1}, one obtains the discrete Euler-Lagrange equations, which can be compactly expressed as

\begin{align}
\label{DELF for k=2}
\sum_{\substack{l,\square\\ \scriptscriptstyle (j,i)=\square^l}}\frac{\partial \tilde L}{\partial y^l} (y_{\square^1},\ldots,y_{\square^4},&X_{\square^1},\ldots,X_{\square^4}) + \nonumber \\
&+\sum_{\substack{l,\boxbar\\ \scriptscriptstyle (j,i)=\boxbar^l}} \lambda_{\boxbar^2}\frac{\partial G}{\partial y^l} (y_{\boxbar^1},\ldots,y_{\boxbar^6},X_{\boxbar^1},\ldots,X_{\boxbar^6})  =0, \nonumber \\
\nonumber \\
\sum_{\substack{l,\square\\ \scriptscriptstyle (j,i)=\square^l}}\frac{\partial \tilde L}{\partial X^l} (y_{\square^1},\ldots,y_{\square^4},&X_{\square^1},\ldots,X_{\square^4}) + \nonumber \\
&+\sum_{\substack{l,\boxbar\\ \scriptscriptstyle (j,i)=\boxbar^l}} \lambda_{\boxbar^2}\frac{\partial G}{\partial X^l} (y_{\boxbar^1},\ldots,y_{\boxbar^6},X_{\boxbar^1},\ldots,X_{\boxbar^6})  =0, \nonumber \\
\nonumber \\
G \big(y^j_{i-1},y^j_i,y^j_{i+1},y^{j+1}_{i+1},y&^{j+1}_i,y^{j+1}_{i-1},X^j_{i-1},X^j_i,X^j_{i+1},X^{j+1}_{i+1},X^{j+1}_i,X^{j+1}_{i-1}\big)=0
\end{align}

\noindent
for all $(j,i) \in \textrm{int}\,\mathcal{U}$. If we know $y^{j-1}_i$, $X^{j-1}_i$, $y^{j}_i$, $X^{j}_i$ and $\lambda^{j-1}_i$ for $i=1,\ldots,N$, this system of equations allows us to solve for $y^{j+1}_i$, $X^{j+1}_i$ and $\lambda^{j}_i$.

Let us define the extension $\tilde L_{\textrm{ext}}:J^2_0 \tilde Y \longrightarrow \mathbb{R}$ of the Lagrangian density $\tilde L$ by setting

\begin{equation}
\label{eq:Lext}
\tilde L_{\textrm{ext}}(y_{\boxbar^1},\ldots,X_{\boxbar^6})=
\left\{
\begin{array}{ll}
\tilde L(y_{\square^1},\ldots,X_{\square^4}) & \textrm{if $\boxbar^2 = (j,0),(j,N+1)$,} \\
 & \textrm{where $\square=\boxbar \cap \mathcal{U}$,}\\
\frac{1}{2}\sum_{\square \subset \boxbar} \tilde L(y_{\square^1},\ldots,X_{\square^4}) & \textrm{otherwise.}
\end{array}
\right.
\end{equation}

\noindent
Let us also set $G(y_{\square^1},\ldots,X_{\square^4})=0$ if $\boxbar^2 = (j,0),(j,N+1)$. Define $\mathcal{A}=\{\boxbar \, | \, \boxbar^2,\boxbar^5 \in \mathcal{U} \}$. Then \eqref{eq:Discrete Augmented Action for k=2} can be written as

\begin{align}
\label{eq:Discrete Augmented Action for k=2 with Lext}
\tilde S_C[\sigma,\tilde \lambda] = \sum_{\boxbar \in \mathcal{A}}\Big[\tilde L_{\textrm{ext}}(j^2_0 \sigma) - \lambda(\boxbar^2) G(j^2_0 \sigma)\Big]=\sum_{\boxbar \in \mathcal{A}}\tilde L_C(j^2_0 \sigma,j^2_0 \tilde \lambda),
\end{align}

\noindent
where the last equality defines the augmented Lagrangian $\tilde L_C : J^2_0 \tilde Y_C \longrightarrow \mathbb{R}$ for $\tilde Y_C = Y \oplus \mathcal{B} \oplus \mathcal{V}$. Therefore, we can consider an unconstrained second-order field theory in terms of the fields $\tilde \varphi$, $\tilde X$ and $\tilde \lambda$, and the solutions of \eqref{DELF for k=2} will satisfy a discrete multisymplectic form formula very similar to the one proved in \cite{KouranbaevaShkoller}. The only difference is the fact that the authors analyzed a discretization of the Camassa-Holm equation and were able to consider an even smaller subbundle of the second jet of the configuration bundle. As a result it was sufficient for them to consider a discretization based on squares $\square$ rather than 6-tuples $\boxbar$. In our case there will be six discrete 2-forms $\Omega^l_{\tilde L_C}$ for $l=1,\ldots,6$ instead of just four.

\paragraph{Remark.} In both cases we showed that our discretization leads to integrators that are multisymplectic on the augmented jets $J^k \tilde Y_C$. However, just like in the continuous setting, it is an interesting problem whether there exists a discrete multisymplectic form formula in terms of objects defined on $J^k \tilde Y$ only.

\paragraph{Example: Trapezoidal rule.} Consider the semi-discrete Lagrangian \eqref{eq:ParticularLN}. We can use the trapezoidal rule to define the discrete Lagrangian \eqref{eq:DiscreteLagrangianDefinition} as

\begin{align}
\label{eq: Discrete Lagrangian for the Trapezoidal Rule}
\tilde L_d(y^j,X^j,y^{j+1},X^{j+1}) = \frac{\Delta t}{2} \tilde L_N \bigg(y^j&, X^j,\frac{y^{j+1}-y^j}{\Delta t},\frac{X^{j+1}-X^j}{\Delta t} \bigg) \nonumber \\
 &+ \frac{\Delta t}{2} \tilde L_N \bigg(y^{j+1},X^{j+1},\frac{y^{j+1}-y^j}{\Delta t},\frac{X^{j+1}-X^j}{\Delta t} \bigg),
\end{align}

\noindent
where $y^j=(y^j_1,\ldots,y^j_N)$ and $X^j=(X^j_1,\ldots,X^j_N)$. The constrained version (see \cite{MarsdenWestVarInt}) of the Discrete Euler-Lagrange equations \eqref{eq:Discrete Euler-Lagrange equations} takes the form

\begin{align}
\label{eq: Constrained Discrete Euler-Lagrange equations}
D_2 \tilde L_d(q^{j-1},q^j) + D_1 \tilde L_d(q^j,q^{j+1}) &= Dg(q^j)^T \lambda^j, \nonumber \\
g(q^{j+1}) &=0,
\end{align}

\noindent
where for brevity $q^j=(y^j_1,X^j_1,\ldots,y^j_N,X^j_N)$, $\lambda^j=(\lambda^j_1,\ldots,\lambda^j_N)$ and $g$ is an adaptation constraint, for instance \eqref{eq:ArclengthConstraint}. If $q^{j-1}$, $q^j$ are known, then \eqref{eq: Constrained Discrete Euler-Lagrange equations} can be used to compute $q^{j+1}$ and $\lambda^j$. It is easy to verify that the condition \eqref{eq:Condition for index 3} is enough to ensure solvability of \eqref{eq: Constrained Discrete Euler-Lagrange equations}, assuming the time step $\Delta t$ is sufficiently small, so there is no need to introduce slack degrees of freedom as in \eqref{eq:Augmented Lagrangian Big Mass Matrix}. If the mass matrix \eqref{eq:MassMatrix} was constant and nonsingular, then \eqref{eq: Constrained Discrete Euler-Lagrange equations} would result in the SHAKE algorithm, or in the RATTLE algorithm if one passes to the position-momentum formulation (see \cite{HLWGeometric}, \cite{MarsdenWestVarInt}).

Using \eqref{eq:Lagrangian2} and \eqref{eq:LN FEM2} we can write

\begin{equation}
\label{eq: Discrete Lagrangian for the Trapezoidal Rule in terms of L tilde}
\tilde L_d(y^j,X^j,y^{j+1},X^{j+1}) = \sum_{i=0}^N \tilde{L} \big( y^{j}_{i},y^{j}_{i+1},y^{j+1}_{i+1},y^{j+1}_{i}, X^{j}_{i},X^{j}_{i+1},X^{j+1}_{i+1},X^{j+1}_{i} \big),
\end{equation}

\noindent
where we defined the discrete Lagrangian $\tilde L : J^1 \tilde Y \longrightarrow \mathbb{R}$ by the formula

\begin{align}
\label{eq: Discrete Multisymplectic Lagrangian for the Trapezoidal Rule}
\tilde L\Big( y^{j}_{i},y^{j}_{i+1},y^{j+1}_{i+1},y^{j+1}_{i},X^{j}_{i},X^{j}_{i+1},X^{j+1}_{i+1},X^{j+1}_{i} \Big) &= \nonumber \\
\frac{\Delta t}{2} \int_{x_i}^{x_{i+1}} \tilde{\mathcal{L}}\Big(\bar\varphi^j(x), \bar X^j(x),\bar\varphi^j_x(x)&, \bar X^j_x(x),\bar\varphi_t(x),\bar X_t(x)  \Big)\,dx \nonumber \\
+\frac{\Delta t}{2} \int_{x_i}^{x_{i+1}} \tilde{\mathcal{L}}\Big(\bar\varphi^{j+1}(x)&, \bar X^{j+1}(x),\bar \varphi^{j+1}_x(x),\bar X^{j+1}_x(x),\bar \varphi_t(x),\bar X_t(x)  \Big)\,dx
\end{align}

\noindent
with

\begin{align}
\label{eq: Auxiliary terms for Discrete Lagrangian for the Trapezoidal Rule in terms of L tilde }
\bar\varphi^j(x) &= y^j_i \eta_i(x)+y^j_{i+1} \eta_{i+1}(x), \nonumber \\
\bar\varphi^j_x(x) &= \frac{y^{j}_{i+1}-y^{j}_{i}}{\Delta x}, \nonumber \\
\bar\varphi_t(x) &= \frac{y^{j+1}_i-y^j_i}{\Delta t} \eta_i(x)+\frac{y^{j+1}_{i+1}-y^j_{i+1}}{\Delta t} \eta_{i+1}(x),
\end{align}

\noindent
and similarly for $\bar X(x)$. Given the Lagrangian density $\tilde{\mathcal{L}}$ as in \eqref{eq:ParticularLagrangian} one can compute the integrals in \eqref{eq: Discrete Multisymplectic Lagrangian for the Trapezoidal Rule} explicitly. Suppose that the adaptation constraint $g$ has a \textquoteleft local' structure, for instance

\begin{equation}
\label{eq: Local structure of the constraint for k=1}
g_i(y^j,X^j)=G(y^j_i,y^j_{i+1},y^{j+1}_{i+1},y^{j+1}_i,X^j_i,X^j_{i+1},X^{j+1}_{i+1},X^{j+1}_i),
\end{equation}

\noindent
as in \eqref{eq:G for k1} or 

\begin{equation}
\label{eq: Local structure of the constraint for k=2}
g_i(y^j,X^j)=G(y_{\boxbar^l},X_{\boxbar^r}), \qquad \textrm{where $\boxbar^2=(j,i)$,}
\end{equation}

\noindent
as in \eqref{eq:G for k2}. It is straightforward to show that \eqref{DELF for k=1} or \eqref{DELF for k=2} are equivalent to \eqref{eq: Constrained Discrete Euler-Lagrange equations}, that is, the variational integrator defined by \eqref{eq: Constrained Discrete Euler-Lagrange equations} is also multisymplectic.

For reasons similar to the ones pointed out in Section~\ref{subsec: multisymp approach1}, the 2-nd and 4-th order Lobatto IIIA-IIIB methods that we used for our numerical computations are not multisymplectic.

\section{Numerical results}
\label{sec:numerics}

\subsection{The Sine-Gordon equation}
\label{sec: SG equation}

We applied the methods discussed in the previous sections to the Sine-Gordon equation

\begin{equation}
\label{eq:Sine-Gordon Equation}
\frac{\partial^2 \phi}{\partial t^2}-\frac{\partial^2 \phi}{\partial X^2}+\sin \phi=0.
\end{equation}

\noindent
This equation results from the (1+1)-dimensional scalar field theory with the Lagrangian density

\begin{equation}
\label{eq:SG Lagrangian density}
\mathcal{L}(\phi,\phi_X,\phi_t) = \frac{1}{2} \phi_t^2 - \frac{1}{2} \phi_X^2 - (1-\cos \phi).
\end{equation}

\noindent
The Sine-Gordon equation arises in many physical applications. For instance, it governs the propagation of dislocations in crystals, the evolution of magnetic flux in a long Josephson-junction transmission line or the modulation of a weakly unstable baroclinic wave packet in a two-layer fluid. It also has applications in the description of one-dimensional organic conductors, one-dimensional ferromagnets, liquid crystals, or in particle physics as a model for baryons (see \cite{DrazinSolitonsIntroduction}, \cite{Rajaraman}).

The Sine-Gordon equation has interesting soliton solutions. A single soliton traveling at the speed $v$ is given by

\begin{equation}
\label{eq: SG soliton}
\phi_S(X,t) = 4 \arctan \bigg[ \exp \bigg(\frac{X-X_0-vt}{\sqrt{1-v^2}}\bigg) \bigg].
\end{equation}

\noindent
It is depicted in Figure~\ref{fig: SG soliton}. The backscattering of two solitons, each traveling with the velocity $v$, is described by the formula

\begin{equation}
\label{eq: SG two solitons}
\phi_{SS}(X,t) = 4 \arctan \Bigg[ \frac{v \sinh (\frac{X}{\sqrt{1-v^2}} )} {\cosh(\frac{vt}{\sqrt{1-v^2}})} \Bigg].
\end{equation}

\noindent
It is depicted in Figure~\ref{fig: SG two solitons}. Note that if we restrict $X \geq0$, then this formula also gives a single soliton solution satisfying the boundary condition $\phi(0,t)=0$, that is, a soliton bouncing from a rigid wall.

\begin{figure}[tbp]
	\centering
		\includegraphics[width=\textwidth]{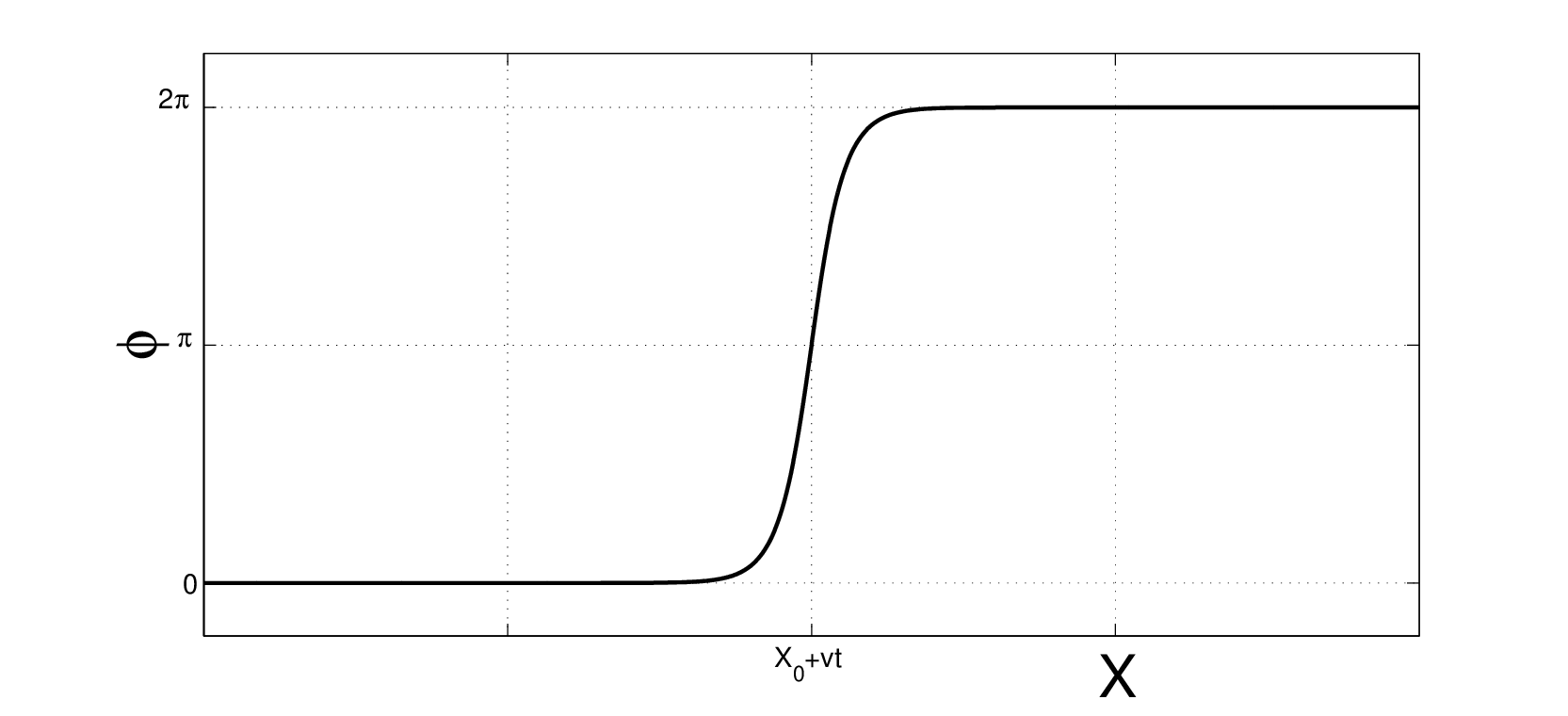}
		\caption{The single soliton solution of the Sine-Gordon equation.}
		\label{fig: SG soliton}
\end{figure}

\begin{figure}[tbp]

	\centering
		\includegraphics[width=\textwidth]{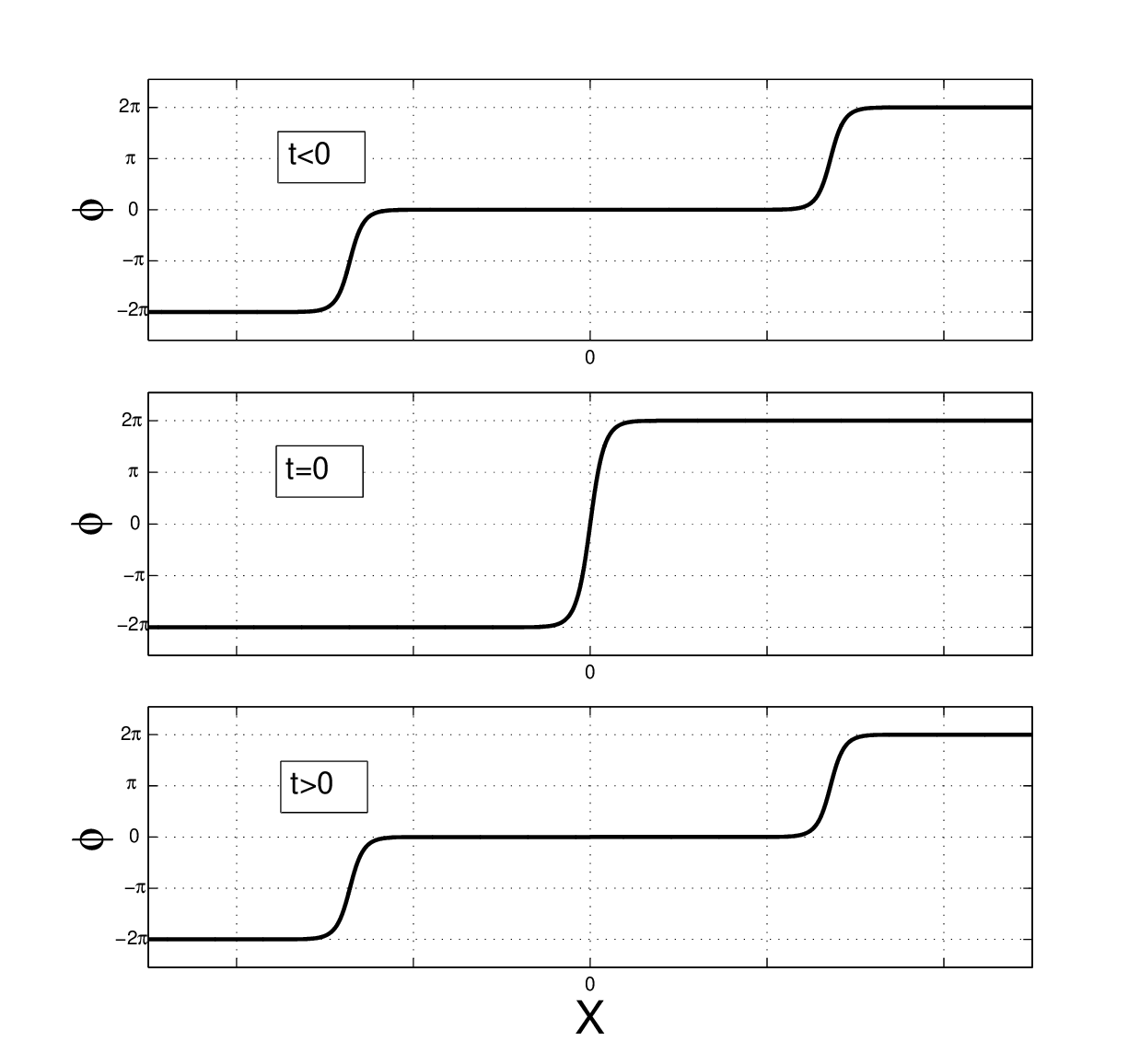}
		\caption{The two-soliton solution of the Sine-Gordon equation.}
		\label{fig: SG two solitons}
\end{figure}

\subsection{Generating consistent initial conditions}
\label{sec: Initial Conditions}

Suppose we specify the following initial conditions

\begin{align}
\label{eq:Initial Conditions}
\phi(X,0) &= a(X),\nonumber \\
\phi_t(X,0) &= b(X),
\end{align}

\noindent
and assume they are consistent with the boundary conditions \eqref{eq:bndcond}. In order to determine appropriate consistent initial conditions for \eqref{eq:Hamiltonian DAE} and \eqref{eq:E-L equations - augmented} we need to solve several equations. First we solve for the $y_i$'s and $X_i$'s. We have $y_0=\phi_L$, $y_{N+1}=\phi_R$, $X_0=0$, $X_{N+1}=X_{max}$. The rest are determined by solving the system

\begin{align}
\label{eq:Initial Conditions for y and X}
y_i&=a(X_i), \nonumber \\
0&=g_i(y_1,\ldots,y_N,X_1,\ldots,X_N),
\end{align}

\noindent
for $i=1,\ldots,N$. This is a system of $2N$ nonlinear equations for $2N$ unknowns. We solve it using Newton's method. Note, however, that we do not \emph{a priori} know good starting points for Newton's iterations. If our initial guesses are not close enough to the desired solution, the iterations may converge to the wrong solution or may not converge at all. In our computations we used the constraints \eqref{eq:ArclengthConstraint}. We found that a very simple variant of a homotopy continuation method worked very well in our case. Note that for $\alpha=0$ the set of constraints \eqref{eq:ArclengthConstraint} generates a uniform mesh. In order to solve \eqref{eq:Initial Conditions for y and X} for some $\alpha>0$, we split $[0,\alpha]$ into $d$ subintervals by picking $\alpha_k = (k/d)\cdot\alpha$ for $k=1,\ldots,d$. We then solved \eqref{eq:Initial Conditions for y and X} with $\alpha_1$ using the uniformly spaced mesh points $X^{(0)}_i = (i/(N+1))\cdot X_{max}$ as our initial guess, resulting in $X^{(1)}_i$ and $y^{(1)}_i$. Then we solved \eqref{eq:Initial Conditions for y and X} with $\alpha_2$ using $X^{(1)}_i$ and $y^{(1)}_i$ as the initial guesses, resulting in $X^{(2)}_i$ and $y^{(2)}_i$. Continuing in this fashion, we got $X^{(d)}_i$ and $y^{(d)}_i$ as the numerical solution to \eqref{eq:Initial Conditions for y and X} for the original value of $\alpha$. Note that for more complicated initial conditions and constraint functions, predictor-corrector methods should be used---see \cite{NumericalContinuation} for more information. Another approach to solving \eqref{eq:Initial Conditions for y and X} could be based on relaxation methods (see \cite{HuangRussellREVIEW}, \cite{HuangRussellBOOK}).

Next, we solve for the initial values of the velocities $\dot y_i$ and $\dot X_i$. Since $\varphi(x,t)=\phi(X(x,t),t)$, we have $\varphi_t(x,t)=\phi_X(X(x,t),t)X_t(x,t)+\phi_t(X(x,t),t)$. We also require that the velocities be consistent with the constraints. Hence the linear system

\begin{align}
\label{eq:Initial Conditions for velocities}
\dot y_i &= a'(X_i) \dot X_i + b(X_i), \quad\quad\quad\quad i=1,\ldots,N \nonumber \\
0&= \frac{\partial g}{\partial y}(y,X) \dot y + \frac{\partial g}{\partial X}(y,X) \dot X.
\end{align}

\noindent
This is a system of $2N$ linear equations for the $2N$ unknowns $\dot y_i$ and $\dot X_i$, where $y=(y_1,\ldots,y_N)$ and $X=(X_1,\ldots,X_N)$. We can use those velocities to compute the initial values of the conjugate momenta. For the control-theoretic approach we use $p_i=\partial \tilde L_N/\partial \dot y_i$, as in Section~\ref{sec:timeintegration}, and for the Lagrange multiplier approach we use \eqref{eq:MatrixLegendre}. In addition, for the Lagrange multiplier approach we also have the initial values for the slack variables $r_i=0$ and their conjugate momenta $B_i=\partial \tilde L^A_N/\partial \dot r_i=0$. It is also useful to use \eqref{eq:Solving for the Lagrange multipliers} to compute the initial values of the Lagrange multipliers $\lambda_i$ that can be used as initial guesses in the first iteration of the Lobatto IIIA-IIIB algorithm. The initial guesses for the slack Lagrange multipliers are trivially $\mu_i=0$. Note that both $\lambda$ and $\mu$ are algebraic variables, so their values at each time step are completely determined by the Lobatto IIIA-IIIB algorithm (see \cite{HLWGeometric}, \cite{JayLobatto}, \cite{JaySPARK} for details), and therefore no further initial or boundary conditions are necessary.

\subsection{Convergence}
\label{sec: Convergence}

In order to test the convergence of our methods as the number of mesh points $N$ is increased, we considered a single soliton bouncing from two rigid walls at $X=0$ and $X=X_{max}=25$. We imposed the boundary conditions $\phi_L=0$ and $\phi_R=2\pi$, and as initial conditions we used \eqref{eq: SG soliton} with $X_0=12.5$ and $v=0.9$. It is possible to obtain the exact solution to this problem by considering a multi-soliton solution to \eqref{eq:Sine-Gordon Equation} on the whole real line. Such a solution can be obtained using a B\"{a}cklund transformation (see \cite{DrazinSolitonsIntroduction}, \cite{Rajaraman}). However, the formulas quickly become complicated and, technically, one would have to consider an infinite number of solitons. Instead, we constructed a nearly exact solution by approximating the boundary interactions with \eqref{eq: SG two solitons}:

\begin{equation}
\label{eq:SG nearly exact solution}
\phi_{exact}(X,t)=
\left\{
\begin{array}{ll}
\phi_{SS}\big(X-X_{max},t-(4n+1)T\big)+2 \pi & \quad \textrm{if $t \in \big[4nT,(4n+2)T\big)$,} \\
\phi_{SS}\big(X,t-(4n+3)T\big) & \quad \textrm{if $t \in \big[(4n+2)T,(4n+4)T\big)$,}
\end{array}
\right.
\end{equation}

\noindent
where $n$ is an integer number and $T$ satisfies $\phi_{SS}(X_{max}/2,T)=\pi$ (we numerically found $T\approx13.84$). Given how fast \eqref{eq: SG soliton} and \eqref{eq: SG two solitons} approach its asymptotic values, one may check that \eqref{eq:SG nearly exact solution} can be considered exact to machine precision.

We performed numerical integration with the constant time step $\Delta t=0.01$ up to the time $T_{max}=50$. For the control-theoretic strategy we used the 1-stage and 2-stage Gauss method (2-nd and 4-th order respectively), and the 2-stage and 3-stage Lobatto IIIA-IIIB method (also 2-nd/4-th order). For the Lagrange multiplier strategy we used the 2-stage and 3-stage Lobatto IIIA-IIIB method for constrained mechanical systems (2-nd/4-th order). See \cite{HLWGeometric}, \cite{HWODE1}, \cite{HWODE2} for more information about the mentioned symplectic Runge-Kutta methods. We used the constraints \eqref{eq:ArclengthConstraint} based on the generalized arclength density \eqref{eq:ArclengthDensity}. We chose the scaling parameter to be $\alpha=2.5$, so that approximately half of the available mesh points were concentrated in the area of high gradient. A few example solutions are presented in Figure~\ref{fig: Numerical solution Lagrange N=15}-\ref{fig: Numerical solution Control-Theoretic Gauss N=31}. Note that the Lagrange multiplier strategy was able to accurately capture the motion of the soliton with merely 17 mesh points (that is, $N=15$). The trajectories of the mesh points for several simulations are depicted in Figure~\ref{fig: Mesh point trajectories for Lagrange multiplier strategy} and Figure~\ref{fig: Mesh point trajectories for control-theoretic strategy}. An example solution computed on a uniform mesh is depicted in Figure~\ref{fig: Numerical solution on a uniform mesh Gauss N=31}.

For the convergence test, we performed simulations for several $N$ in the range 15-127. For comparison, we also computed solutions on a uniform mesh for $N$ in the range 15-361. The numerical solutions were compared against the solution \eqref{eq:SG nearly exact solution}. The $L^\infty$ errors are depicted in Figure~\ref{fig: Convergence}. The $L^\infty$ norms were evaluated over all nodes and over all time steps. Note that in case of a uniform mesh the spacing between the nodes is $\Delta x = X_{max}/(N+1)$, therefore the errors are plotted versus $(N+1)$. The Lagrange multiplier strategy proved to be more accurate than the control-theoretic strategy. As the number of mesh points is increased, the uniform mesh solution becomes quadratically convergent, as expected, since we used linear finite elements for spatial discretization. The control-theoretic strategy also shows near quadratic convergence, whereas the Lagrange multiplier method seems to converge slightly slower. While there are very few analytical results regarding the convergence of $r$-adaptive methods, it has been observed that the rate of convergence depends on several factors, including the chosen mesh density function. Our results are consistent with the convergence rates reported in \cite{BeckettNUMERICAL} and \cite{Wacher}. Both papers deal with the viscous Burgers' equation, but consider different initial conditions. Computations with the arclength density function converged only linearly in \cite{BeckettNUMERICAL}, but quadratically in \cite{Wacher}.

\begin{figure}[tbp]
	\centering
		\includegraphics[width=\textwidth]{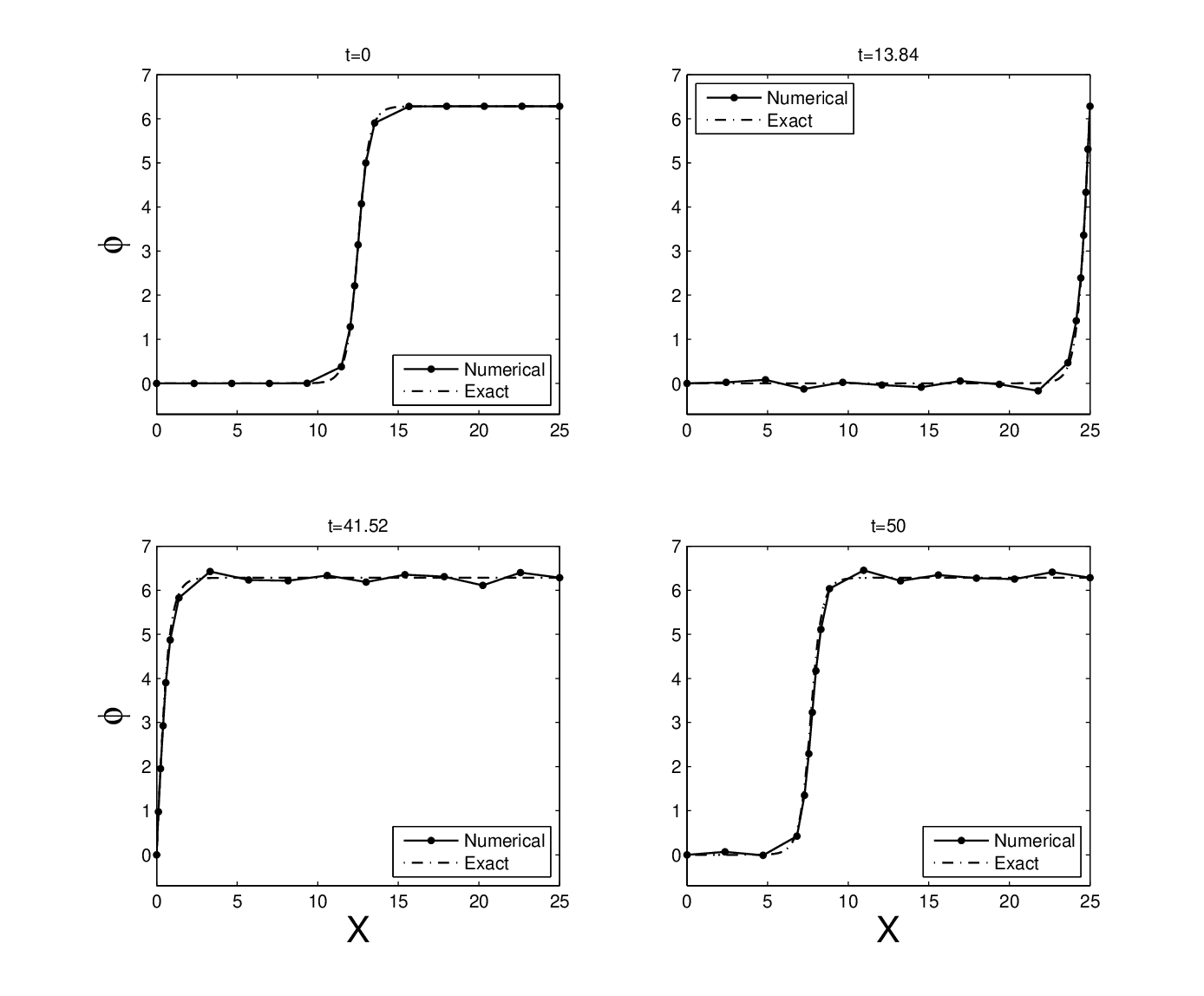}
		\caption{The single soliton solution obtained with the Lagrange multiplier strategy for $N=15$. Integration in time was performed using the 4-th order Lobatto IIIA-IIIB scheme for constrained mechanical systems. The soliton moves to the right with the initial velocity $v=0.9$, bounces from the right wall at $t=13.84$ and starts moving to the left with the velocity $v=-0.9$, towards the left wall, from which it bounces at $t=41.52$.}
		\label{fig: Numerical solution Lagrange N=15}
\end{figure}

\begin{figure}[tbp]
	\centering
		\includegraphics[width=\textwidth]{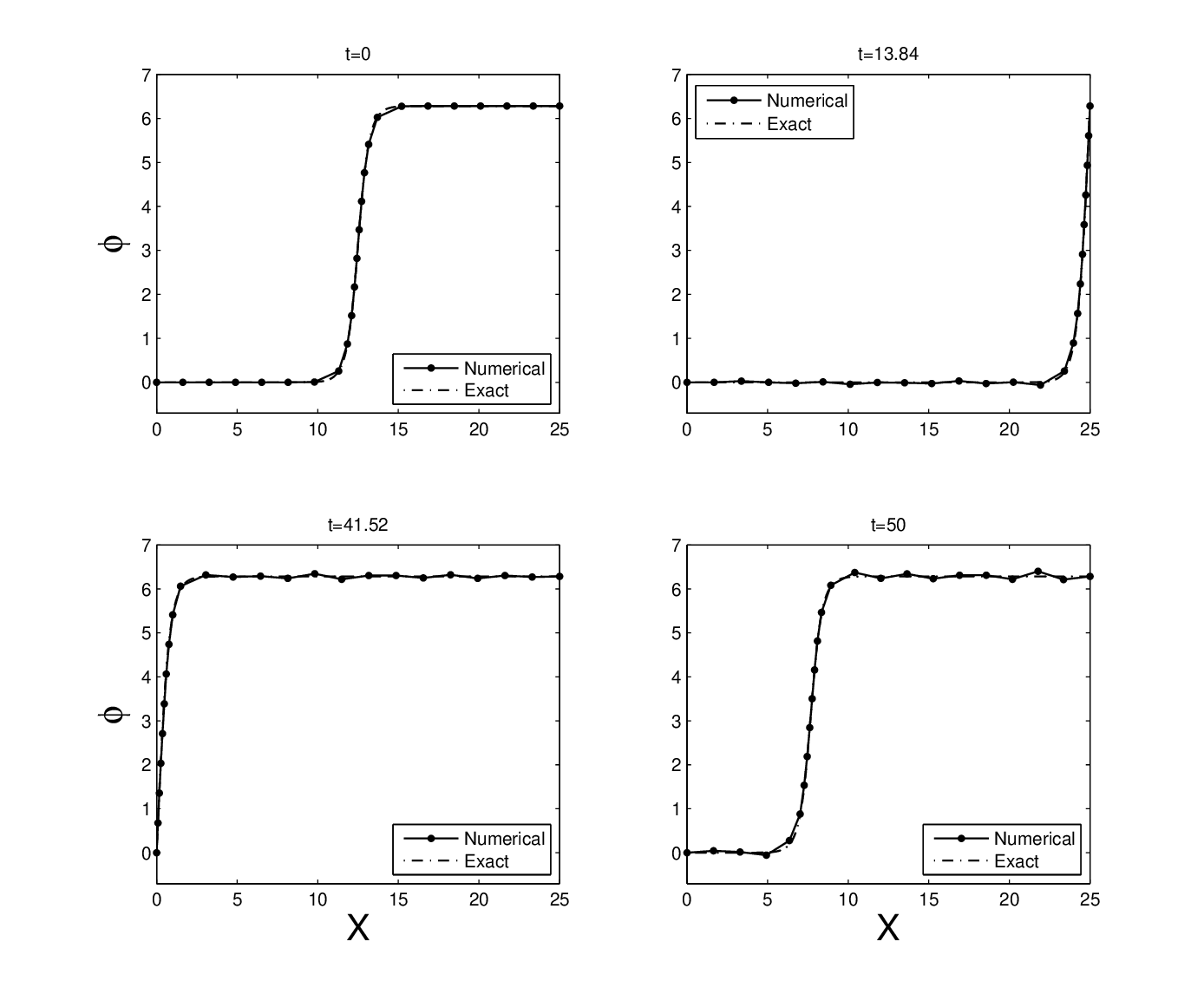}
		\caption{The single soliton solution obtained with the Lagrange multiplier strategy for $N=31$. Integration in time was performed using the 4-th order Lobatto IIIA-IIIB scheme for constrained mechanical systems.}
		\label{fig: Numerical solution Lagrange N=31}
\end{figure}

\begin{figure}[tbp]
	\centering
		\includegraphics[width=\textwidth]{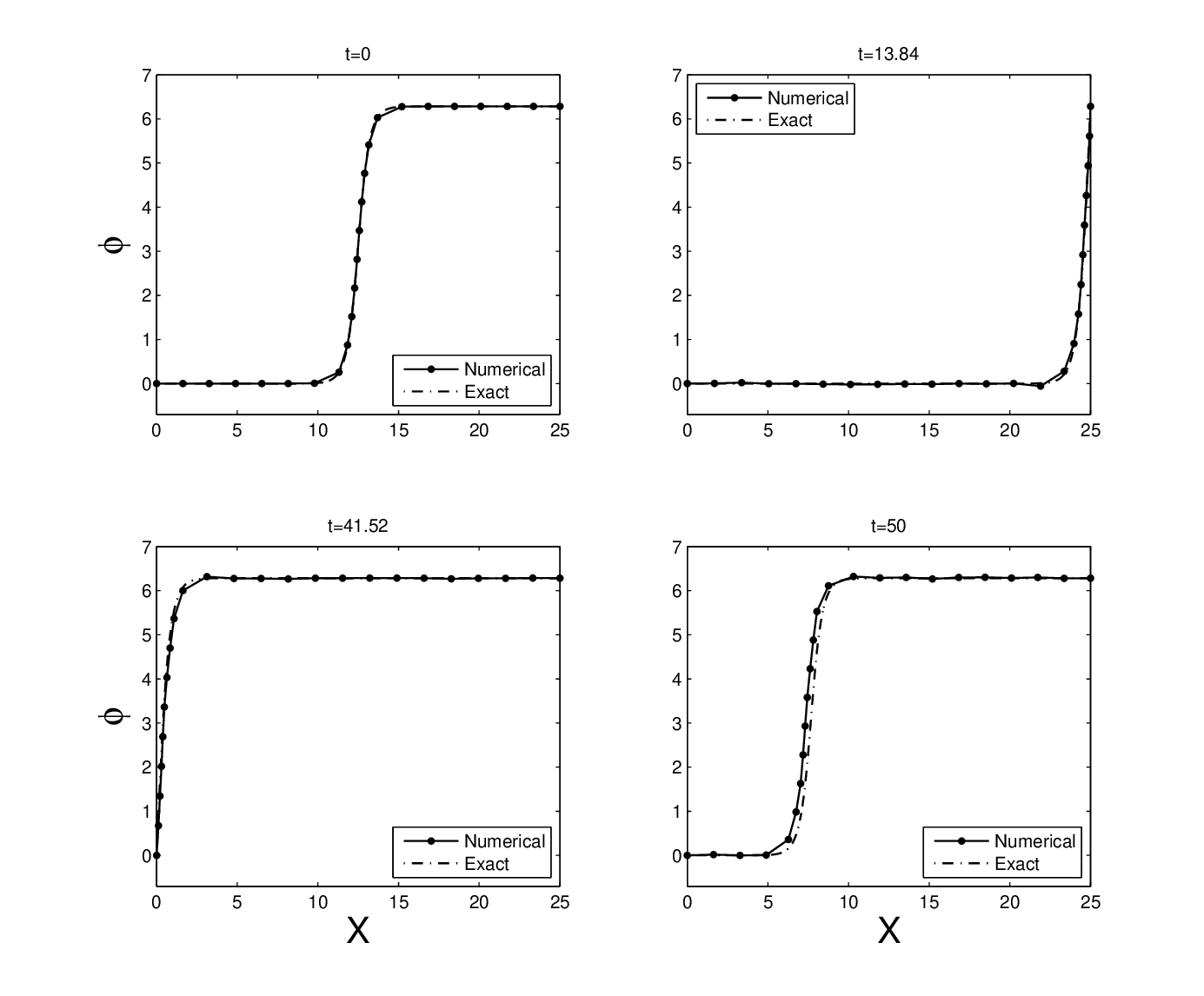}
		\caption{The single soliton solution obtained with the control-theoretic strategy for $N=22$. Integration in time was performed using the 4-th order Gauss scheme. Integration with the 4-th order Lobatto IIIA-IIIB yields a very similar level of accuracy.}
		\label{fig: Numerical solution Control-Theoretic Gauss N=22}
\end{figure}

\begin{figure}[tbp]
	\centering
		\includegraphics[width=\textwidth]{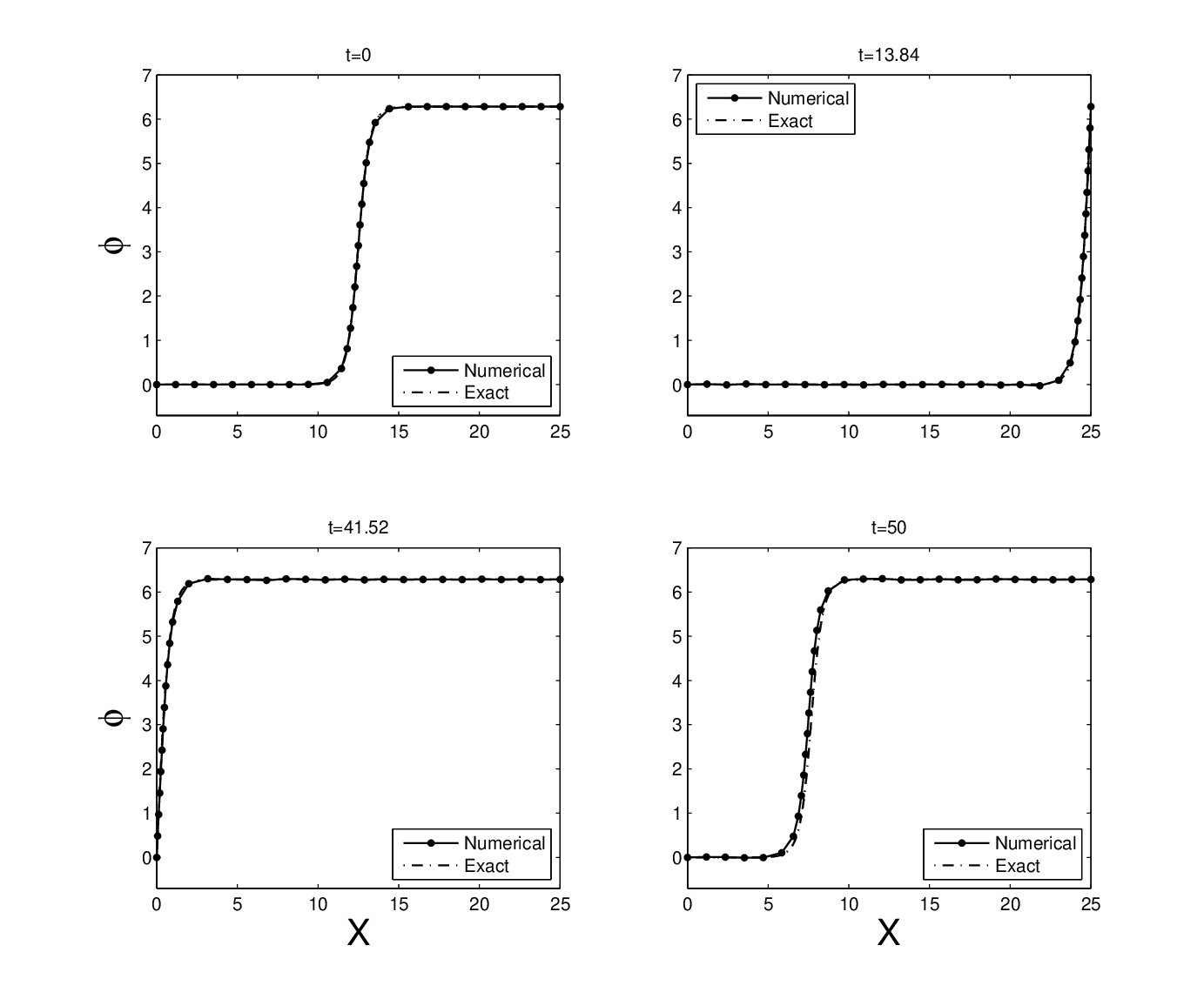}
		\caption{The single soliton solution obtained with the control-theoretic strategy for $N=31$. Integration in time was performed using the 4-th order Gauss scheme. Integration with the 4-th order Lobatto IIIA-IIIB yields a very similar level of accuracy.}
		\label{fig: Numerical solution Control-Theoretic Gauss N=31}
\end{figure}

\begin{figure}[tbp]
	\centering
		\includegraphics[width=\textwidth]{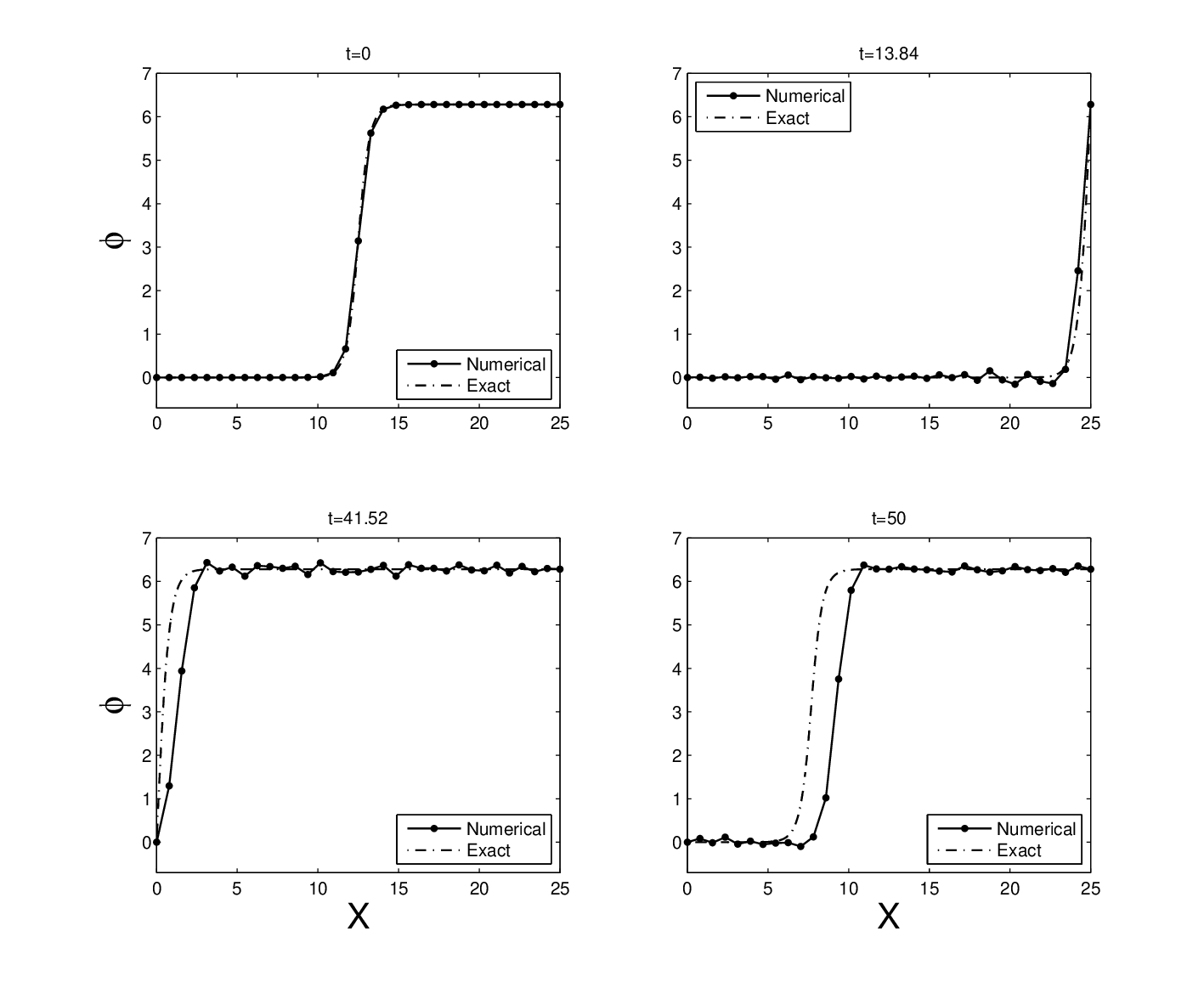}
		\caption{The single soliton solution computed on a uniform mesh with $N=31$. Integration in time was performed using the 4-th order Gauss scheme. Integration with the 4-th order Lobatto IIIA-IIIB yields a very similar level of accuracy.}
		\label{fig: Numerical solution on a uniform mesh Gauss N=31}
\end{figure}

\begin{figure}[tbp]
	\centering
			\begin{tabular}{cc}
				\includegraphics[width=0.38\textwidth]{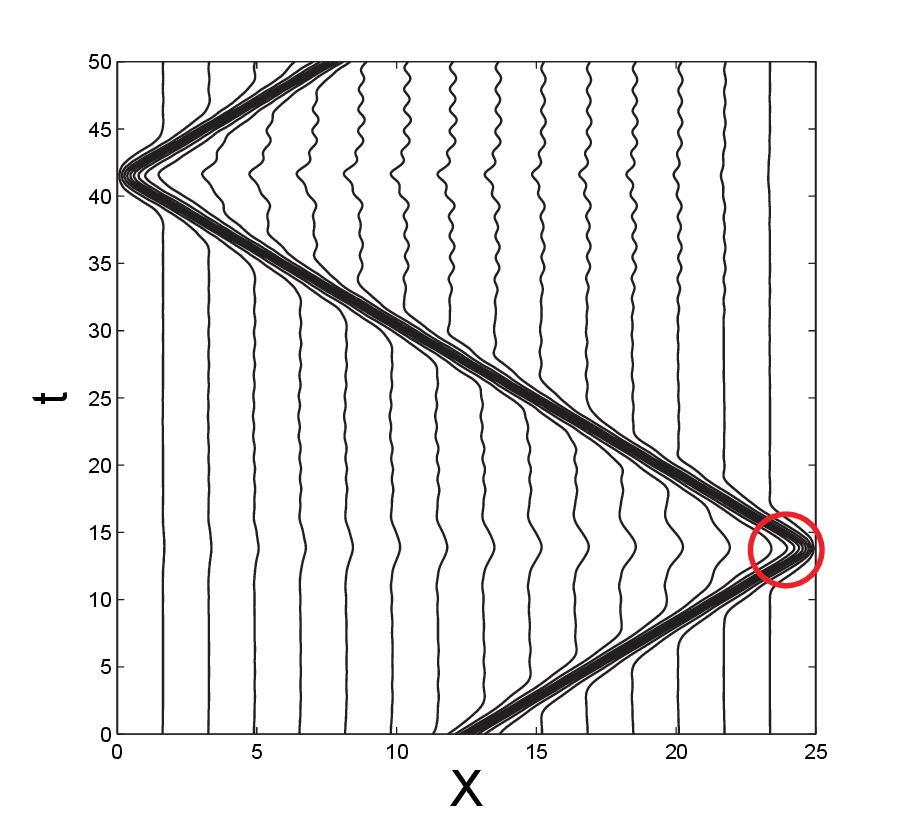} \hspace*{-.8cm} \includegraphics[width=0.12\textwidth]{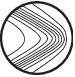}
    		&
				\includegraphics[width=0.38\textwidth]{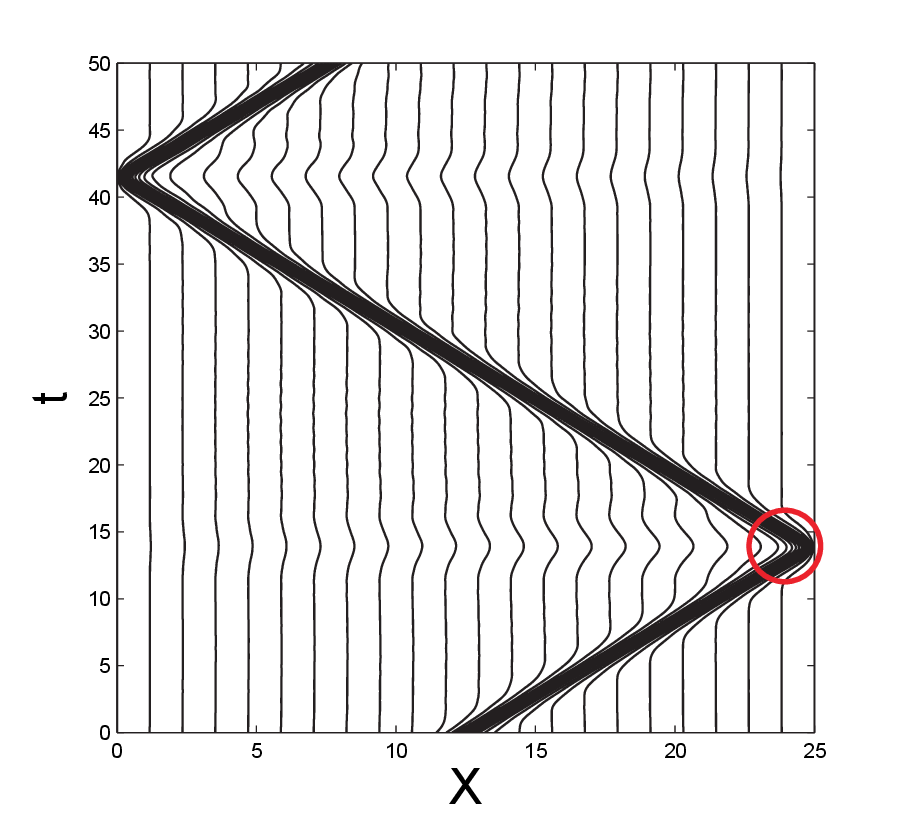} \hspace*{-.8cm} \includegraphics[width=0.12\textwidth]{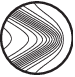}
			\end{tabular}

	\caption{The mesh point trajectories (with zoomed-in insets) for the Lagrange multiplier strategy for $N=22$ (left) and $N=31$ (right). Integration in time was performed using the 4-th order Lobatto IIIA-IIIB scheme for constrained mechanical systems.}
	\label{fig: Mesh point trajectories for Lagrange multiplier strategy}
\end{figure}

\begin{figure}[tbp]
	\centering
			\begin{tabular}{cc}
				\includegraphics[width=0.38\textwidth]{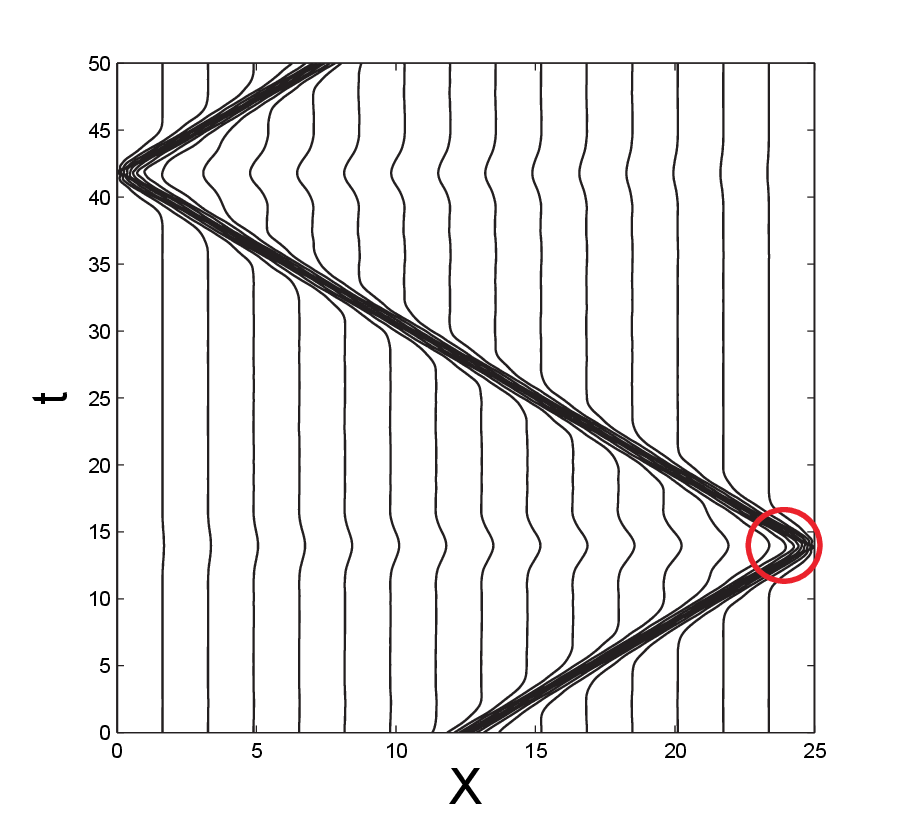} \hspace*{-.8cm} \includegraphics[width=0.12\textwidth]{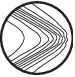}
    		&
				\includegraphics[width=0.38\textwidth]{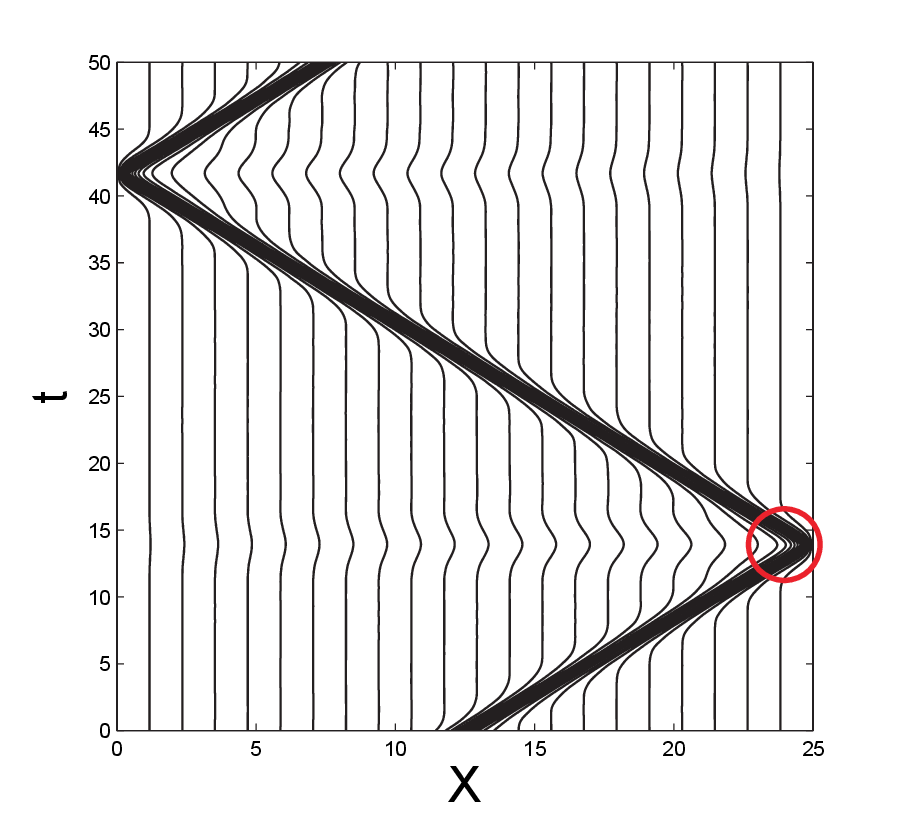} \hspace*{-.8cm} \includegraphics[width=0.12\textwidth]{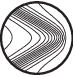}
			\end{tabular}

	\caption{The mesh point trajectories (with zoomed-in insets) for the control-theoretic strategy for $N=22$ (left) and $N=31$ (right). Integration in time was performed using the 4-th order Gauss scheme. Integration with the 4-th order Lobatto IIIA-IIIB yields a very similar result.}
	\label{fig: Mesh point trajectories for control-theoretic strategy}
\end{figure}

\begin{figure}[tbp]
	\centering
		\includegraphics[width=0.8\textwidth]{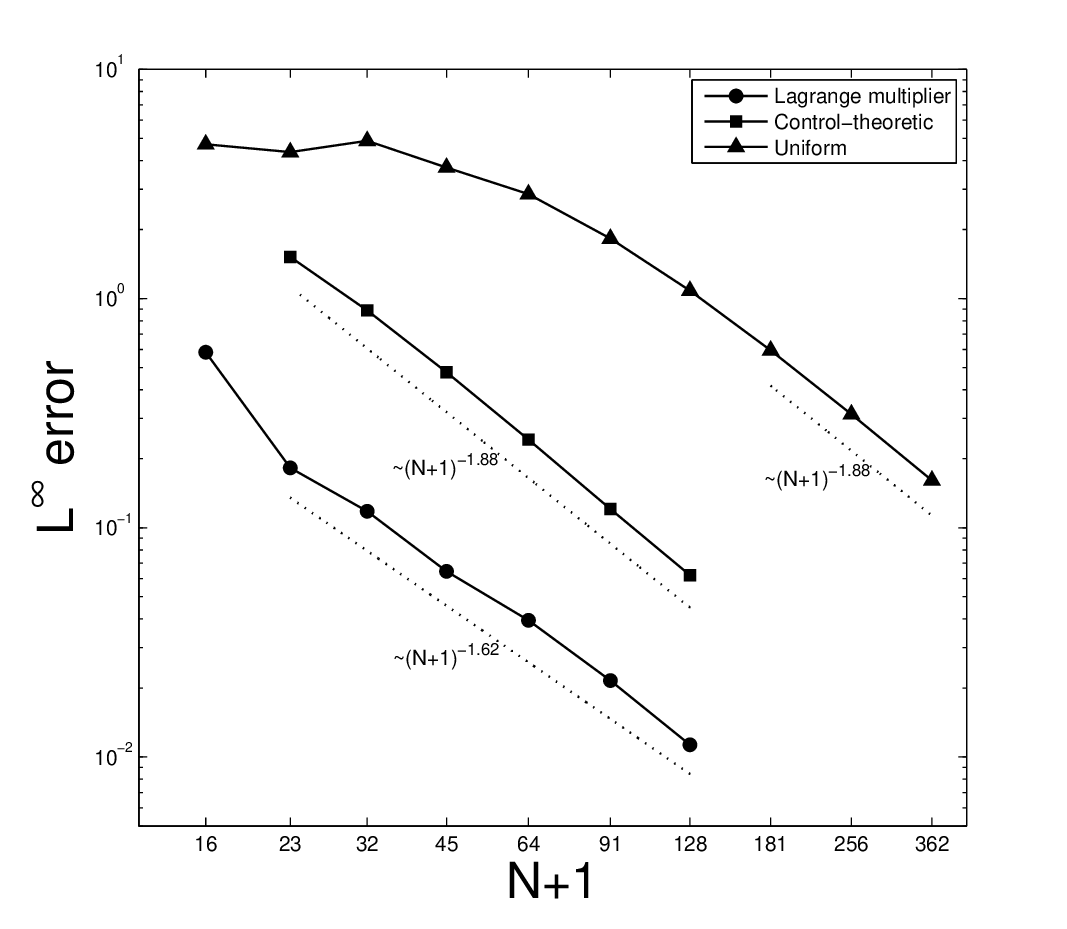}
		\caption{Comparison of the convergence rates of the discussed methods. Integration in time was performed using the 4-th order Lobatto IIIA-IIIB method for constrained systems in case of the Lagrange multiplier strategy, and the 4-th order Gauss scheme in case of both the control-theoretic strategy and the uniform mesh simulation. The 4-th order Lobatto IIIA-IIIB scheme for the control-theoretic strategy and the uniform mesh simulation yields a very similar level of accuracy. Also, using 2-nd order integrators gives very similar error plots.}
		\label{fig: Convergence}
\end{figure}

\subsection{Energy conservation}
\label{sec: Energy Conservation}

As we pointed out in Section~\ref{sec:BEA}, the true power of variational and symplectic integrators for mechanical systems lies in their excellent conservation of energy and other integrals of motion, even when a big time step is used. In order to test the energy behavior of our methods, we performed simulations of the Sine-Gordon equation over longer time intervals. We considered two solitons bouncing from each other and from two rigid walls at $X=0$ and $X_{max}=25$. We imposed the boundary conditions $\phi_L=-2 \pi$ and $\phi_R=2\pi$, and as initial conditions we used $\phi(X,0)=\phi_{SS}(X-12.5,-5)$ with $v=0.9$. We ran our computations on a mesh consisting of 27 nodes (N=25). Integration was performed with the time step $\Delta t=0.05$, which is rather large for this type of simulations. The scaling parameter in \eqref{eq:ArclengthConstraint} was set to $\alpha=1.5$, so that approximately half of the available mesh points were concentrated in the areas of high gradient. An example solution is presented in Figure~\ref{fig: Numerical two kink solution}.

\begin{figure}[tbp]
	\centering
		\includegraphics[width=\textwidth]{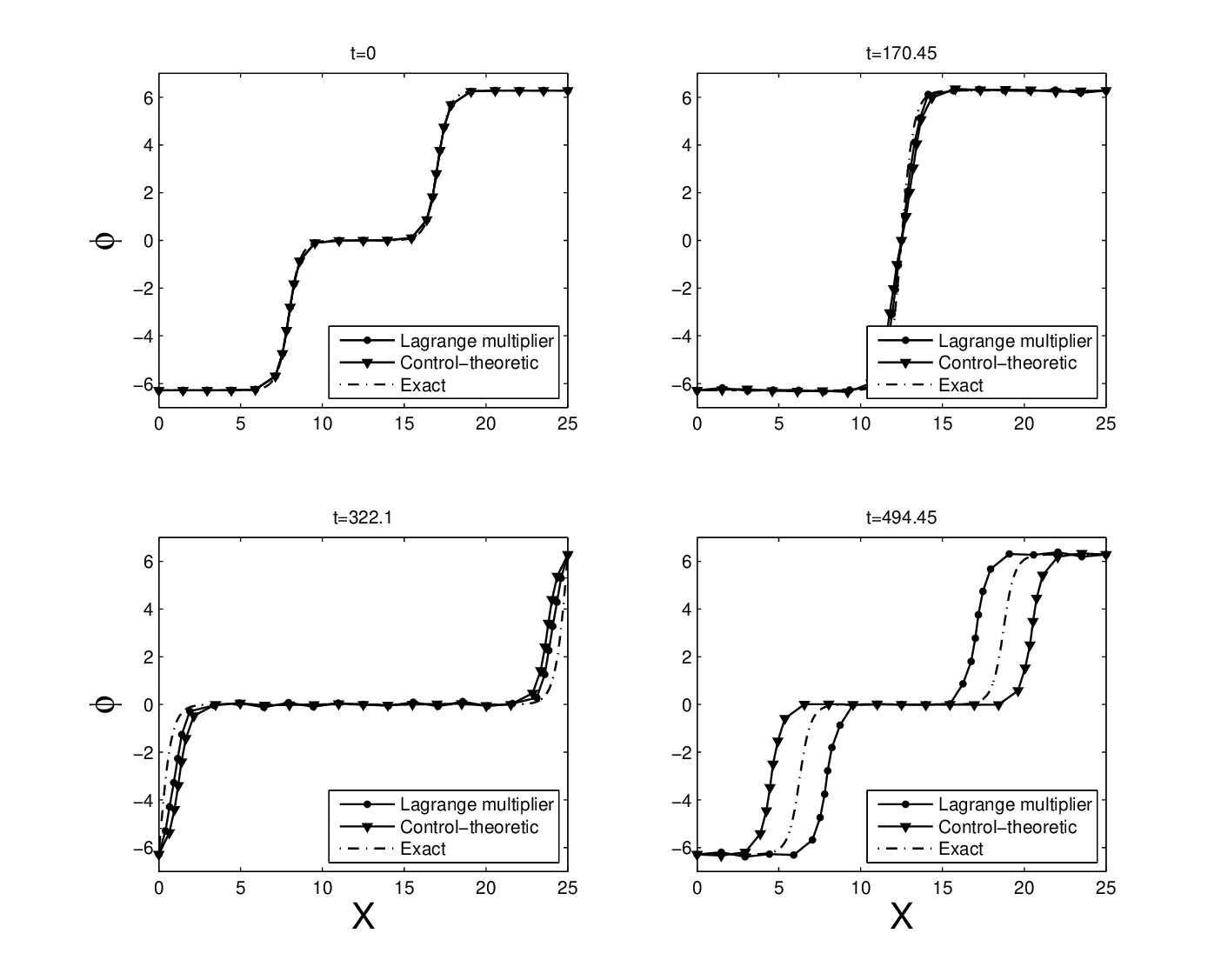}
		\caption{The two-soliton solution obtained with the control-theoretic and Lagrange multiplier strategies for $N=25$. Integration in time was performed using the 4-th order Gauss quadrature for the control-theoretic approach, and the 4-th order Lobatto IIIA-IIIB quadrature for constrained mechanical systems in case of the Lagrange multiplier approach. The solitons initially move towards each other with the velocities $v=0.9$, then bounce off of each other at $t=5$ and start moving towards the walls, from which they bounce at $t=18.79$. The solitons bounce off of each other again at $t=32.57$. This solution is periodic in time with the period $T_{period}=27.57$. The nearly exact solution was constructed in a similar fashion as \eqref{eq:SG nearly exact solution}. As the simulation progresses, the Lagrange multiplier solution gets ahead of the exact solution, whereas the control-theoretic solution lags behind.}
		\label{fig: Numerical two kink solution}
\end{figure}

The exact energy of the two-soliton solution can be computed using \eqref{eq:Energy1}. It is possible to compute that integral explicitly to obtain $E = 16/\sqrt{1-v^2} \approx 36.71$. The energy associated with the semi-discrete Lagrangian \eqref{eq:ParticularLN} can be expressed by the formula

\begin{equation}
\label{eq: Discrete Energy for Sine Gordon}
E_N=\frac{1}{2} \dot q^T \tilde M_N(q) \, \dot q + R_N(q),
\end{equation}

\noindent
where $R_N$ was defined in \eqref{eq:RN} and for our Sine-Gordon system is given by

\begin{equation}
\label{eq: RN for Sine Gordon}
R_N(q)=\sum_{k=0}^N \bigg[ \frac{1}{2} \bigg(\frac{y_{k+1}-y_k}{X_{k+1}-X_k}\bigg)^2 + 1 - \frac{\sin y_{k+1}- \sin y_k}{y_{k+1}-y_k}  \bigg] (X_{k+1}-X_k),
\end{equation}

\noindent
and $M_N$ is the mass matrix \eqref{eq:MassMatrix}. The energy $E_N$ is an approximation to \eqref{eq:Energy1} if the integrand is sampled at the nodes $X_0$,$\ldots$,$X_{N+1}$ and then piecewise linearly approximated. Therefore, we used $E_N$ to compute the energy of our numerical solutions.

The energy plots for the Lagrange multiplier strategy are depicted in Figure~\ref{fig: Energy Plot - Lagrange multiplier}. We can see that the energy stays nearly constant in the presented time interval, showing only mild oscillations, which are reduced as higher order of integration in time is used. The energy plots for the control-theoretic strategy are depicted in Figure~\ref{fig: Energy Plot - control theoretic}. In this case the discrete energy is more erratic and not as nearly preserved. Moreover, the symplectic Gauss and Lobatto methods show virtually the same energy behavior as the non-symplectic Radau IIA method, which is known for its excellent stability properties when applied to stiff differential equations (see \cite{HWODE2}). It seems that we do not gain much by performing symplectic integration in this case. It is consistent with our observations in Section~\ref{sec:BEA} and shows that the control-theoretic strategy does not take the full advantage of the underlying geometry. 

\begin{figure}[tbp]
	\centering
		\includegraphics[width=\textwidth]{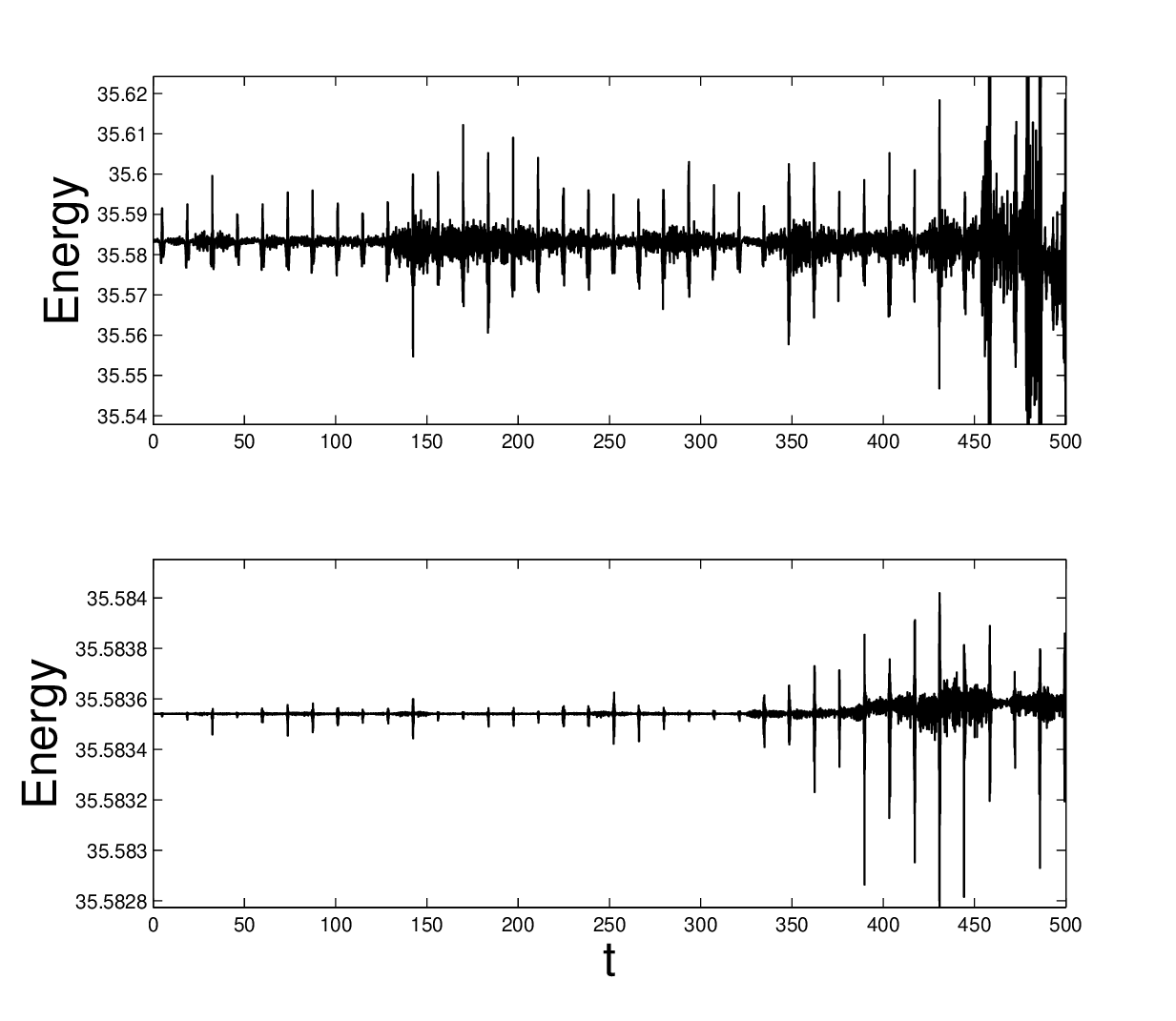}
		\caption{The discrete energy $E_N$ for the Lagrange multiplier strategy. Integration in time was performed with the 2-nd (top) and 4-th (bottom) order Lobatto IIIA-IIIB method for constrained mechanical systems. The spikes correspond to the times when the solitons bounce off of each other or of the walls.}
		\label{fig: Energy Plot - Lagrange multiplier}
\end{figure}

\begin{figure}[tbp]
	\centering
		\includegraphics[width=\textwidth]{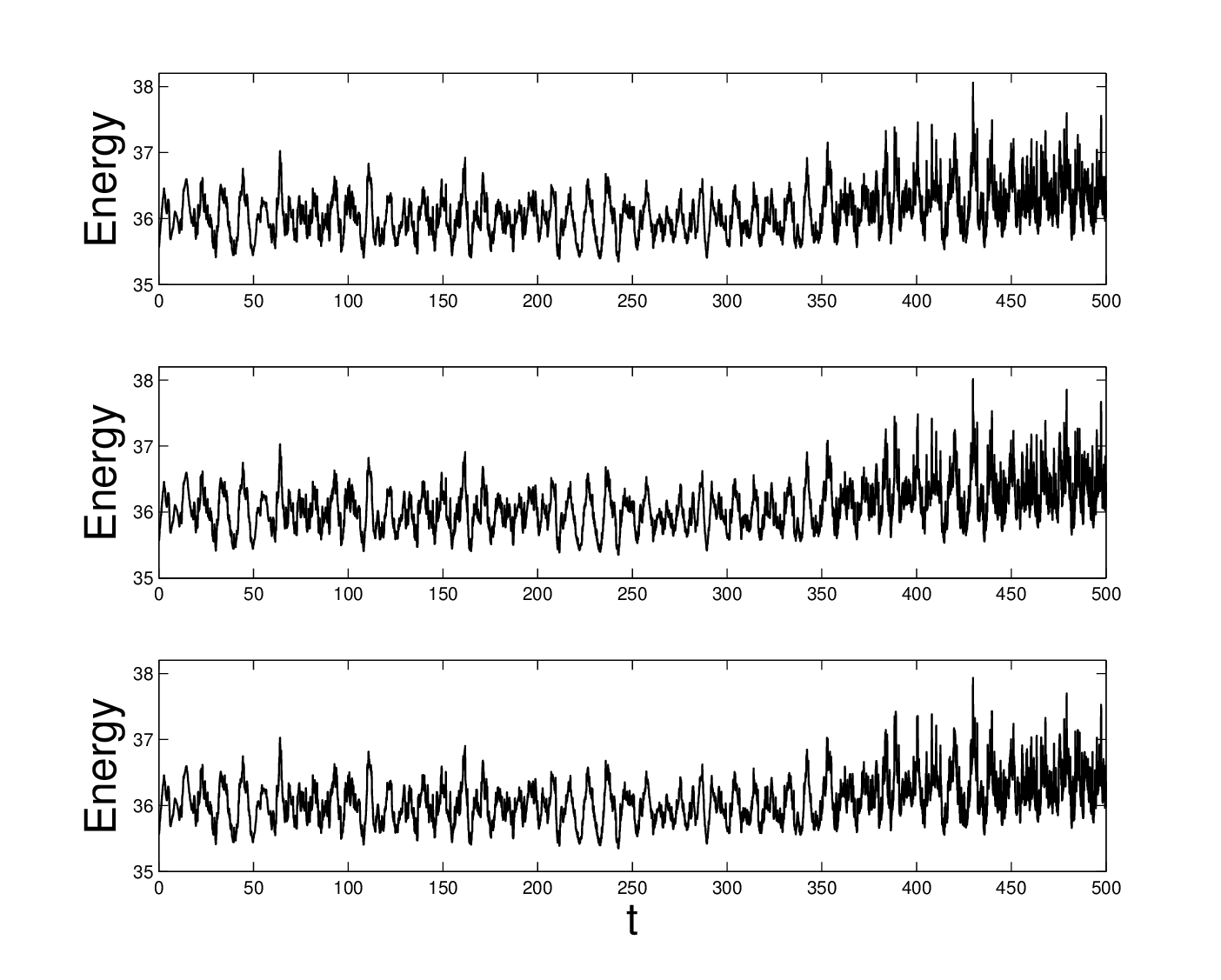}
		\caption{The discrete energy $E_N$ for the control-theoretic strategy. Integration in time was performed with the 4-th order Gauss (top), 4-th order Lobatto IIIA-IIIB (middle) and non-symplectic 5-th order Radau IIA (bottom) methods.}
		\label{fig: Energy Plot - control theoretic}
\end{figure}

As we did not use adaptive time-stepping and did not implement any mesh smoothing techniques, the quality of the mesh deteriorated with time in all the simulations, eventually leading to mesh crossing, i.e. two mesh points collapsing or crossing each other. The control-theoretic strategy, even though less accurate, retained good mesh quality longer, with the break-down time $T_{break}>1000$, as opposed to $T_{break}\sim 600$ in case of the Lagrange multiplier approach (both using a rather large constant time step). We discuss extensions to our approach for increased robustness in Section~\ref{sec:summary}. 

\section{Summary and future work}
\label{sec:summary}

We have proposed two general ideas how $r$-adaptive meshes can be applied in geometric numerical integration of Lagrangian partial differential equations. We have constructed several variational and multisymplectic integrators and discussed their properties. We have used the Sine-Gordon model and its solitonic solutions to test our integrators numerically. 

Our work can be extended in many directions. Interestingly, it also opens many questions in geometric mechanics and multisymplectic field theory. Addressing those questions may have a broad impact on the field of geometric numerical integration.

\subsubsection*{Non-hyperbolic equations}

The special form of the Lagrangian density \eqref{eq:ParticularDensity} we considered leads to a hyperbolic PDE, which poses a challenge to $r$-adaptive methods, as at each time step the mesh is adapted \emph{globally} in response to \emph{local} changes in the solution. \emph{Causality} and the structure of the characteristic lines of hyperbolic systems make $r$-adaptation prone to instabilities and integration in time has to be performed carefully. The literature on $r$-adaptation almost entirely focuses on parabolic problems (see \cite{HuangRussellREVIEW}, \cite{HuangRussellBOOK} and references therein). Therefore, it would be interesting to apply our methods to PDEs that are first-order in time, for instance the Korteweg-de Vries, Nonlinear Schr\"{o}dinger or Camassa-Holm equations. All three equations are first-order in time and are not hyperbolic in nature. Moreover, all can be derived as Lagrangian field theories (see \cite{CamassaHolm}, \cite{CamassaHolmHyman}, \cite{ChenQin}, \cite{DrazinSolitonsIntroduction}, \cite{FaouGeometricSchrodinger}, \cite{GotayKdV}, \cite{KouranbaevaShkoller}). The Nonlinear Schr\"{o}dinger equation has applications to optics and water waves, whereas the Korteweg-de Vries and Camassa-Holm equations were introduced as models for waves in shallow water. All equations possess interesting solitonic solutions. The purpose of $r$-adaptation would be to improve resolution, for instance, to track the motion of solitons by placing more mesh points near their centers and making the mesh less dense in the asymptotically flat areas.

\subsubsection*{Hamiltonian Field Theories}

Variational multisymplectic integrators for field theories have been developed in the Lagrangian setting (\cite{KouranbaevaShkoller}, \cite{MarsdenPatrickShkoller}). However, many interesting field theories are formulated in the Hamiltonian setting. They may not even possess a Lagrangian formulation. It would be interesting to construct Hamiltonian variational integrators for multisymplectic PDEs by generalizing the variational characterization of discrete Hamiltonian mechanics. This would allow to handle Hamiltonian PDEs without the need for converting them to the Lagrangian framework. Recently Leok \& Zhang~\cite{LeokZhang} and Vankerschaver \& Ciao \& Leok~\cite{VankerschaverLeok} have laid foundations for such integrators. It would also be interesting to see if the techniques we used in our work could be applied in order to construct $r$-adaptive Hamiltonian integrators.

\subsubsection*{Time adaptation based on local error estimates}

One of the challenges of $r$-adaptation is that it requires solving differential-algebraic or stiff ordinary differential equations. This is because there are two different time scales present: one defined by the physics of the problem and one following from the strategy we use to adapt the mesh. Stiff ODEs and DAEs are known to require time integration with an adaptive step size control based on local error estimates (see \cite{PetzoldDAE}, \cite{HWODE2}). In our work we used constant time-stepping, as adaptive step size control is difficult to combine with geometric numerical integration. Classical step size control is based on past information only, time symmetry is destroyed and with it the qualitative properties of the method. Hairer \& S\"{o}derlind~\cite{HairerSoderlind} developed explicit, reversible, symmetry-preserving, adaptive step size selection algorithms for geometric integrators, but their method is not based on local error estimation, thus it is not useful for $r$-adaptation. Symmetric error estimators are considered in \cite{JaySPARK} and some promising results are discussed. Hopefully, the ideas presented in those papers could be combined and generalized. The idea of Asynchronous Variational Integrators (see \cite{LewAVI}) could also be useful here, as this would allow to use a different time step for each cell of the mesh.

\subsubsection*{Constrained multisymplectic field theories}

The multisymplectic form formula \eqref{eq:MultisymplecticFormFormula} was first introduced in \cite{MarsdenPatrickShkoller}. The authors, however, consider only unconstrained field theories. In our work we start with the unconstrained field theory \eqref{eq:action}, but upon choosing an adaptation strategy represented by the constraint $G=0$ we obtain a constrained theory, as described in Section~\ref{sec:approach2} and Section~\ref{sec:multisymp approach2}. Moreover, this constraint is essentially nonholonomic, as it contains derivatives of the fields, and the equations of motion are obtained using the \emph{vakonomic} approach (also called variational nonholonomic) rather than the Lagrange-d'Alembert principle. All that gives rise to many very interesting and general questions. Is there a multisymplectic form formula for such theories? Is it derived in a similar fashion? Do variational integrators obtained this way satisfy some discrete multisymplectic form formula? These issues have been touched upon in \cite{MarsdenPekarskyShkollerWest}, but by no means resolved.

\subsubsection*{Mesh smoothing and variational nonholonomic integrators}

The major challenge of $r$-adaptive methods is \emph{mesh crossing}, which occurs when two mesh points collapse or cross each other. In order to avoid mesh crossing and retain good mesh quality, mesh smoothing techniques were developed (\cite{HuangRussellREVIEW}, \cite{HuangRussellBOOK}). They essentially attempt to regularize the exact equidistribution constraint $G=0$ by replacing it with the condition $\epsilon \, \partial X / \partial t = G$, where $\epsilon$ is a small parameter. This can be interpreted as adding some attraction and repulsion pseudoforces between mesh points. If one applies the Lagrange multiplier approach to $r$-adaptation as described in Section~\ref{sec:approach2}, then upon finite element discretization one obtains a finite dimensional Lagrangian system with a nonholonomic constraint. This constraint is enforced using the vakonomic (nonholonomic variational) formulation. Variational integrators for systems with nonholonomic constraints have been developed mostly in the Lagrange-d'Alembert setting, but there have also been some results regarding discrete vakonomic mechanics. The ideas presented in \cite{BenitoDiego2005}, \cite{Garcia}, and \cite{ColomboDiego2013} may be used to design structure-preserving mesh smoothing techniques.

\section*{Acknowledgements}

We would like to extend our gratitude to Michael Holst, Eva Kanso, Patrick Mullen, Tudor Ratiu, Ari Stern and Abigail Wacher for useful comments and suggestions. We are particularly indebted to Joris Vankerschaver and Melvin Leok for support, discussions and interest in this work. We dedicate this paper in memory of Jerrold E. Marsden, who began this project with us.



\end{document}